\documentclass[12pt,reqno]{amsart}
\usepackage[margin=1in]{geometry}
\usepackage{amsmath}
\usepackage{paralist}
\usepackage{physics}
\usepackage{graphics} %% add this and next lines if pictures should be in esp format
\usepackage{epsfig} %For pictures: screened artwork should be set up with an 85 or 100 line screen
\usepackage{graphicx}  \usepackage{epstopdf}
\usepackage[colorlinks=true]{hyperref}
\hypersetup{urlcolor=blue, citecolor=red}
\makeatletter
\def\section{\@startsection{section}{1}%
	\z@{.7\linespacing\@plus\linespacing}{.5\linespacing}%
	{\bfseries\normalfont\scshape
		\centering
}}
\def\@secnumfont{\bfseries}
\makeatother
\usepackage{hyperref}
\usepackage{amsmath,amsfonts,amsthm}
\usepackage{graphicx,amssymb}
\usepackage{amsmath,amssymb,amsfonts}
\usepackage{amscd}
\usepackage{amsthm}
\usepackage{epstopdf}
\usepackage{xcolor}
\usepackage{enumitem}
\numberwithin{equation}{section}
\newtheorem{theorem}{Theorem}[section]
\newtheorem{corollary}{Corollary}

\newtheorem{lemma}[theorem]{Lemma}
\newtheorem{proposition}{Proposition}
\newtheorem{assumption}{Assumption}

\theoremstyle{definition}
\newtheorem{definition}[theorem]{Definition}
\newtheorem{remark}{Remark}

\def\ed{\end{document}}
\begin{document}
\title[Optimal Control of the 3D Damped Navier-Stokes-Voigt Equations  ]
	{Optimal Control of the 3D Damped Navier-Stokes-Voigt \\\vspace{.03in} Equations  with Control Constraints }
	\author[Sakthivel Kumarasamy]{ Sakthivel Kumarasamy}
	\address{Faculty of Mathematics \\
			Indian Institute of Space Science and Technology (IIST) \\
			Trivandrum- 695 547, INDIA}
	\email{sakthivel@iist.ac.in, pktsakthi@gmail.com}
	\curraddr{}
	%\thanks{}
	%\date{}
	%\dedicatory{}
	\maketitle
\begin{abstract}
	In this paper, we consider the 3D  Navier-Stokes-Voigt (NSV) equations with nonlinear damping $|u|^{r-1}u, r\in[1,\infty)$ in bounded and space-periodic domains. We formulate an optimal control problem of minimizing the curl of the velocity field in the energy norm subject to the flow velocity satisfying the damped NSV equation with a distributed control force. The control also needs to obey box-type constraints. For any $r\geq 1,$	the existence and uniqueness of a weak solution is discussed when the domain $\Omega$ is periodic/bounded in $\mathbb R^3$ while a unique strong solution is obtained in the case of space-periodic boundary conditions.  We prove the existence of an optimal pair for the control problem. Using the classical adjoint problem approach, we show that the optimal control satisfies a  first-order necessary optimality condition given by a variational inequality. Since the optimal control problem is non-convex, we obtain a  second-order  sufficient optimality condition showing that an admissible control is locally optimal. Further, we derive optimality conditions  in terms of  adjoint state defined with respect to the growth of the damping term for a global optimal control.
\end{abstract}

%%%%%%%%%%%%%%%%%%%%%%%%%%%%%%%%%%%%%%%%%%%%%%%%%%%%%%
%                   6. BODY
%%%%%%%%%%%%%%%%%%%%%%%%%%%%%%%%%%%%%%%%%%%%%%%%%%%%%%

% Only the first word and proper nouns of section titles should be capitalized.
% The title of section 1:
\section{Introduction}
Optimal control of fluid mechanics has been one of the crucial topics in applied mathematics. One such problem in this topic is the minimization of turbulence in the flow field by acting upon the region by an external force through the interior of the flow field or the boundary of the flow domain. In this work, we study the optimal control problem of  Navier-Stokes-Voigt equations with nonlinear damping  and distributed control on the right-hand side of the state equations, describing the motion of homogeneous incompressible  fluids, given by
{\small\begin{eqnarray}\label{1.1}\textrm{(NSVD)}\left\{\begin{array}{rlll}
			u_t-\mu \Delta u_t -\nu \Delta u
			+(u\cdot\nabla)u+\nabla p+\alpha u+\beta|u|^{r-1}u &=& U  \ \mbox{in} \ \Omega_T,\\ [1mm]
			\nabla\cdot u=0  \ \ \mbox{in} \ \ \Omega_T, \ \ \ u(x,0)&=&u_0 \  \mbox{in} \  \Omega,
		\end{array}\right.
\end{eqnarray}}where $ \Omega_T:=\Omega\times (0,T],$ $\Omega$ is a bounded domain in $\mathbb{R}^3$ with smooth boundary $\partial\Omega$ or a periodic domain as in Section \ref{spd},  and  $T>0$ is a fixed given time. In the case of the bounded domain, \eqref{1.1} is supplemented with the Dirichlet boundary conditions $u=0$ on $\Sigma_T:=\partial\Omega \times [0,T].$ The unknowns $u=(u_1(x,t), u_2(x,t),u_3(x,t))$ denote the velocity of the flow field, $p=p(x,t)$ is a scalar valued function representing the pressure field and $u_0$ is the given initial flow velocity. The function $U=(U_1(x,t), U_2(x,t),U_3(x,t))$ is the applied distributed control force that can be realized, for instance, as the electromagnetic (Lorentz) force distribution in salt water and liquid metals(see, \cite{We},\cite{Fa}).  The parameters, $\mu>0$ denotes the length scale characterizing the elasticity of the fluid,  $\nu>0$ is the kinematic viscosity, the damping coefficients $\alpha,\beta>0,$ and $r\in [1,\infty),$ the order of the nonlinearity. We also assume $\int_\Omega p(t,x)dx =0  \ \mbox{in} \ (0,T)$ for uniqueness of the pressure.

Let us look at some of the special cases of \eqref{1.1} when the external force $U$ (mostly) equals zero. Evidently, when the coefficients $\mu=0$ and $\alpha=\beta=0,$ the model problem \eqref{1.1} reduces to the classical incompressible Navier-Stokes (NS) equations. The existence, uniqueness, and other qualitative properties of the Navier-Stokes equations have been studied to a great extent by several mathematicians(see, \cite{Fl,Te1}).   Though plenty of articles are available for the 3D NS equations,  the uniqueness of weak solutions and the global existence of strong solutions for this problem remains an open problem that attracts many researchers to look into this case deeply. There are various generalizations of the classical NS equations have been proposed in the literature that leads to the global well-posedness and decay of the solutions. More precisely,  when the parameters $\alpha=\beta=0,$ the equation \eqref{1.1} reduces to the standard NSV equation, which was introduced in \cite{Os} as a model for the approximation of the Kelvin-Voigt linear viscoelastic incompressible fluid flow. The NSV equation was suggested in \cite{Ca} as a regularized model for the classical Navier-Stokes equations for obtaining the numerical simulations.
Another important model considered in the literature is the convective  Brinkman-Forchheimer (CBF) equation (\cite{Ha}), which can be obtained by setting the length scale parameter $\mu=0.$   In this special case,  the parameter $\nu$ is referred to as the   Brinkman coefficient (effective viscosity),  $\alpha>0$ denotes the Darcy (permeability of porous medium) coefficient, and $\beta>0$  is the Forchheimer coefficient (proportional to the porosity of the material). In the absence of $\alpha u,$ the CBF equation has the same scaling as the classical NS equation, which is referred to as the NS equation with damping (or absorption) term $\beta|u|^{r-1}u.$  The damping term can be physically motivated as a resistance to the motion of the flow, that is, an external force field in momentum equation accounting for a friction process arising inside a flow domain  (see, \cite{Ant,Cai} and references therein). In the paper, \cite{Ant}, the existence of Leray-Hopf weak solution has been obtained for any dimension $n\geq 2,$ while the uniqueness is obtained when $n=2.$  The existence and uniqueness of weak solutions of the Kelvin-Voigt-Brinkman-Forchheimer (KVBF) equations (see, \cite{An}), which is similar to  the model \eqref{1.1}  but with a general damping term ($f(u)$)  has been studied in a bounded/unbounded domain in $\mathbb R^3.$   For  the same KVBF equation, the existence and uniqueness of strong solutions, when $1\leq r\leq 5,$ and the existence of exponential attractors were obtained in \cite{Mo}.

The flow control problems of classical deterministic and stochastic  NS equations have been well studied over the past few decades. The optimal control of minimizing vorticity of the flow field governed by the 2D NS equations on a bounded domain with distributed control was studied in \cite{Ab}. The existence of optimal boundary control of the NS equation has been the subject of \cite{Sr1}  by showing that the value function, which is the minimum for an objective functional, is the viscosity solution of the associated HJB equation. The time-optimal control of the NS equation was studied in \cite{Ba}. Optimal control of 2D and 3D NS equations in a flow domain exterior to a bounded domain have been rigorously discussed in \cite{Sr11,Fu1}.  Besides, optimal control problems of the 3D NS equations  have been treated in the presence of state constraints, which can arise in the context of suppression of turbulence in a selected flow region \cite{Fa, Wa, Liu}, and also box-type control constraints \cite{Tr1,Ki}.   The existence of optimal control for deterministic and stochastically forced fluid flow models has been studied in \cite{Sr3} by establishing the space-time statistical solutions (see, also \cite{Sa}). Recently, a control problem for the regularized 3D NS equations (NSV equation) in a bounded domain with distributed control and tracking type cost functional has been studied in \cite{An2}, and the time-optimal control of this model has been considered in \cite{An3}. Optimal control of 2D CBF equations is examined for  $r=1,2$ and $3$ in \cite{Mo1}.

Apart from the literature mentioned above on flow control of NS or NSV equations, to the best of the author's knowledge, optimal control of the 3D NSV equations with damping  has not been studied in the literature.  In this paper, we consider an optimal control problem for the NSV equations with damping  $\beta|u|^{r-1}u, r\in[1,\infty)$  (or it is also called KVBF equation) in 3D bounded/periodic domain with distributed control in the state equation subject to control constraints. More precisely, suppose a target velocity field $u_d\in \mathsf L^2(0,T;\mathbb V)$ is given.  We consider the \emph{objective/cost functional}
\begin{eqnarray} \label{OF}
	\mathcal J(u,U):=\frac{\kappa}{2}\int_0^T \|\text{curl}(u(t)-u_d(t))\|^2_{\mathbb{H}} dt + \frac{\lambda}{2}\int_0^T \|U(t)\|^2_{\mathbb{H}} dt	
\end{eqnarray}
and the set of \emph{admissible controls} with constraint
{\small\begin{eqnarray} \label{accx}
		\mathcal U_{ad}:=\left\{ U\in \mathsf L^2(0,T;\mathbb H) \ :\ U_{\min}(x,t)\leq U(x,t) \leq U_{\max}(x,t), \ \mbox{a.e.} \ (x,t)\in \Omega_T \right\}.
\end{eqnarray}}The first term in the cost functional  defines the \emph{enstrophy of the flow field}, that is, it amounts to the kinetic energy of the fluid field, and the second  term specifies energy associated with the control input.  The parameters $\kappa$ and $\lambda$ are fixed nonnegative constants. The control constraints $U_{\min},U_{\max}\in \mathsf L^\infty(\Omega_T)$ are given functions such that $U_{\min}(x,t)\leq U_{\max}(x,t), \  \mbox{a.e.} \ (x,t)\in \Omega_T.$ We intend to find  an  optimal control $U$ minimizing the objective functional $\mathcal J(u,U)$ subject to the control constraint $U\in\mathcal U_{ad}$ and the pair $(u,U)$ is a  solution of the state equation \eqref{1.1}. For various other significant cost functionals, such as the minimization of energy and boundary control problems  associated with the  incompressible fluid dynamics equations, one may refer to \cite{Ab,Sr2,Fu1}.

We describe  the main contributions of this paper. The function spaces used here are defined in Sections \ref{DFS} and \ref{spd}.  For any $r\geq 1,$  $u_0 \in \mathbb{V},$ and the control $U\in \mathsf{L}^2(0,T;\mathbb{H}),$  we formally show that there exists a unique weak solution of \eqref{1.1}.  Further,  when the data  $u_0 \in \mathbb{H}^2\cap \mathbb{V}$ and the domain $\Omega$ is periodic, we also show that this weak solution is a strong solution with $u\in\mathcal{Z},$ where $\mathcal{Z}:=\mathsf H^1(0,T;\mathbb{H}^2) \cap \mathsf{L}^\infty(0,T;\mathbb{L}^{r+1}), \ r\geq 1 .$  The restriction to the space-periodic boundary conditions arises due to some technical issues as given in Section \ref{spd}.

Since $\|\textup{curl} u\|_{\mathbb{H}}=\|\nabla \times u\|_{\mathbb{H}}=\|\nabla u\|_{\mathbb{H}}, u\in \mathbb{V}$ (see, \cite{Do}, Chapter 1), the optimal control problem  we intend to investigate can be stated as follows:
\begin{eqnarray*}
	\textrm{(OCP)}\left\{\begin{array}{lclclc}
		\text{minimize} \ \mathcal{J}(u,U) \\
		\text{subject to the control constraint} \ U\in \mathcal U_{ad} \\
		\text{and } \ u \in \mathcal{Z} \ \text{is a strong solution of  } \ \eqref{1.1} \ \text{in response to} \ U\in \mathcal U_{ad} ,
	\end{array}\right.
\end{eqnarray*}
where the functional
\begin{eqnarray*}\label{cf1}
	\mathcal{J}(u,U)=\frac{\kappa}{2}\int_0^T \|\nabla(u(t)-u_d(t))\|^2_{\mathbb{H}} dt + \frac{\lambda}{2}\int_0^T \|U(t)\|^2_{\mathbb{H}} dt.
\end{eqnarray*}

In the first main result, we have proved the existence of an optimal solution pair $(\widetilde u,\widetilde U)$ solving (OCP). The proof of this result can also be completed within the framework of weak solutions of \eqref{1.1}, which we could get  for the case of the bounded domain (see, Remark \ref{rem4}). The second main theorem of this paper is the first-order necessary optimality conditions satisfied by an optimal pair that is given in terms of a \emph{variational inequality} (Theorem \ref{foc2}) since an optimal control needs to obey the box-type constraints, which is only a subset of $\mathsf L^2(0,T;\mathbb H).$ A simplified variational inequality is obtained by using the classical adjoint problem approach. The required Fr\'echet differentiability of the cost functional $\mathcal J(\cdot,\cdot)$ is proved by introducing a linearized system of \eqref{1.1}.  One of the main issues here is that the rigorous proof of the preceding sequence of results also demands the well-posedness of the linearized system and that of the adjoint system of \eqref{1.1} whose coefficients, in turn, require regularity of solutions of  \eqref{1.1}. It is justified with the help of strong solutions of  \eqref{1.1} obtained in the periodic domain.  Nevertheless, when we restrict the nonlinear damping term $|u|^{r-1}u$ to  the range of $2\leq r\leq 5$ (and $r=1$), we can establish the first-order optimality conditions in the bounded domain by  using the weak solutions of \eqref{1.1}.  This is summarized in Section \ref{FOBD}.

As we are dealing with a non-convex optimal control problem, a second-order sufficient optimality condition for optimal control is also obtained in the third main result for any $r\geq 3.$  This follows from the classical method of showing that the control $\widetilde U\in\mathcal U_{ad}$ satisfying the variational inequality together with the condition that the \emph{reduced cost functional} obeys $\mathfrak J^{\prime\prime}(\widetilde U)[U,U]> 0$ in a cone of critical directions result in a \emph{locally minimizing control}  $\widetilde U$ of  $\mathfrak J(\cdot)$ (Theorem \ref{SOC0}). However, this result doesn't give further information on the global optimality of the control.  We have proved  another  result concerning this question (see, Theorem \ref{SCOCP}). In this theorem, we show that an admissible control satisfying the variational inequality together with a  feasible condition obeyed by the corresponding adjoint state would lead to a  \emph{global optimal control,} and that control also can be obtained uniquely under additional restrictions for any $r\geq 2$ and $r=1.$   Moreover,  we  also proved a global optimality condition  for the case of the bounded domain when  $2\leq r\leq 5$ and $r=1$ (see, Remark \ref{SOCBD}).

Let us further analyze these results closely associated with the existing literature on the bounded domain. For the optimal control problem of 2D CBF equation ($\mu=0$ in \eqref{1.1}) in the bounded domain, a first-order necessary condition was  derived in \cite{Mo1} for $r=2,3.$ The results obtained in this paper for the regularized (OCP) (\eqref{1.1}-\eqref{accx})  generalizes \cite{Mo1} to the 3D bounded domain for any $r\in[2,5].$ The regularized model further helps to prove the crucial Fr\'echet differentiability, Lipschitz continuity of the control-to-state operator, and that of the control-to-costate operator in better function spaces. These properties lead to the second-order Fr\'echet derivative of the cost functional, which is used to get a strict local optimality condition.  It is worth noting that in the global optimality condition result (Theorem \ref{SCOCP}), we stated a special case $r=1$ explicitly, which accounts for the linear perturbation of the NSV model. To the  author's knowledge, such a result is not available in the literature for  optimal control of the  NSV equation.  A local second-order optimality condition  for this  case was obtained in \cite{An2} with tracking type cost functional. Besides, a  global  optimality condition given in terms of the solution of the adjoint problem is also useful in computation (see, \cite{AA,Tr1}).

The manuscript is organized as follows.  Section \ref{Se2} sets up the essential function spaces and collects some standard inequalities that we used throughout the manuscript. The well-posedness of \eqref{1.1}, namely weak and strong solutions, are obtained in Section \ref{Se3}. In the same section, we also prove one of the main theorems concerning the continuous dependence of data and control function. Section  \ref{Se4} discusses the existence of optimal control, and  first-order optimality conditions are given in Section  \ref{Se5}.  Section  \ref{Se6} is devoted to deriving second-order optimality conditions for local optimal control. Finally, global optimality conditions in terms of the adjoint state are given in Section \ref{Se7}.

\section{Function spaces and mathematical inequalities } \label{Se2}  In this section, by taking the divergence-free flow field into account, we introduce the necessary divergence-free function spaces used throughout the article. We state some standard inequalities and derive properties of linear and nonlinear terms that occur in \eqref{1.1}. For more details on these function spaces, one may refer to \cite{Te}.

\subsection{Divergence free function spaces} \label{DFS}  Let $\Omega\subset \mathbb{R}^3$ be a bounded domain with smooth boundary $\partial\Omega.$    For any  $1\leq p<\infty$ or  $p=\infty$ and $m\geq 0,$ let  $\mathsf{W}^{m,p}(\Omega)$  denote the Sobolev spaces of functions in $\mathsf{L}^p(\Omega)$ whose weak derivatives of order less than or equal to $m$ are also in $\mathsf{L}^p(\Omega).$ The norms corresponding to these function spaces are denoted by $\|\cdot\|_{\mathsf L^p}$ and $\|\cdot\|_{\mathsf W^{m,p}}.$ For the special case when $p=2,$ instead of the space $\mathsf{W}^{m,2}(\Omega),$ we shall write $\mathsf{H}^{m}(\Omega)$ with the norm $\|\cdot\|_{\mathsf H^{m}}.$  We also use the time dependent function spaces  $\mathsf{L}^2(0,T;\mathsf{H}^m(\Omega))$ consisting of all measurable functions from $(0,T)$ to $\mathsf{H}^m(\Omega)$ such that square of their $\mathsf{H}^m$-norm is integrable over $(0,T).$ The space $\mathsf H^1(0,T;\mathsf H^{m})$ denotes the space of functions and whose first-order time derivative both belong to $\mathsf{L}^2(0,T;\mathsf{H}^m(\Omega)).$  Since $u,\nabla p$ and $U$ appearing in the governing equations are vector fields, we view them as  belonging to the product spaces  $ (\mathsf {L}^p(\Omega))^3,(\mathsf{H}^m(\Omega))^3,(\mathsf{L}^2(0,T;\mathsf{H}^m(\Omega)))^3,$ etc.

We define the \emph{divergence free Hilbert space}
%$\mathcal{V}(\Omega):=\{v\in (\mathsf{C}^\infty_0(\Omega))^3  :  \nabla\cdot v=0 \ \mbox{in} \  \Omega\},$
$$\mathbb{H}:=\{v\in (\mathsf{L}^2(\Omega))^3 \ : \ \nabla\cdot v=0 \ \mbox{in} \ \ \Omega, \  v\cdot n=0 \ \ \mbox{on} \  \partial\Omega\},$$
with  norm  $\|v\|_{\mathbb{H}}:=\Big(\int_\Omega|v|^2dx\Big)^{1/2},$ where $n$ is the unit outward normal to the boundary $\partial\Omega.$ The $\mathsf{H}^1$ variant of the divergence free Sobolev space is defined as
$$\mathbb{V}:=\{v\in (\mathsf{H}^1(\Omega))^3 \  : \ \nabla\cdot v=0 \ \mbox{in} \  \Omega, \ v=0 \  \mbox{on} \  \partial\Omega\}$$
with  norm  $\|v\|_{\mathbb{V}}:=\Big(\int_\Omega|\nabla v|^2dx\Big)^{1/2}.$
%In fact on the bounded domains, $\mathbb{V}$ is the closure of $\mathcal V$ in $(\mathsf{H}^1_0(\Omega))^3.$
We also use the second-order Sobolev space of functions $\mathbb {H}^2.$     For $p\in (2,\infty),$ we need the divergence free Lebesgue space
$$\mathbb{L}^p:=\{v\in(\mathsf{L}^p(\Omega))^3 \ : \ \nabla\cdot v=0 \ \mbox{in} \  \Omega, \ v\cdot n=0 \  \mbox{on} \  \partial\Omega\}$$ with usual $\mathsf{L}^p$-norm $\|v\|_{\mathbb{L}^p}:=\Big(\int_\Omega|v|^pdx\Big)^{1/p}.$ We shall write the standard Lebesgue space as $\mathbf{L}^p:=(\mathsf{L}^p(\Omega))^3.$
By the zero Dirichlet boundary conditions, the \emph{Poincar\'e inequality} can be employed to show that the semi-norm $\|\cdot\|_{\mathbb{V}}$   and the standard Sobolev norm of the space $\mathsf{H}^1(\Omega)$ are equivalent.  We denote the inner product in the Hilbert space $\mathbb{H}$ by $(\cdot,\cdot).$   We  use $\langle\cdot,\cdot\rangle$ to denote the induced duality between the space $\mathbb{V}$ and it's dual $\mathbb{V}^\prime$ as well as the duality between the space $\mathbb{L}^p$ and it's dual $\mathbb{L}^q,$ where $\frac{1}{p}+\frac{1}{q}=1.$  We write $\|\cdot\|_{\mathbb{V}^\prime}$ for the dual norm in $\mathbb{V}^\prime.$ Since the Hilbert space $\mathbb{H}$ can be identified with it's dual $\mathbb{H}^\prime$ and $\mathbb{V}$ is continuously embedded in $\mathbb{H},$ we have the continuous and dense inclusions, the so-called \emph{ Gelfand triple} such that $\mathbb{V}\subset \mathbb{H} \equiv \mathbb{H}^\prime \subset \mathbb{V}^\prime.$

We also use the space $\mathbb{V}\cap\mathbb{L}^p$  endowed with norm $ \|v\|_{\mathbb{V}}+\|v\|_{\mathbb{L}^p} $ for any $v\in\mathbb{V}\cap\mathbb{L}^p$ and it's dual space $\mathbb{V}^\prime+\mathbb{L}^q$ with the norm
$$ \inf\Big\{\max\{\|v_1\|_{\mathbb{V}^\prime}, \|v_2\|_{\mathbb{L}^q}\}  :  v=v_1+v_2,   v_1\in\mathbb{V}^\prime,  v_2\in\mathbb{L}^q\Big\}.$$ From the definition of $\mathbb{H},$ it is clear that the following continuous embedding holds: $\mathbb{V}\cap\mathbb{L}^p \hookrightarrow \mathbb{H} \hookrightarrow \mathbb{V}^\prime+\mathbb{L}^q.$

\subsection{Some classical inequalities}
We shall use the following inequalities frequently  in the rest of the paper. For the convenience of a reader, we state  those inequalities.
\begin{lemma}[Gagliardo-Nirenberg, \cite{Ni}, Theorem 2.1] \label{GNI}
	Let $\Omega\subset\mathbb R^n$ and $v\in{\mathsf W}_0^{1,p}(\Omega),p\geq 1.$ Then for every fixed numbers $q,r\geq 1,$ there exists a constant $C(\Omega,p,q)>0$ satisfying the inequality
	$\|v\|_{\mathsf{L}^r}\leq C\|v\|^{1-\lambda}_{\mathsf L^q}\|\nabla v\|^\lambda_{\mathsf L^p}, \  \lambda\in[0,1]
	$
	where  $p,q,r$ and $\lambda$ are related by  $\lambda=\Big(\frac{1}{q}-\frac{1}{r}\Big)\Big(\frac{1}{n}-\frac{1}{p}+\frac{1}{q}\Big)^{-1}.$
\end{lemma}
We also recall the  well known inequalities due to Ladyzhenskaya (\cite{La}, Chapter 1, Lemmas 2 and 3) that  can also  be deduced  from Lemma \ref{GNI}.  When $n=3, r=4$ and $p=q=2,$ we get that
\begin{eqnarray} \label{la1}
	\|v\|_{\mathbb L^4} \leq C \|v\|^{1/4}_{\mathbb L^2}\|\nabla v\|^{3/4}_{\mathbb L^2}.
\end{eqnarray}
Further,  $n=3, r=6$ and $p=q=2$ give
$	\|v\|_{\mathbb L^6} \leq C \|\nabla v\|_{\mathbb L^2}, \ \mbox{for all}  \ v\in \mathbb V.
$
\begin{lemma}[Agmon, \cite{Ag}, Lemma 13.2] Let $\Omega\subset\mathbb R^n$ and  $v\in \mathsf H^{m_2}(\Omega).$ Let us choose $m_1$ and $m_2$  such that $m_1<\frac{n}{2}<m_2.$ Then for any $0<\lambda<1,$ there exists a constant $C(\Omega,m_1,m_2)>0$  such that
	$\|v\|_{\mathsf L^\infty} \leq C\|v\|^{\lambda}_{\mathsf H^{m_1}}\|v\|^{1-\lambda}_{\mathsf H^{m_2}},
	$
	where  the numbers $n,\lambda, m_1$ and $m_2$ satisfy the relation
	$\frac{n}{2}=\lambda m_1+(1-\lambda)m_2.$ 	
\end{lemma}
The special case is obtained by taking $n=3,$ $m_1=1$ and $m_2=2$ which give rise to $\lambda=\frac{1}{2}$ satisfying the following  inequality
$\|v\|_{\mathbb L^\infty} \leq C\|v\|^{1/2}_{\mathbb V}\|v\|^{1/2}_{\mathbb H^2}, \  \mbox{for all} \ v\in \mathbb H^2.$

The following interpolation inequality is also useful (see, \cite{Ev}, Appendix B).
\begin{lemma}\label{ipi}Let $p,q,r$ be such that $1\leq p\leq r\leq q\leq \infty$ and  $\frac{1}{r}=\frac{\theta}{p}+\frac{1-\theta}{q}.$ Let $\Omega\subset\mathbb R^n$ and $v\in \mathsf L^p(\Omega)\cap \mathsf L^q(\Omega).$ Then $v\in \mathsf L^r(\Omega)$   and it  holds that
	$	\|v\|_{\mathsf L^r(\Omega)}\leq \|v\|^\theta_{\mathsf L^p(\Omega)}\|v\|^{1-\theta}_{\mathsf L^q(\Omega)}. $	
	
\end{lemma}
\subsection{Properties of linear and nonlinear operators}  \label{plno} Let $\mathbb{P}_p : \mathbf{L}^p\to \mathbb L^p, p\in[2,\infty)$ be the \emph{Helmholtz-Hodge  projection} operator (\cite{Ko}). For simplicity, we denote $\mathbb{P}_p$ by $\mathbb{P}.$ The case $p=2$ corresponds to orthogonal projection.  We define the \emph{Stokes operator }
$\mathsf{A}:\mathbb{V}\cap \mathbb{H}^2\to \mathbb{H}, \ \ \mathsf{A}v:=-\mathbb{P}\Delta v.$
For any $v\in \mathbb{V},$ one can get the relation $\langle \mathsf{A}v,v\rangle=\|v\|_{\mathbb{V}}^2$ and $\|\mathsf{A}v\|_{\mathbb{V}^\prime} \leq \|v\|_{\mathbb{V}} .$  From the Gelfand triple, one may also note that $\mathsf{A}$ maps from $\mathbb{V}$ into $\mathbb{V}^\prime.$

We define a \emph{bilinear operator} $\mathsf{B}:\mathcal{D}_\mathsf{B}\subset \mathbb{H}\times\mathbb{V} \to \mathbb{H}, \   \mathsf{B}(u,v):=\mathbb{P}(u\cdot \nabla)v.$ We shall write $ \mathsf{B}(u)= \mathsf{B}(u,u).$
Let us define a \emph{trilinear form}  $\mathrm{b}(\cdot,\cdot,\cdot): \mathbb{V}\times\mathbb{V}\times\mathbb{V}\to \mathbb{R}$  as follows:
$$\mathrm{b}(u,v,w)=\sum_{i,j=1}^3\int_\Omega u_i\frac{\partial v_j}{\partial x_i}w_j dx = \int_\Omega \big(u(x)\cdot\nabla\big)v(x) \cdot w(x) dx.$$
Integrating by parts with respect to space and using the divergence free condition $\nabla\cdot u=0,$ we obtain the following useful observations
\begin{eqnarray} \label{2.2}
	\mathrm{b}(u,v,w)=-\mathrm{b}(u,w,v), \ \ \mbox{and so} \ \ \ \mathrm{b}(u,v,v)=0, \ \ \forall \ u,v,w\in \mathbb{V}.
\end{eqnarray}
Using H\"older's inequality and the continuous Sobolev  embedding  $\mathbb{V}\hookrightarrow \mathbb{L}^q, 2\leq q\leq 6,$    we have
\begin{eqnarray}\label{2.1}
	|\mathrm{b}(u,v,w)|\leq \|u\|_{\mathbb{L}^4}\|\nabla v\|_{\mathbb{H}} \|w\|_{\mathbb{L}^4}
	\leq C \|u\|_{\mathbb{V}}\|v\|_{\mathbb{V}} \|w\|_{\mathbb{V}}, \ \forall u,v,w\in \mathbb{V}.
\end{eqnarray}
It shows that $\mathrm{b}(\cdot,\cdot,\cdot)$ is a continous trilinear form on $\mathbb{V}\times\mathbb{V}\times\mathbb{V}.$  Consequently, we can identify $\mathsf{B}(\cdot,\cdot)$ as a member of $\mathbb{V}^\prime$ such that $\mathrm{b}(u,v,w)=\langle \mathsf{B}(u,v), w  \rangle$ for all $u,v,w\in \mathbb{V}.$
Indeed, from the estimate \eqref{2.1}, it is also clear that $\mathsf{B}(\cdot)$ is a \emph{bilinear continuous  operator} from $\mathbb{V}$ into $\mathbb{V}^\prime:$
\begin{eqnarray} \label{2.31}
	\| \mathsf{B}(u)\|_{\mathbb{V}^\prime} \leq C \|u\|_{\mathbb{V}}^2, \ \forall u\in\mathbb{V}.
\end{eqnarray}
%For any $u\in\mathbb{V}\cap\mathbb{L}^3,$ by invoking \eqref{2.2}, H\"older's and Sobolev inequalities, we get
%\begin{eqnarray*}
%	|\langle \mathsf{B}(u), v\rangle&=&|\mathrm{b}(u,v,u)|\\
%	&\leq& \|u\|_{\mathbb{L}^3}\|\nabla v\|_{\mathbb{H}} \|u\|_{\mathbb{L}^6}
%	\leq C \|u\|_{\mathbb{L}^3}\|u\|_{\mathbb{V}}  \|v\|_{\mathbb{V}},
%\end{eqnarray*}
%for all $v\in \mathbb{V}\cap\mathbb{L}^3.$ It can also be seen  that the operator $\mathsf{B}(\cdot)$ maps from $\mathbb{V}\cap\mathbb{L}^3 $ into $\mathbb{V}^\prime+\mathbb{L}^\frac{3}{2},$ that is,
%\begin{eqnarray}\label{2.32}
%	\| \mathsf{B}(u)\|_{\mathbb{V}^\prime+\mathbb{L}^\frac{3}{2}} \leq C \|u\|_{\mathbb{L}^3}\|u\|_{\mathbb{V}}, \ \ \forall u\in \mathbb{V}\cap\mathbb{L}^3.
%\end{eqnarray}
By appealing to H\"older's inequality and interpolation inequality (Lemma \ref{ipi})  with $2\leq 2\frac{r+1}{r-1} \leq r+1,$ for any $r\geq 3,$ we further obtain
\begin{eqnarray*}
	|\mathrm{b}(u,u,v)|=|\mathrm{b}(u,v,u)| &\leq& \|u\|_{\mathbb{L}^{r+1}}\|\nabla v\|_{\mathbb{H}} \|u\|_{\mathbb{L}^\frac{2(r+1)}{r-1}} \\
	&\leq&  \|u\|_{\mathbb{L}^{r+1}}^{\frac{r+1}{r-1}} \|u\|_{\mathbb{H}}^{\frac{r-3}{r-1}}\| v\|_{\mathbb{V}},  \ \ \forall u,v\in \mathbb{L}^{r+1}\cap \mathbb{V}.
\end{eqnarray*}
It clearly leads to the estimate
\begin{eqnarray}\label{2.33}
	\| \mathsf{B}(u)\|_{\mathbb{V}^\prime+\mathbb{L}^\frac{r+1}{r}} \leq \|u\|_{\mathbb{L}^{r+1}}^{\frac{r+1}{r-1}} \|u\|_{\mathbb{H}}^{\frac{r-3}{r-1}}, \ \ \forall u\in \mathbb{L}^{r+1}\cap \mathbb{V}.
\end{eqnarray}
Also, note that $\mathbb{L}^{r+1}\cap \mathbb{V}=\mathbb V$ for $r\in[1,5].$  From \eqref{2.31}-\eqref{2.33}, we can see that $\mathsf{B}(\cdot)$ maps from $\mathbb{L}^{r+1}\cap \mathbb{V}$ into $\mathbb{V}^\prime+\mathbb{L}^\frac{r+1}{r}.$ 
If $u$ is regular,  by applying H\"older's inequality, the estimate \eqref{2.1} can be improved  as follows (see, \cite{Te}, Section 3, Lemma 3.8)
\begin{eqnarray}\label{2.r1}
	|\mathrm{b}(u,v,w)|
	\leq C \|u\|_{\mathbb{V}}\|v\|^{1/2}_{\mathbb{V}}\|v\|_{\mathbb{H}^2}^{1/2} \|w\|_{\mathbb{H}}, \ \forall u,v\in \mathbb V \cap \mathbb H^2, \ \ w\in\mathbb H,
\end{eqnarray}
whence
\begin{eqnarray} \label{2.r2}
	\| \mathsf{B}(u)\|_{\mathbb{H}} \leq C \|u\|^{3/2}_{\mathbb{V}}\|u\|_{\mathbb{H}^2}^{1/2} , \ \forall u\in \mathbb V \cap \mathbb H^2.
\end{eqnarray}
Finally, we notice that the \emph{nonlinear damping operator} $\mathsf{D}(u):=\mathbb{P} (|u|^{r-1}u), \ r\geq 1$ maps from $\mathbb{L}^{r+1}$ into $\mathbb{L}^\frac{r+1}{r}.$ By applying H\"older's inequality, we obtain the following estimate:
\begin{eqnarray}\label{2.34}
	|\langle \mathsf{D}(u),v\rangle| = |\langle\mathbb{P} (|u|^{r-1}u), v\rangle | \leq \|u\|_{\mathbb{L}^{r+1}}^r \|v\|_{\mathbb{L}^{r+1}}, \ \ \forall u,v\in \mathbb{L}^{r+1},
\end{eqnarray}
whence $\| \mathsf{D}(u)\|_{\mathbb{L}^\frac{r+1}{r}} \leq \|u\|_{\mathbb{L}^{r+1}}^r$ for all $u\in \mathbb{L}^{r+1}.$  Further, if $u\in \mathbb{H}^2,$ using the embedding $\mathbb H^2\hookrightarrow \mathbb L^\infty,$ we also have
\begin{eqnarray}\label{2.34-1}
	|(\mathsf{D}(u),v)| \leq C\|u\|_{\mathbb L^\infty}^{r} \|v\|_{\mathbb H} \leq C\|u\|_{\mathbb H^2}^{r} \|v\|_{\mathbb H}, \ \ \forall v\in \mathbb H,
\end{eqnarray}
so that, $\| \mathsf{D}(u)\|_{\mathbb H} \leq C \|u\|_{\mathbb{H}^2}^{r}.$
\section{Well-posedness of the control problem} \label{Se3}
In this section, we discuss the existence and uniqueness of weak and strong solutions for the control problem (OCP).  The generic constants $C,C_1,C_2\cdots,$ which  may depend on \emph{system parameters $\mu,\nu,\alpha, \beta,r$} are used without writing the dependence explicitly.  We use the following assumption on the control function.
\begin{assumption}\label{assc}
	$\mathcal U$ is a non-empty open bounded subset of the space $\mathsf L^2(0,T;\mathbb H)$ containing the admissible controls $\mathcal U_{ad}$ and there exists a constant $R>0$ such that $\|U\|_{\mathsf L^2(0,T;\mathbb H)} < R, \ \forall U\in \mathcal U.$
\end{assumption}
\begin{definition}[Weak Solution] \label{ws}
	Let $0<T<+\infty, u_0 \in \mathbb{V}$ and the control $U\in \mathcal U$ be given. A function $u$ is called a weak solution of \eqref{1.1} on the interval $[0,T]$ if the following hold:
	\begin{enumerate}
		\item[(i)] For any $r\geq 1,$ $u\in \mathsf C([0,T];\mathbb{V})\cap \mathsf L^\infty(0,T;\mathbb{V}) \cap \mathsf{L}^{r+1}(0,T;\mathbb{L}^{r+1})$  \vspace{.01in}
		\item[(ii)] $(u_t,v) +\mu (\nabla u_t , \nabla v) +\nu  (\nabla u,\nabla v)
		+\mathrm b(u,u,v)+\alpha (u,v)+\beta\langle |u|^{r-1}u,v\rangle=(U,v)$ \\ for all $v\in \mathbb{V}\cap \mathbb{L}^{r+1}, \ a.e. \ t\in(0,T)$ \vspace{.01in}
		\item [(iii)] $u(0)=u_0$ in $\mathbb{V}.$
	\end{enumerate}
\end{definition}
\begin{remark} \label{rem1a}
	From the estimate \eqref{2.34} it is clear that $\mathsf D(u)\in \mathbb{L}^\frac{r+1}{r},$ and so we need to take the test function $v\in \mathbb{L}^{r+1}$ to make sense of the duality pairing  $\langle |u|^{r-1}u,v\rangle.$ By the embedding $\mathbb V\hookrightarrow  \mathbb L^{r+1}$ for $r \leq 5,$ we have $\mathbb V\cap \mathbb L^{r+1}=\mathbb{V}.$ If we have $1\leq r\leq 5,$ it is enough to have $v\in\mathbb{V}$ in Definition \ref{ws}-(ii).
	Moreover, in the weak form, the pressure term of \eqref{1.1} disappears due to the fact that
	$	\int_\Omega \nabla p \cdot v dx = - \int_\Omega  p \nabla \cdot v dx =0, \  \forall v\in \mathbb{V}.$
\end{remark}
The formulation of the weak and strong solutions can also be written  in terms of the operators discussed in Section \ref{plno}. Applying the projection opeator $\mathbb{P}$ to the equation \eqref{1.1} and recalling that the bilinear operator $\mathsf{B}(\cdot)$ maps   from $\mathbb{V}\cap\mathbb{L}^{r+1} $ into $\mathbb{V}^\prime+\mathbb{L}^\frac{r+1}{r}$ and the damping operator $\mathsf{D}(\cdot)$ maps from $\mathbb{L}^{r+1}$ into $\mathbb{L}^\frac{r+1}{r},$ the weak form  can be understood  as follows:
\begin{eqnarray} \label{sf}\left\{\begin{array}{lcl}
		\frac{d}{dt}\left[u(t)+\mu  \mathsf A u(t)\right]+\nu  \mathsf A u(t)+\mathsf B(u(t))
		+ \alpha u(t) + \beta \mathsf{D}(u(t))\\ \hspace{.31in}= U(t)  \ \mbox{in} \ \mathbb{V}^\prime+\mathbb{L}^\frac{r+1}{r},\ a.e. \ t\in(0,T),\\[1mm]
		u(0)=u_0   \ \mbox{in} \   \mathbb{V},
	\end{array}\right.
\end{eqnarray}
where we have written that $U=\mathbb P U.$ Indeed, let $u\in \mathsf L^\infty(0,T;\mathbb{V}) \cap \mathsf{L}^{r+1}(0,T;\mathbb{L}^{r+1}).$  From \eqref{2.31} and \eqref{2.34}, it is immediate that
\begin{eqnarray} \label{rt}
	\| \mathsf{B}(u)\|^2_{\mathsf L^2(0,T;\mathbb{V}^\prime)} \leq C T \|u\|^4_{\mathsf L^\infty(0,T;\mathbb{V})}, \nonumber\\
	\label{rt1}
	\| \mathsf{D}(u)\|^{\frac{r+1}{r}}_{\mathsf L^{\frac{r+1}{r}}(0,T;\mathbb{L}^\frac{r+1}{r})}  \leq \|u\|^{r+1}_{\mathsf{L}^{r+1}(0,T;\mathsf{L}^{r+1})},
\end{eqnarray}
whence the equality \eqref{sf} holds true in $\mathbb{V}^\prime+ \mathbb{L}^{\frac{r+1}{r}},$ a.e. $t\in (0,T)$ (see, \cite{An}).
Moreover, when $u\in \mathsf L^\infty(0,T;\mathbb{H}^2), u_t\in \mathsf L^2(0,T;\mathbb{H}^2)$  making use of \eqref{2.r2} and \eqref{2.34-1}, we have
\begin{eqnarray} \label{2.r22}
	\| \mathsf{B}(u)\|^2_{\mathsf L^2(0,T;\mathbb{H})}  &\leq& C T \|u\|^{3}_{\mathsf L^\infty(0,T;\mathbb{V})}\|u\|_{\mathsf L^\infty(0,T;\mathbb{H}^2)}, \nonumber\\
	\| \mathsf{D}(u)\|^2_{\mathsf L^2(0,T;\mathbb{H})} &\leq& C T\|u\|^{2r}_{\mathsf L^\infty(0,T;\mathbb{H}^2)}.
\end{eqnarray}
Hence, equation \eqref{sf} holds true in $\mathbb{H},$ a.e. $t\in(0,T),$ and $u(0)=u_0$ in $\mathbb H^2\cap \mathbb V$ which will lead to the strong solution of \eqref{1.1}. All these arguments are justified below.
\begin{theorem}[Existence and Uniqueness of Weak Solution] \label{eu} Let $\Omega\subset \mathbb R^3$ be a bounded domain with smooth boundary $\partial\Omega.$  Let $u_0 \in \mathbb{V}$ and  $U\in \mathcal U$	be arbitrary. Then for any $r\in[1,\infty)$ and $T>0,$ there exists a unique weak solution $u(\cdot)$ in the sense of Definition \ref{ws} for  \eqref{1.1}.
	Further, there exists a constant $C_1>0$ depending on system parameters, $\Omega, C_P,R$ and $\|u_0\|_{\mathbb V}$  such that
	\begin{eqnarray} \label{ee}
		\sup_{t\in(0,T]}\Big(\|u(t)\|_{\mathbb{H}}^2 + \mu\|u(t)\|_{\mathbb{V}}^2\Big)+\nu\int_0^T\|u(t)\|_{\mathbb{V}}^2 dt+\beta \int_0^T\|u(t)\|_{\mathbb{L}^{r+1}}^{r+1} dt \leq C_1.
	\end{eqnarray}
\end{theorem}	
\begin{proof}
	The existence and uniqueness of weak solutions can be proven by using the Faedo-Galerkin approximation  and the arguments similar to \cite{An}. By taking the test function $u=v$ in the weak form  of Definition \ref{ws}-(ii) and using the fact that the trilinear form satisfies $\mathrm b(u,u,u)=0,$ we get the \emph{energy identity,}
	\begin{eqnarray} \label{ei}
		\frac{1}{2}\frac{d}{dt}\left[\|u(t)\|_{\mathbb{H}}^2 + \mu \|u(t)\|_{\mathbb{V}}^2\right] +\nu  \|u(t)\|_{\mathbb{V}}^2+\alpha \|u(t)\|_{\mathbb{H}}^2+\beta \|u(t)\|_{\mathbb{L}^{r+1}}^{r+1} = (U,u).
	\end{eqnarray}
	Since, $|(U,u)| \leq 	\frac{1}{2\alpha} \|U(t)\|_{\mathbb{H}}^2 + \frac{\alpha}{2} \|u(t)\|_{\mathbb{H}}^2,$
	integrating \eqref{ei} over $(0,t)$ for any $t\in (0,T]$  and using the Poincar\'e inequality for the data $\|u_0\|_{\mathbb{H}} \leq C_P\|u_0\|_{\mathbb{V}},$  one can get
	\begin{eqnarray*} \label{eee}
		\|u(t)\|_{\mathbb{H}}^2 + \mu \|u(t)\|_{\mathbb{V}}^2+  \int_0^t \left(\alpha \|u(s)\|_{\mathbb{H}}^2+ \nu\|u(s)\|_{\mathbb{V}}^2\right) ds +\beta \int_0^t\|u(s)\|_{\mathbb{L}^{r+1}}^{r+1} ds \\
		\leq C(C_P)\Big(\|u_0\|_{\mathbb{V}}^2+\|U\|_{\mathsf L^2(0,T;\mathbb{H})}^2 \Big), \ \ \forall t\in(0,T], \ r\in[1,\infty). \nonumber
	\end{eqnarray*}
	This leads to the estimate \eqref{ee} and completes the proof.
\end{proof}	
\subsection{Strong solutions in periodic domain} \label{spd}
In this section, we move on to the proof of the ($\mathbb H^2$) regularity of solutions of \eqref{1.1}. It seems not easy to obtain this result for all $r\in[1,\infty)$ in the case of bounded domains.  More precisely, as usual, when we test \eqref{1.1} with $\Delta u,$ we get one of the terms as $(\nabla p,\Delta u).$ This term may not vanish since by integration by parts,  we get a non-zero boundary term as $\Delta u|_{\partial \Omega}\neq 0.$ It appears that the abstract form \eqref{sf} is also not useful in this case. When we multiply \eqref{sf} by $\mathsf A u,$ we need to handle the integral $(\mathsf D(u),\mathsf A u).$ Since the damping term grows arbitrarily, the projection $\mathbb P $ in $\mathsf D(u)$  is not useful (\cite{Ka}), and $\mathbb P $ also doesn't commute with differential operators, like $\Delta u$ (in bounded domain) (see, \cite{Ha1}). Therefore, we prove the existence and uniqueness of a strong solution of \eqref{1.1} in the periodic domain.

Unless it is specified,  hereafter,  the domain $\Omega$ in \eqref{1.1} is changed as
$\Omega=[0,\mathrm{L}]\times[0,\mathrm{L}]\times[0,\mathrm{L}].$ We use the same notation for the function spaces introduced in Section \ref{DFS} to define on \emph{space-periodic domain $\Omega$} with appropriate modifications.  Let  $\mathbb H, \mathbb V$ and $\mathbb L^p, $ etc.,  be the spaces of functions, respectively, in $ \mathbb H_{loc}(\mathbb R^3), \mathbb V_{loc}(\mathbb R^3)$  and $\mathbb L_{loc}^p(\mathbb R^3),$  which  are $\Omega$-periodic, that is,   $u(x+\mathrm Le_i)=u(x), x\in \mathbb R^3, i=1,2,3,$ and divergence free ($\nabla\cdot u=0$), where $\{e_1,e_2,e_3\}$ is the canonical basis of $\mathbb R^3$ and $\mathrm L$ is the fixed period in all three directions.   We also endow these function spaces with the \emph{vanishing space average condition $\int_\Omega u(x)dx=0$} so that the Poincar\'e inequality holds (see, \cite{Foi}, Chapter II, Section 5). For further details on the function spaces in periodic domains, one may refer to \cite{Foi,Te0}.  We make use of the other notations and inequalities given in Section \ref{DFS} with required modification. It is also evident that the existence and uniqueness of weak solutions of \eqref{1.1} given by Theorem \ref{eu} also holds true for periodic domain.

\begin{definition}[Strong Solution]
	Let $0<T<+\infty,\ u_0 \in \mathbb H^2\cap\mathbb{V}$ and $U\in \mathcal U$ be arbitrary. For any $r\geq 1,$  a function  $u$ is called a strong solution of the system \eqref{1.1}, if it is a weak solution of \eqref{1.1} and
	$u\in \mathsf H^1(0,T;\mathbb{H}^2) \cap \mathsf{L}^\infty(0,T;\mathbb{L}^{r+1}).$
\end{definition}
Since $u\in \mathsf H^1(0,T;\mathbb{H}^2),$ we can conclude that $u\in \mathsf C([0,T];\mathbb{H}^2)$ (see, \cite{Ev}, Section 5.9.2, Theorem 2). Next, we prove a priori estimates required to get the strong solutions.
\begin{proposition} \label{P1} Let $\Omega$ be a periodic domain in $\mathbb R^3.$ Let $u_0 \in \mathbb{H}^2\cap \mathbb{V}$ and  $U\in\mathcal U$ be arbitrary. Then for any smooth solution $(u,p)$ of \eqref{1.1}, there exists a constant $C_2>0$ depending on system parameters, $C_1,\Omega, T,C_P,R,$ and $\|u_0\|_{\mathbb H^2}$  such that
	\begin{eqnarray} \label{ps1}
		\lefteqn{\sup_{t\in(0,T]}\left(\|u(t)\|_{\mathbb{H}^2}^2 +\frac{\beta}{r+1}\|u(t)\|_{\mathbb{L}^{r+1}}^{r+1}\right)
			+ \int_0^T\left(\|u_t(t)\|_{\mathbb{H}}^2+\nu\|\Delta u(t)\|_{\mathbb{H}}^2\right)dt}\nonumber\\
		&&+ \int_0^T\left(\alpha\|u(t)\|_{\mathbb{V}}^2+\mu  \|u_t(t)\|_{\mathbb{V}}^2\right)dt \nonumber\\
		&&+ \beta \int_0^T\int_\Omega|u|^{r-1}|\nabla u|^2 dxdt+ \frac{\beta(r-1)}{2} \int_0^T\int_\Omega|u|^{r-3}\left|\nabla| u|^2\right|^2 dxdt \leq  C_2,  \ \  \ \
	\end{eqnarray}
	for all $r\in[1,\infty).$  Further, $\|u_t\|_{\mathsf L^2(0,T;\mathbb H^2)}^2 \leq C_3,$ where the constant $C_3>0$ depends on $C_2.$
\end{proposition}
\begin{proof}  Multiplying \eqref{1.1} by  $u_t$,  using the \emph{integral identity}  $$\beta \int_\Omega |u|^{r-1} u \cdot u_t dx= \frac{\beta}{r+1}\frac{d}{dt}\Big[\|u(t)\|_{\mathbb{L}^{r+1}}^{r+1}\Big],$$ and applying \eqref{2.1}, we obtain
	\begin{eqnarray}\label{ps2}
		\hspace*{-4cm}&&\lefteqn{\frac{1}{2} \frac{d}{dt}\left[\alpha \|u(t)\|^2_{\mathbb{H}}+\nu \|u(t)\|^2_{\mathbb{V}} +\frac{2\beta}{r+1}\|u(t)\|_{\mathbb{L}^{r+1}}^{r+1}\right] +\|u_t(t)\|^2_{\mathbb{H}}+\mu \|u_t(t)\|^2_{\mathbb{V}}\nonumber}\\
		\hspace*{-4cm}&=&-\mathrm b(u,u,u_t)+(U,u_t)\nonumber\\
		\hspace*{-4cm}&\leq& C \|u(t)\|_{\mathbb{V}}^4+ \frac{\mu}{2}\|u_t(t)\|^2_{\mathbb{V}}  +\frac{1}{2}\|U(t)\|_{\mathbb{H}}^2+\frac{1}{2}\|u_t(t)\|_{\mathbb{H}}^2,
	\end{eqnarray}
	where we  used
	$\int_\Omega \nabla p \cdot u_t dx= -\int_\Omega  p (\nabla\cdot u)_t dx=0,$ since
	$\nabla\cdot u=0$ on $\Omega$ and $u$ is $\Omega$-periodic function.
	
	Multiplying \eqref{1.1} by $-\Delta u,$  integrating over $\Omega$
	and adding with \eqref{ps2}, one can get
	\begin{eqnarray}\label{ps3}
		&&\lefteqn{\frac{1}{2} \frac{d}{dt}\left[\alpha \|u(t)\|^2_{\mathbb{H}}+(\nu+1) \|u(t)\|^2_{\mathbb{V}}+ \mu \|\Delta u(t)\|^2_{\mathbb{H}}+\frac{2\beta}{r+1}\|u(t)\|_{\mathbb{L}^{r+1}}^{r+1}\right]} \nonumber \\
		&&+\frac{1}{2}\|u_t(t)\|^2_{\mathbb{H}}+\frac{\mu}{2} \|u_t(t)\|^2_{\mathbb{V}}+\frac{3\nu}{4}\|\Delta u(t)\|^2_{\mathbb{H}}+ \alpha \|u(t)\|^2_{\mathbb{V}}  -\beta\int_\Omega |u|^{r-1} u \Delta u dx \nonumber\\
		&\leq& (1/2+1/\nu)\|U(t)\|_{\mathbb{H}}^2 + C\|u(t)\|^4_{\mathbb V}+\mathrm b(u,u,\Delta u),
	\end{eqnarray}	
	where we employed  the fact that $\int_\Omega \nabla p \cdot \Delta u dx=- \int_\Omega  p \Delta(\nabla\cdot u) dx=0.$ 	
	Further,  integration by parts leads to the identity (see, \cite{Cai} for whole space)
	\begin{eqnarray}\label{NI}
		-\beta\int_\Omega |u|^{r-1} u \Delta u dx= \beta \int_\Omega |u|^{r-1} |\nabla u|^2 dx
		+\frac{\beta (r-1)}{4}\int_\Omega|u|^{r-3}\left|\nabla| u|^2\right|^2 dx.
	\end{eqnarray}
	Using  H\"older's inequality and the embedding $\mathbb V\hookrightarrow \mathbb L^4,$ we obtain
	\begin{eqnarray}\label{ir}
		|\mathrm b(u,u,\Delta u)| &\leq&  \|u(t)\|_{\mathbb L^4}\|\nabla u(t)\|_{\mathbb L^4}\|\Delta u(t)\|_{\mathbb L^2} \nonumber\\
		& \leq& C\|u(t)\|^2_{\mathbb V}\|u(t)\|_{\mathbb{H}^2}^2+\frac{\nu}{4}\|\Delta u(t)\|_{\mathbb{H}}^2.
	\end{eqnarray}
	Substituting \eqref{NI}-\eqref{ir} into \eqref{ps3} and integrating over $(0,t),$   one can get
	\begin{eqnarray}\label{ps4}
		&&\lefteqn{\; C\|u(t)\|^2_{\mathbb{H}^2}+\frac{2\beta}{r+1}\|u(t)\|_{\mathbb{L}^{r+1}}^{r+1}+\int_0^t\|u_s(s)\|^2_{\mathbb{H}}ds} \nonumber \\
		&&+\mu\int_0^t \|u_s(s)\|^2_{\mathbb{V}}ds+\nu\int_0^t\|\Delta u(s)\|^2_{\mathbb{H}}ds+ 2\alpha \int_0^t\|u(s)\|^2_{\mathbb{V}}ds \nonumber  \\
		&& +2\beta \int_0^t\int_\Omega |u|^{r-1} |\nabla u|^2 dxds
		+\frac{\beta (r-1)}{2}\int_0^t\int_\Omega|u|^{r-3}\left|\nabla|u|^2\right|^2 dxds \nonumber \\
		&\leq& C\left(\|u_0\|^2_{\mathbb H^2}+\|u_0\|^{r+1}_{\mathbb L^{r+1}}+\|U\|_{\mathsf L^2(0,T;\mathbb{H})}^2 \nonumber\right. \nonumber\\
		&&\left.+ \|u\|^4_{\mathsf L^\infty(0,T;\mathbb V)}+\|u\|^2_{\mathsf L^\infty(0,T;\mathbb V)}\int_0^t\|u(s)\|^2_{\mathbb H^2}ds\right),
	\end{eqnarray}	
	for all $t\in (0,T],$ where we also used the fact that  $\|\Delta u\|_{\mathbb{H}}$ is a norm on $\mathbb{H}^2 \cap \mathbb{V},$ which is equivalent to the norm induced by $\mathbb{H}^2$ (see, \cite{Te}, Chapter 3, Lemma 3.7).  Applying Gronwall's inequality by keeping only the first term on the left side, we get
	\begin{eqnarray}\label{H2}
		\sup_{t\in(0,T]}\|u(t)\|^2_{\mathbb{H}^2} & \leq&  C\exp\left( CT\|u\|^2_{\mathsf L^\infty(0,T;\mathbb V)}\right)\left(\|u_0\|^2_{\mathbb H^2}+\|u_0\|^{r+1}_{\mathbb L^{r+1}}\right.\nonumber\\
		&&\left.+\|U\|_{\mathsf L^2(0,T;\mathbb{H})}^2+ \|u\|^4_{\mathsf L^\infty(0,T;\mathbb V)}\right).
	\end{eqnarray}
	Using \eqref{ee} in \eqref{H2} and substituting \eqref{H2} back into \eqref{ps4}, we get the estimate \eqref{ps1}.
	Finally, we show that $u_t\in \mathsf L^2(0,T;\mathbb H^2).$ From the  equation \eqref{1.1}, it is immediate that
	\begin{eqnarray}\label{H22}
		\int_0^T\|\Delta u_t(t)\|_{\mathbb H}^2dt
		&\leq& 6\int_0^T\|U(t)\|_{\mathbb H}^2dt
		+ C \int_0^T\big(\|u_t(t)\|_{\mathbb{ H}}^2+\|\Delta u(t)\|_{\mathbb{ H}}^2 +\|u(t)\|_{\mathbb{H}}^2\big)dt \nonumber  \\
		&&+ C\int_0^T\big(\|(u(t)\cdot \nabla) u(t)\|^2_{\mathbb H} + \||u(t)|^{r-1}u(t)\|_{\mathbb H}^2\big) dt,
	\end{eqnarray}
	where we used  $\int_\Omega \nabla p\cdot \Delta u_t dx=- \int_\Omega  p \Delta(\nabla\cdot u_t) dx=0.$
	The continuous embedding $\mathbb H^2\hookrightarrow \mathbb L^\infty$ gives that
	\begin{eqnarray}\label{H23}
		&&\lefteqn{	\int_0^T\left(\|(u(t)\cdot \nabla) u(t)\|^2_{\mathbb H} + \||u(t)|^{r-1}u(t)\|_{\mathbb H}^2\right) dt\nonumber}\\
		& \leq& \int_0^T\|u(t)\|^2_{\mathbb L^\infty}\|u(t)\|^2_{\mathbb V}dt + C(\Omega) \int_0^T\|u(t)\|^{2r}_{\mathbb L^\infty} dt\nonumber\\
		&	\leq&  C(\Omega,T)\left(\|u\|^2_{\mathsf L^\infty(0,T;\mathbb H^2)} \|u\|^2_{\mathsf L^2(0,T;\mathbb V)} +\|u\|^{2r}_{\mathsf L^\infty(0,T;\mathbb H^2)}\right).
	\end{eqnarray}
	Making use of \eqref{H23} into \eqref{H22} and using \eqref{ps1}, one can finish the proof.
\end{proof}
\begin{theorem}[Existence and Uniqueness of Strong Solution] \label{sst} Let $\Omega$ be a periodic domain in $\mathbb R^3.$ Let $0<T<+\infty, u_0 \in \mathbb{H}^2\cap\mathbb{V}$ and  $U\in \mathcal U$ be arbitrary. Then there exists a unique strong solution to the system \eqref{1.1}. 	
\end{theorem}
\begin{proof}
	The existence of strong solution can be proved through the standard Galerkin approximation. The estimates in Proposition \ref{P1}   is enough to get the relevant weak and strong convergence of the approximated sequence. Nevertheless,  a complete convergence argument has been shown in the context of existence of an optimal control of (OCP) in Theorem \ref{EOC}. By Theorem \ref{eu}, the weak solution is unique and hence if the strong solution exists, then it is unique.
\end{proof}
We prove the continuous dependence of strong solutions  with respect to initial data and control.
\begin{theorem} \label{cdt} Let $\Omega$ be a periodic domain in $\mathbb R^3.$
	Let $u_1,u_2$ be two strong solutions of \eqref{1.1} corresponding to the controls $U_1,U_2 \in \mathcal U,$  initial data $(u_0)_1, (u_0)_2\in \mathbb H^2\cap\mathbb{V}$ and pressure $p_1,p_2$ respectively. Then there exists a constant  $C_3>0$ depending only on system parameters, $\Omega, T,C_P$ and $C_2$ satisfying the estimate 	
	\begin{eqnarray} \label{se}
		&&\sup_{t\in (0,T]}\left(\|u_1(t)-u_2(t)\|_{\mathbb{H}}^2\right)+\|u_1-u_2\|_{\mathsf H^1(0,T;\mathbb{V})}^2 +\|u_1-u_2\|^{r+1}_{\mathsf L^{r+1}(0,T;\mathbb L^{r+1})} \nonumber\\
		&\leq& C_3\left(\|(u_0)_1-(u_0)_2\|_{\mathbb{V}}^2+\|U_1-U_2\|_{\mathsf L^2(0,T;\mathbb{H})}^2  \right), \ r\geq 1.
	\end{eqnarray}
\end{theorem}
\begin{proof}
	Let $u:=u_1-u_2, p:=p_1-p_2, U:=U_1-U_2 $ and $u_0:=(u_0)_1-(u_0)_2.$ Then $(u,p,U)$ solves
	\begin{eqnarray}\label{SDE}
		&&	u_t-\nu\Delta u-\mu  \Delta u_t +\nabla p+\alpha u
		+\beta\big(f(u_1)-f(u_2)\big)\nonumber\\
		&=&U-\big((u_1\cdot \nabla)u_1-(u_2\cdot \nabla)u_2\big),
	\end{eqnarray}
	where $f(u):=|u|^{r-1}u.$
	Testing \eqref{SDE} by $u,$ taking into account that $\mathrm b(u_1,u_1,u)-\mathrm b(u_2,u_2,u)=\mathrm b(u,u_1,u)+\mathrm b(u_2,u,u)$ and using \eqref{2.2}-\eqref{2.1},   we have
	\begin{eqnarray} \label{se1}
		&&\lefteqn{\frac{1}{2}\frac{d}{dt}\left[\|u(t)\|_{\mathbb{H}}^2 + \mu \|u(t)\|_{\mathbb{V}}^2\right] +\nu  \|u(t)\|_{\mathbb{V}}^2+\alpha \|u(t)\|_{\mathbb{H}}^2} \nonumber\\
		&&+\beta \int_\Omega \big(f(u_1)-f(u_2)\big)\cdot u dx = ( U,u)- \mathrm b(u,u_1,u)  \nonumber\\
		&\leq &\frac{1}{2\alpha} \|U(t)\|_{\mathbb{H}}^2+\frac{\alpha}{2} \|u(t)\|_{\mathbb{H}}^2+C\|u_1(t)\|^2_{\mathbb{V}} \|u(t)\|_{\mathbb{V}}^2+\frac{\nu}{2} \|u(t)\|_{\mathbb{V}}^2.
	\end{eqnarray}
	By Lemma 2.1 of \cite{Ha1} (also refer to \cite{Mo1}), there exists a constant $C(r)>0$ such that the following lower bound holds:
	\begin{eqnarray}\label{D1}
		\int_\Omega \big(f(u_1)-f(u_2)\big)\cdot \left(u_1-u_2\right) dx \geq	C(r)\|u_1(t)-u_2(t)\|^{r+1}_{\mathbb L^{r+1}}, \ \mbox{for any} \ r\geq 1.
	\end{eqnarray}
	\iffalse
	By adding and subtracting the terms $|u_1|^{r-1}u_2$ and $|u_2|^{r-1}u_1,$ one may get  the following identity
	Moreover, the integrals $I_2$ and $I_3$ can be rewritten as follows
	\begin{eqnarray}\label{D2}
		I_2+I_3=-\frac{I_1}{2}+\frac{1}{2}\int_\Omega \left(|u_1|^{r-1}-|u_2|^{r-1}\right)\left(|u_1|^2-|u_2|^2\right)dx.
	\end{eqnarray}
	Noting that the second integral of \eqref{D2}  is non-negative and combining it with \eqref{D1},  and applying the inequality $(a+b)^p\leq 2^{p-1}(a^p+b^p),$ for any $a,b\geq 0$ and $1\leq p <\infty,$ we get
	\begin{eqnarray}\label{D3}
		D_0 \geq   \frac{1}{2}	\int_\Omega \left(|u_1|^{r-1}+|u_2|^{r-1}\right)|u_1-u_2|^2dx
		\geq \frac{1}{2^{r-1}}	\int_\Omega \left(|u_1|+|u_2|\right)^{r-1}|u_1-u_2|^2dx\nonumber \\
		\geq \frac{1}{2^{r-1}}	\int_\Omega \left(|u_1-u_2|\right)^{r-1}|u_1-u_2|^2dx =\frac{1}{2^{r-1}}\|u(t)\|^{r+1}_{\mathbb L^{r+1}}.
	\end{eqnarray}
	\fi
	Testing \eqref{SDE} with $u_t$ and again applying trilinear estimate \eqref{2.1}, one may obtain
	\begin{eqnarray} \label{se1x}
		&&\lefteqn{\frac{1}{2}\frac{d}{dt}\left[\alpha \|u(t)\|_{\mathbb{H}}^2 + \nu \|u(t)\|_{\mathbb{V}}^2\right] +\|u_t(t)\|_{\mathbb{H}}^2+\mu  \|u_t(t)\|_{\mathbb{V}}^2}\nonumber\\
		&=& (U,u_t)- \mathrm b(u,u_1,u_t)-\mathrm b(u_2,u,u_t) -\beta (f(u_1)-f(u_2),u_t). \nonumber\\
		&\leq& \|U(t)\|_{\mathbb{H}}^2+\frac{1}{4}\|u_t(t)\|_{\mathbb{H}}^2+\frac{\mu}{2}\|u_t(t)\|_{\mathbb V}^2 \nonumber\\
		&&+  C\left(\|u_1(t)\|^2_{\mathbb V}+\|u_2(t)\|^2_{\mathbb V}\right)\|u(t)\|_{\mathbb V}^2-\beta (f(u_1)-f(u_2),u_t).
	\end{eqnarray}
	By Taylor's formula the damping term is written as follows:
	\begin{eqnarray*}\label{t2}
		\beta \int_\Omega\Big( f(u_1)-f(u_2) \Big)\cdot u_tdx
		%=\beta\int_\Omega \left( \int_0^1\frac{d}{d\tau} f(z) d\tau \right)\cdot u_tdx \\
		= \beta\int_\Omega \int_0^1 \left( f^\prime(z) u \cdot u_t \right) d\tau  dx:=D,
	\end{eqnarray*}
	where $z:=\tau u_1+(1-\tau)u_2$ and the derivative of $f(\cdot)$ is given by
	\begin{eqnarray}\label{ld}
		f^\prime (z)w=\left\{\begin{array}{ll}
			(r-1)|z|^{r-3}(z\cdot w)z+|z|^{r-1}w,  & r\geq 3 \\
			(r-1)\frac{z}{|z|^{3-r}}(z\cdot w)+|z|^{r-1}w, &1<r<3, z\neq 0 \\
			0 &1<r<3, z= 0 \\
			w &r=1.
		\end{array}	\right.
	\end{eqnarray}
	For any $r\geq 3,$ thanks to the embedding $\mathbb H^2\hookrightarrow\mathbb L^\infty$ and the Poincar\'e inequality, we have
	\begin{eqnarray}\label{d3e}
		D	&=&\beta \int_\Omega\int_0^1 \left( (r-1)|z|^{r-3} (z\cdot u)(z\cdot u_t) + |z|^{r-1}(u\cdot u_t)\right) d\tau dx \nonumber\\
		&\leq& \beta r\sup_{\tau\in (0,1)}\int_\Omega\int_0^1 |\tau u_1+(1-\tau)u_2|^{r-1}|u||u_t|d\tau dx \nonumber\\
		&\leq& \beta r2^{r-2}\left(\|u_1(t)\|^{r-1}_{\mathbb L^\infty}+\|u_2(t)\|^{r-1}_{\mathbb L^\infty}\right)\|u(t)\|_{\mathbb H}\|u_t(t)\|_{\mathbb H} \nonumber\\
		&\leq& C(C_P)\left(\|u_1(t)\|^{2(r-1)}_{\mathbb H^2}+\|u_2(t)\|^{2(r-1)}_{\mathbb H^2}\right)\|u(t)\|^2_{\mathbb V}+\frac{1}{4}\|u_t(t)\|^2_{\mathbb H}.	
	\end{eqnarray}	
	For $r=1$ and any $1<r<3, z\neq 0,$ the  estimation similar to \eqref{d3e} holds true as well.
	Substituting \eqref{D1} and \eqref{d3e} respectively into \eqref{se1} and \eqref{se1x}, and adding them together, we get
	\begin{eqnarray}\label{Gie}
		\frac{d}{dt}\Big[X(t)+Y(t)\Big] \leq (2+1/\alpha) \|U(t)\|_{\mathbb{H}}^2 +CZ(t)\left[X(t)+Y(t)\right],
	\end{eqnarray}
	where
	\begin{eqnarray*}
		X(t)&:=&(1+\alpha) \|u(t)\|_{\mathbb{H}}^2 + (\mu+\nu) \|u(t)\|_{\mathbb{V}}^2,\\
		Y(t)&:=&\int_0^t\left(\alpha \|u(s)\|_{\mathbb{H}}^2 + \nu \|u(s)\|_{\mathbb{V}}^2+\|u_s(s)\|_{\mathbb{H}}^2+\mu\|u_s(s)\|_{\mathbb{V}}^2\right.\\
		&&\left.+C(r)\|u(s)\|^{r+1}_{\mathbb L^{r+1}}\right)ds,\\
		Z(t)&:=& \left(\|u_1(t)\|^2_{\mathbb{V}} + \|u_2(t)\|^2_{\mathbb{V}}\right)+\left(\|u_1(t)\|^{2(r-1)}_{\mathbb{H}^2} + \|u_2(t)\|^{2(r-1)}_{\mathbb{H}^2}\right).
	\end{eqnarray*}
	Since by  \eqref{ps1},  $\|Z\|_{\mathsf L^\infty(0,T)} \leq  C(C_2),$	
	applying Gronwall's inequality over $(0,t)$ to \eqref{Gie} leads to the estimate
	\begin{eqnarray} \label{Giex}
		X(t)+Y(t)\leq C\exp\left(C\|Z\|_{\mathsf L^\infty(0,T)}T\right)\left(\|u_0\|^2_{\mathbb H}+\|u_0\|^2_{\mathbb V}+\|U\|_{\mathsf L^2(0,T;\mathbb{H})}^2 \right),
	\end{eqnarray}
	for all $t\in (0,T].$	Further, employing the Poincar\'e inequality $\|u_0\|_{\mathbb H}\leq C_P\|u_0\|_{\mathbb V},$ we complete the proof of \eqref{se}. 	
\end{proof}
\begin{remark}\label{remw} Since for any $1\leq r\leq 5,$ the imbedding  $\mathbb V\hookrightarrow \mathbb{L}^{r+1}$ is continuous, we obtain from \eqref{ps2} and \eqref{ee} that
	\begin{eqnarray}\label{ps2e}
		&&\sup_{t\in(0,T]}\left[\alpha \|u(t)\|^2_{\mathbb{H}}+\nu \|u(t)\|^2_{\mathbb{V}} +\frac{2\beta}{r+1}\|u(t)\|_{\mathbb{L}^{r+1}}^{r+1}\right] +\mu\int_0^T \|u_t(t)\|^2_{\mathbb{V}}dt\nonumber\\
		&\leq &C \left(\|u_0\|^2_{\mathbb{V}}+\|U\|_{\mathsf L^2(0,T;\mathbb{H})}^2+\|u\|_{\mathsf L^\infty(0,T;\mathbb{V})}^4\right) \leq \widetilde C_1,
	\end{eqnarray}
	where the constant $\widetilde C_1>0$ depends on $C_1$ only.  In this case, for any $u_0\in \mathbb V,$ the weak solution $u$ of \eqref{1.1} has the regularity $u\in \mathsf H^1(0,T;\mathbb V).$
	
	If  $3\leq r\leq 5,$ the estimate \eqref{d3e} can be modified by  using the embeddings $\mathbb V\hookrightarrow \mathbb{L}^{\frac{6(r+1)}{11-r}}$ and  $\mathbb V\hookrightarrow \mathbb{L}^{r+1}, \ r\leq 5$  as follows:
	\begin{eqnarray} \label{gte1}
		\beta\Big\langle \int_0^1f^\prime(z) u  d\tau, u_t\Big\rangle\!\!&\leq& \beta\left\|\int_0^1f^\prime\big(\tau u_1+(1-\tau)u_2\big) u  d\tau\right\|_{\mathbb{L}^{6/5}}\|u_t(t)\|_{\mathbb{L}^6} \nonumber \\
		&\leq& C\left(\|u_1(t)\|^{r-1}_{\mathbb L^{r+1}}+\|u_2(t)\|^{r-1}_{\mathbb L^{r+1}}\right)\|u(t)\|_{\mathbb{L}^{\frac{6(r+1)}{11-r}}} \|u_t(t)\|_{\mathbb{L}^6} \nonumber \\
		& \leq& C\left(\|u_1(t)\|^{r-1}_{\mathbb L^{r+1}}+\|u_2(t)\|^{r-1}_{\mathbb L^{r+1}}\right)\|u(t)\|_{\mathbb V}  \|u_t(t)\|_{\mathbb V}  \nonumber\\
		\!\!	& \leq&\!\! C\left(\|u_1(t)\|^{2(r-1)}_{\mathbb V}+\|u_2(t)\|^{2(r-1)}_{\mathbb V}\right)\|u(t)\|^2_{\mathbb V}+\frac{\mu}{4}  \|u_t(t)\|^2_{\mathbb{V}}. \hspace{.25in}
	\end{eqnarray}
	Therefore, estimate \eqref{Giex} can be established by invoking \eqref{ps2e}, that is, by the weak solutions of \eqref{1.1}, since $u_1,u_2\in \mathsf L^\infty(0,T;\mathbb V).$ The estimate \eqref{gte1} can also be   obtained for  $1\leq r <3$ as well.  Hence, Theorem \ref{cdt} holds true in the case of bounded domain when  $1\leq r\leq 5$ and $u_0\in \mathbb V.$
\end{remark}

\section{Optimal control problem with control constraints}  \label{Se4}

In this section, we follow the classical methodologies developed for optimal control problems in \cite{Fu,Li,Sr2} to prove the existence of optimal control for (OCP). A pair $(u,U)$ is called an \emph{admissible pair} if it holds that $(u,U)\in \mathcal{Z}\times \mathcal U_{ad},$
$u$ is a strong solution of \eqref{1.1} corresponding to  $U\in \mathcal U_{ad}$ and $\mathcal J(u,U)<+\infty.$
We denote by $\mathcal A_{ad},$ the class of all admissible pairs for (OCP). Throughout this section, we assume that $u_d\in \mathsf L^2(0,T;\mathbb V)$ and the constraints $U_{\min},U_{\max} \in \mathsf L^\infty(\Omega_T)$ with $U_{\min} \leq U_{\max}, \ \mbox{a.e.} \ (x,t)\in \Omega_T.$ One can prove the following standard result using the arguments, for instance, in \cite{Fu, Sr11}.
\begin{lemma} \label{mc}
	\begin{itemize}
		\item[]
		\item [$\mathrm{(i)}$] The admissible class $\mathcal A_{ad}$ is non-empty.
		%		\item [$\mathrm{(ii)}$] The functional $\mathcal{J}:\mathcal Z\times \mathcal U_{ad}\to \mathbb{R}^+$ is coercive, that is, $\mathcal J(u,U)\to +\infty$ as the norm $\|(u,U)\|_{\mathcal Z\times \mathcal U_{ad}}\to \infty$ for any $(u,U)\in \mathcal Z\times \mathcal U_{ad}.$
		\item [$\mathrm{(ii)}$] The functional $\mathcal{J}:\mathcal Z\times \mathcal U_{ad}\to \mathbb{R}^+$ is weakly sequentially lower-semicontinuous, that is,  if $u_n\overset{w}{\rightharpoonup} u$  in 	$\mathcal{Z}$ and $U_n\overset{w}{\rightharpoonup} U$  in $\mathcal{U}_{ad},$ then we have $\mathcal J(u,U) \leq \liminf_{n\to \infty} \mathcal J(u_n,U_n).$
	\end{itemize}
\end{lemma}
Next, we define the optimal solution to the control problem (OCP) and prove the main theorem.  
\begin{definition}[Optimal Solution] \label{os}
	An admissible pair of solutions $(\widetilde  u,\widetilde  U)\in \mathcal A_{ad}$ is called an optimal pair if it solves  $\textrm{(OCP)}.$ More precisely, the cost functional $\mathcal J(u,U)$ achieves infimum at $(\widetilde  u,\widetilde  U)$, that is, $\mathcal J(\widetilde  u,\widetilde  U)= \inf_{(u,U)\in \mathcal A_{ad}}\mathcal J( u, U).$  The  control $\widetilde  U$ is called an optimal control and the corresponding  solution $\widetilde  u$ is called an optimal state.
\end{definition}
\begin{theorem}[Existence of Optimal Solution]\label{EOC}
	Let $\Omega$ be a periodic domain in $\mathbb R^3.$ Let $u_0\in \mathbb{H}^2\cap \mathbb{V}$ and $T>0$ be given.  Then there exists an optimal pair $(\widetilde  u,\widetilde  U)\in \mathcal{A}_{ad}$ solving (OCP) in the sense of Definition \ref{os}.
\end{theorem}
\begin{proof} In view of Lemma \ref{mc}-(i), the adimissible class of solutions $\mathcal{A}_{ad}$ is non-empty. 
	Since the functional $\mathcal{J}(\cdot,\cdot)$ is bounded from below by zero, there exists a real number $m\geq 0$ such that $m=\displaystyle\inf_{(u,U)\in\mathcal{A}_{ad}} \mathcal{J}(u,U),$ and there exists a minimizing sequence $\{(u_n,U_n)\}\subset\mathcal{A}_{ad}$ such that
	\begin{eqnarray} \label{oc1}
		m=\inf_{(u,U)\in\mathcal{A}_{ad}} \mathcal{J}(u,U) =\lim_{n\to\infty} \mathcal J(u_n,U_n).
	\end{eqnarray}	
	Further, the sequence $\{u_n\}$ is a strong solution of \eqref{1.1} corresponding to the control $\{U_n\}.$
	Evidently, we have
	{\small\begin{eqnarray} \label{sfe}\left\{\begin{array}{lcl}
				\frac{d}{dt}\left[u_n+\mu  \mathsf A u_n\right]+\nu  \mathsf A u_n+\mathsf B(u_n)
				+ \alpha u_n + \beta \mathsf{D}(u_n)-U_n=0 \   \mbox{in} \  \mathbb H, \ a.e. \ t\in(0,T) \\[1mm]
				u_n(0)=u_0 \ \  \ \mbox{in} \  \mathbb H^2\cap \mathbb V.
			\end{array}\right.
	\end{eqnarray}} 
	Since by Assumption \ref{assc},  $\{U_n\}\subset\mathcal{U}_{ad}\subset \mathcal U,$ and by Proposition \ref{P1}, we obtain the uniform bounds on $\{(u_n,U_n)\}:$
	\begin{eqnarray*}\label{ub1}
		\|U_n\|_{\mathsf L^2(0,T;\mathbb H)} \leq C(R)  \ \ \text{and} \ \ \ \ 	\|u_n\|_{\mathcal{Z}} \leq C(\Omega,T,\|u_0\|_{\mathbb{H}^2},R),
	\end{eqnarray*}
	where $C>0$ is independent of $n.$ It leads to the fact that
	\begin{eqnarray} \label{ub}
		\left\{
		\begin{array}{llll} \{u_n\} \ \ \text{is uniformly bounded in} \ \ \mathsf L^\infty(0,T;\mathbb{H}^2) \cap \mathsf{L}^\infty(0,T;\mathbb{L}^{r+1}),\\
			\{\frac{d u_n}{dt}\} \ \ \text{is uniformly bounded in} \ \ \mathsf L^2(0,T;\mathbb{H}^2) .
		\end{array}\right.
	\end{eqnarray}	
	Further, in view of  \eqref{ub} and \eqref{2.r22}, $\mathsf D(u_n)$  is uniformly bounded in $\mathsf{L}^2(0,T;\mathbb H).$
	By Banach-Alaoglu theorem, there exists a subsequence, still denoted as $\{u_n\},$ such that
	\begin{eqnarray*}\left\{\begin{array}{cccll} \label{wl4}
			u_n &\overset{w^\ast}{\rightharpoonup}& \widetilde  u  &\text{in} &  \mathsf L^\infty(0,T;\mathbb{H}^2), \\
			u_n &\overset{w}{\rightharpoonup}& \widetilde  u  &\text{in} &  \mathsf L^2(0,T;\mathbb{H}^2), \\
			\frac{d u_n}{dt}&\overset{w}{\rightharpoonup}&   \frac{d \widetilde u}{dt}  &\text{in} &  \mathsf L^2(0,T;\mathbb{H}^2), \medskip\\
			U_n &\overset{w}{\rightharpoonup}& \widetilde  U
			& \text{in} &   \mathsf L^2(0,T;\mathbb{H}),\\
			\mathsf D(u_n) &\overset{w}{\rightharpoonup}& \xi
			& \text{in}&  \mathsf{L}^2(0,T;\mathbb H), \ \ \mbox{as} \ \ n\to \infty.
		\end{array}\right.
	\end{eqnarray*}
	The terms $\mathsf A u_n,   \frac{d \mathsf A u_n}{dt}$ converges weakly, respectively, to $\mathsf A \widetilde  u, \frac{d\mathsf A\widetilde  u}{dt}$   in $\mathsf L^2(0,T;\mathbb{H}).$
	
	Since $\{u_n\}$ is bounded on $\mathsf H^1(0,T;\mathbb H^2),$ by Aubin-Lion's compactness theorem (see, \cite{Si}, Section 6),  we have $u_n \overset{s}{\rightarrow} \widetilde  u$ in   $\mathsf C([0,T];\mathbb{V}).$
	Further,  there exists  a sub-sequence, still denoted as $\{u_n\}$ such that the convergence $u_n\to  \widetilde  u  \text{ a.e. in} \  \Omega\times (0,T)$ holds.  Since, $\mathsf D(u_n)$ is uniformly bounded in $\mathsf{L}^2(0,T;\mathbb H),$ a consequence of the Lebesgue dominated convergence theorem gives
	\begin{eqnarray*}\label{wl1}
		\mathsf D(u_n) \overset{w}{\rightharpoonup} \mathsf D( \widetilde u)
		\ \ \text{in} \ \ \mathsf{L}^2(0,T;\mathbb H) \ \ \mbox{as} \ \ n\to \infty
	\end{eqnarray*}
	and, by uniqueness of the limit, we have $\xi=\mathsf D( \widetilde u).$ Finally, we prove the weak convergence of $\mathsf B(u_n).$ Let us consider a test function  $v\in\mathsf L^2(0,T;\mathbb H).$ By using the bilinearity of $\mathsf B(\cdot)$ and invoking  the estimate \eqref{2.r1}, we have
	{\small\begin{eqnarray*}\label{wl2s}
			&&\lefteqn{\int_0^T \big( \mathsf B(u_n)- \mathsf B(\widetilde  u), v \big) dt}\nonumber\\
			&=& \int_0^T \big(  \mathsf B(u_n,u_n-\widetilde  u)+ \mathsf B(u_n-\widetilde  u,\widetilde  u), v \big) dt\nonumber\\
			&=& \int_0^T \big( \mathrm b(u_n,u_n-\widetilde  u,v)+\mathrm b(u_n-\widetilde  u,\widetilde u,v)\big) dt\nonumber\\
			&\leq& C\int_0^T\left(\|u_n\|_{\mathbb V}\|u_n-\widetilde  u\|_{\mathbb V}^{1/2} \|u_n-\widetilde  u\|_{\mathbb H^2}^{1/2}\|v\|_{\mathbb H}+\|u_n-\widetilde  u\|_{\mathbb V}\|\widetilde u\|^{1/2}_{\mathbb V} \|\widetilde u\|^{1/2}_{\mathbb H^2}\|v\|_{\mathbb H}\right)dt \nonumber\\
			&\leq& C\sqrt T\left(\|u_n\|_{\mathsf L^\infty(0,T;\mathbb V)}\|u_n-\widetilde u\|^{1/2}_{\mathsf L^\infty(0,T;\mathbb H^2)}+\|\widetilde  u\|^{1/2}_{\mathsf L^\infty(0,T;\mathbb V)} \|u_n-\widetilde u\|^{1/2}_{\mathsf L^\infty(0,T;\mathbb V)}\nonumber\right.\\
			&&\left.\quad\quad\times \|\widetilde  u\|^{1/2}_{\mathsf L^\infty(0,T;\mathbb H^2)} \right) \|u_n-\widetilde u\|^{1/2}_{\mathsf L^\infty(0,T;\mathbb V)}\|v\|_{\mathsf L^2(0,T;\mathbb H)}
			\to 0, \ \ \mbox{as} \ \ n\to\infty,
	\end{eqnarray*}}since $\|u_n-\widetilde u\|_{\mathsf L^\infty(0,T;\mathbb H^2)}$ is bounded due to the embedding $\mathsf H^1(0,T; \mathbb H^2)\hookrightarrow \mathsf L^\infty(0,T;\break \mathbb H^2).$ It holds that
	$\mathsf B(u_n)\overset{w}{\rightharpoonup}\mathsf B(\widetilde u)$ in $\mathsf L^2(0,T;\mathbb H).$
	By passing the weak limits for the respective terms in \eqref{sfe}, we conclude that  \eqref{sf} holds true in $\mathbb H,  a.e. \ t\in (0,T).$

	The facts that $\widetilde  u,\frac{d \widetilde u}{dt}\in \mathsf L^2(0,T; \mathbb H^2)$ would again imply that $\widetilde u\in \mathsf H^1(0,T; \mathbb H^2),$ and hence $\widetilde  u\in \mathsf C([0,T];\mathbb{H}^2).$ Consequently, one can verify   the initial condition  $\widetilde  u(0)=u_0$ in $\mathbb H^2\cap \mathbb V.$ Thus, the function  $\widetilde u\in \mathsf H^1(0,T;\mathbb{H}^2)\cap \mathsf L^\infty(0,T;\mathbb L^{r+1})$ is a unique strong solution of \eqref{1.1}	with control $	\widetilde  U\in \mathcal U_{ad},$ since $\mathcal U_{ad}$ is a closed and convex set in $\mathsf L^2(0,T;\mathbb H).$
	
	Finally, we show that  $(\widetilde  u,\widetilde  U)$ is an optimal pair solving (OCP).  Appealing to the lower-semicontinuity of Lemma \ref{mc}-(ii), we  have
	\begin{eqnarray} \label{oc3}
		\mathcal J (\widetilde  u,\widetilde  U) \leq \liminf_{n\to\infty} \mathcal J(u_n,U_n),
	\end{eqnarray}
	whence $(\widetilde  u,\widetilde  U)\in \mathcal A_{ad}.$  Since $m$ is the infimum of $\mathcal J(\cdot,\cdot)$ and $(\widetilde  u,\widetilde  U)$ is any admissible pair,  we get $m\leq \mathcal J(\widetilde  u,\widetilde  U).$ But $\{(u_n,U_n)\}$ is a minimizing sequence, we thus obtain  from \eqref{oc1} and \eqref{oc3} that
	$$m=\inf_{(u,U)\in\mathcal{A}_{ad}} \mathcal{J}(u,U) \leq \mathcal J(\widetilde  u,\widetilde  U) \leq \liminf_{n\to\infty} \mathcal J(u_n,U_n) = \lim_{n\to\infty} \mathcal J(u_n,U_n) =m.$$
	Consequently, we have  $\mathcal J(\widetilde  u,\widetilde  U)= \displaystyle \inf_{(u,U)\in\mathcal{A}_{ad}} \mathcal{J}(u,U),$ and hence $(\widetilde  u,\widetilde  U)\in\mathcal A_{ad}$ is an optimal solution. This completes the proof.
\end{proof}
\begin{remark} \label{rem4} The existence of solution for (OCP) can be proven by using the weak solution of \eqref{1.1}.  In particular, for any  $1\leq r\leq 5,$ Theorem \ref{EOC}  holds true for bounded domain with Dirichlet boundary conditions.    For any $u_0\in\mathbb V,$ in view of Remark \ref{remw}, we consider the solution space $\widetilde{\mathcal Z}:= \mathsf H^1(0,T;\mathbb{V})  .$  Let $\widetilde {\mathcal A}_{ad}$ denote the class of all admissible pairs  $(u,U)\in \widetilde{\mathcal Z}\times \mathcal U_{ad},$ where	$u$ is the  weak solution of \eqref{1.1} corresponding to  $U\in \mathcal U_{ad}.$ Then by the similar arguments of the previous theorem, there exists a sequence  $\{(u_n,U_n)\}\subset \widetilde {\mathcal A}_{ad},$ which is a weak solution of \eqref{1.1}, that is,
	{\small\begin{eqnarray*} \label{sfe1}\left\{\begin{array}{lcl}
				\frac{d}{dt}\left[u_n+\mu  \mathsf A u_n\right]+\nu  \mathsf A u_n+\mathsf B(u_n)
				+ \alpha u_n + \beta \mathsf{D}(u_n)-U_n=0 \ \  \mbox{in} \ \ \mathbb{V}^\prime, \ a.e. \ t\in(0,T) \\[1mm]
				u_n(0)=u_0 \  \mbox{in} \  \mathbb{V}.
			\end{array}\right.
	\end{eqnarray*}}
	
	It is evident that $\{u_n\}$ is uniformly bounded in $\mathsf L^\infty(0,T;\mathbb{V}) ,$ and hence \eqref{2.34} shows that  $\{\mathsf D(u_n)\}$ is uniformly bounded in $\mathsf L^2(0,T;\mathbb{V}^\prime).$ Further, $\{\frac{du_n}{dt}\}$ is uniformly bounded in $\mathsf L^2(0,T;\mathbb V).$
	Consequently, there exists a subsequence, still denoted as $\{u_n\},$ such that
	\begin{eqnarray*}\left\{\begin{array}{cccll} \label{wl0}
			U_n &\overset{w}{\rightharpoonup}& \widetilde  U
			& \text{in} &   \mathsf L^2(0,T;\mathbb{H}),\\
			u_n &\overset{w^\ast}{\rightharpoonup}& \widetilde  u  &\text{in} &   \mathsf{L}^\infty(0,T;\mathbb V) \\
			u_n &\overset{w}{\rightharpoonup}& \widetilde  u  &\text{in} &  \mathsf L^2(0,T;\mathbb{V})\\
			\frac{du_n}{dt} &\overset{w}{\rightharpoonup}&   \frac{d\widetilde u}{dt}  &\text{in} &  \mathsf L^2(0,T;\mathbb V),  \ \ \mbox{as} \ \ n\to \infty,
		\end{array}\right.
	\end{eqnarray*}
	and by Aubin-Lion's compactness theorem,  $u_n \overset{s}{\rightarrow} \widetilde  u$ in   $\mathsf L^2(0,T;\mathbb{H}).$ Further, $\mathsf A u_n,\break   \frac{d \mathsf A u_n}{dt}$ converges weakly, respectively, to $\mathsf A \widetilde  u, \frac{d\mathsf A\widetilde  u}{dt}$   in $\mathsf L^2(0,T;\mathbb{V}^\prime),$   and by the arguments of the above theorem	$\mathsf D(u_n) \overset{w}{\rightharpoonup} \mathsf D( \widetilde u) $
	in $\mathsf L^2(0,T;\mathbb{V}^\prime),$ as $n\to \infty.$
	
	Finally, we prove the weak convergence of $\mathsf B(u_n).$ For any $v\in\mathsf L^2(0,T;\mathbb V),$ by  applying H\"older's inequality, the continuous embedding $\mathbb V\hookrightarrow\mathbb L^4$ and  \eqref{la1},  we have
	\begin{eqnarray*}\label{wl2}
		&&\lefteqn{\int_0^T \langle \mathsf B(u_n)- \mathsf B(\widetilde  u), v \rangle dt }\nonumber\\
		&=&\int_0^T\big(-\mathrm b(u_n,v,u_n-\widetilde  u)+\mathrm b(u_n-\widetilde  u,\widetilde u,v)\big)dt\nonumber\\
		&\leq& \int_0^T\left(\|u_n\|_{\mathbb L^4}\|\nabla v\|_{\mathbb L^2} \|u_n-\widetilde  u\|_{\mathbb L^4}+\|u_n-\widetilde  u\|_{\mathbb L^4}\|\nabla \widetilde  u\|_{\mathbb L^2}\|v\|_{\mathbb L^4}\right)dt \nonumber\\
		&\leq& C\left(\|u_n\|_{\mathsf L^8(0,T;\mathbb V)}+\|\widetilde  u\|_{\mathsf L^8(0,T;\mathbb V)} \right)\|u_n-\widetilde  u\|^{3/4}_{\mathsf L^3(0,T;\mathbb V)}\\
		&&\times\|u_n-\widetilde  u\|^{1/4}_{\mathsf L^2(0,T;\mathbb H)}\|v\|_{\mathsf L^2(0,T;\mathbb V)}
		\to 0, \ \ \mbox{as}  \ \ n\to\infty,
	\end{eqnarray*}
	where we used the facts that $u_n,\widetilde  u\in \mathsf L^\infty(0,T;\mathbb V) $ and $\|u_n-\widetilde  u\|_{\mathsf L^2(0,T;\mathbb H)}\to 0$ as  $n\to\infty.$ Hence $\mathsf B(u_n)\overset{w}{\rightharpoonup}\mathsf B(\widetilde u)$ in $\mathsf L^2(0,T;\mathbb V^\prime).$
	
	Since $\widetilde  u\in \mathsf C([0,T];\mathbb{V})$, we can conclude that \eqref{sf} holds true in $\mathbb V^\prime, \ a.e. \ t\in (0,T),$ that is,  $\widetilde  u\in \mathsf H^1(0,T;\mathbb{V})$ is a unique weak solution of \eqref{1.1} with control $	\widetilde  U\in \mathcal U_{ad}.$ Further, by the previous arguments, $(\widetilde  u,\widetilde  U)$ is an optimal pair solving (OCP).
	
	However, when we prove first and second-order optimality conditions, we require a more regular solution of \eqref{1.1} to obtain an \emph{a priori} estimates for linearized and adjoint systems of	\eqref{1.1}.  It forces us to start with a strong solution as our admissible class $\mathcal A_{ad},$ which we proved for periodic domain (Theorem \ref{sst}).
\end{remark}

\section{First-order necessary optimality conditions} \label{Se5}
The main result of this section is the derivation of first-order necessary conditions satisfied by any optimal pair of the (OCP)  obtained in Section \ref{Se4}, when  the growth factor of the damping term $r\geq 2.$ However, the results are clearly true for $r=1$ as well.  The optimality conditions provide a way to characterize an optimal control in terms of solutions of an \emph{adjoint system} associated with the (OCP). In this context, the Fr\'echet derivative of the cost functional  requires the solvability of a \emph{linearized system} of \eqref{1.1}.  Thus, we first study the well-posedness of the linearized system of  \eqref{1.1}.
\subsection{Control-to-state operator}
For a given class of  controls $U\in \mathcal U,$  let $u_U$ denote the induced unique strong solution of \eqref{1.1} given by Theorem \ref{sst}. This enables us to define a \emph{control-to-state  operator} $\mathcal S:\mathcal U \to \mathcal Z_1, \ \ U\mapsto \mathcal S(U):=u_U,$ where $\mathcal Z_1:=\mathsf H^1(0,T;\mathbb V).$ More precisely, we have $\mathcal {S}(\mathcal U)\subset \mathcal Z \subset \mathcal Z_1,$ since $\mathbb H^2\subset \mathbb V$ and $\mathbb L^{r+1}\subset \mathbb H,$ for any $r> 1,$ where $\mathcal Z$ is the strong solution space defined in  introduction. We rewrite  (OCP)  by using the control-to-state operator $\mathcal S$ as follows:
\begin{eqnarray*}
	\textrm{(MOCP)}\left\{\begin{array}{lclclc}
		\text{minimize} \ \mathfrak{J}(U) \\
		\text{subject to the control constraint} \ U\in \mathcal U_{ad}
	\end{array}\right.
\end{eqnarray*}
where $\mathfrak J(\cdot)$ is the \emph{reduced cost functional} defined by
$\mathfrak J(U):=\mathcal J(\mathcal S(U),U)=\break \mathcal J (u_U,U).$ The Fr\'echet derivative of the reduced cost functional with respect to the control requires that of the control-to-state-operator $\mathcal S,$ which in turn demands   the Lipschitz continuity of this operator.

\begin{lemma}[Lipschitz Continuity of $\mathcal S$] \label{LCS}
	The control-to-state map $\mathcal S:\mathcal U \to \mathcal Z_1$ is Lipschitz continuous, i.e., there exists a constant $K_1>0$ depending on system parameters,   $\Omega, T,C_P,R$ and $\|u_0\|_{\mathbb H^2}$ such that
	\begin{eqnarray}\label{LCC}
		\|\mathcal S(U_1)-\mathcal S(U_2)\|_{\mathcal Z_1} \leq K_1 \|U_1-U_2\|_{\mathsf L^2(0,T;\mathbb H)}, \ \ \forall U_1,U_2\in\mathcal U.
	\end{eqnarray}
\end{lemma}
\begin{proof} The proof is a direct consequence of Theorem \ref{cdt}. Indeed, let
	$u_{U_1},u_{U_2}$ be two strong solutions of \eqref{1.1} corresponding to the controls $U_1,U_2 \in \mathcal U,$ same initial data $u_0\in \mathbb{H}^2\cap\mathbb{V}$ and pressure $p_1,p_2$ respectively. The stability estimate \eqref{se} gives the required estimate \eqref{LCC}.
\end{proof}
\begin{remark} \label{lcr}
	In the case of bounded domain and $1\leq r\leq 5,$  by Theorem \ref{cdt} and Remark \ref{remw} , the control-to-state-operator $\mathcal S$ is Lipschitz continuous with constant $K_1$ depending on the constant $C_1$ of the energy estimate \eqref{ee}.
\end{remark}
The Fr\'echet derivative of the control-to-state operator $\mathcal S$ on $\mathcal U$ is a linear approximation of this operator  at some point in $ \mathcal U,$ 	which will be given by a solution of a linearized version of \eqref{1.1}.
Let $u_{\widetilde U}$ be the unique strong solution of the system \eqref{1.1} corresponding to the data $ u_0 \in \mathbb H^2\cap\mathbb V$ and the control $\widetilde  U\in \mathcal U.$ The linearized  equation of  \eqref{1.1} at some point $\widetilde U\in\mathcal U$ with a  function  $ V\in \mathsf L^2(0,T;\mathbb H)$ is given by
\begin{eqnarray}\label{ls1}\textrm{(L-NSVD)}\left\{\begin{array}{lclcrr}
		\mathcal L_{\widetilde U}w+\nabla q = V  \ \ \mbox{in} \ \ \Omega_T\\[2mm]
		\nabla\cdot w=0  \ \ \ \mbox{in}  \ \ \Omega_T , \ \ \ \
		w(x,0)=0 \ \mbox{in} \ \Omega,
	\end{array}\right.
\end{eqnarray}
where $q$ stands for the linearized pressure, the linear operator
\begin{eqnarray} \label{LL1}
	\mathcal L_{\widetilde U}w:=	w_t-\mu \Delta w_t -\nu \Delta w
	+(w\cdot\nabla) u_{\widetilde  U}+(  u_{\widetilde  U}\cdot\nabla) w+\alpha w+\beta f^\prime(  u_{\widetilde  U})w
\end{eqnarray} and  $f^\prime(\cdot)$  of the function $f(u)=|u|^{r-1} u$ is given in \eqref{ld}.

Now, we discuss the well-posedness of this linearized system. The existence and uniqueness can be completed by using the Faedo-Galerkin approximation technique once we obtain the required an \emph{a priori} estimates.
Testing \eqref{ls1} by $w_t,$ applying Young's inequality and \eqref{2.1},  we obtain
\begin{eqnarray} \label{ls4}
	&&\lefteqn{\frac{1}{2}\frac{d}{dt}\left[\alpha\|w(t)\|^2_{\mathbb H}+\nu\|w(t)\|^2_{\mathbb V}\right]+\|w_t(t)\|^2_{\mathbb H} +\mu\|w_t(t)\|^2_{\mathbb V} }\nonumber\\
	& =&(V,w_t)  -\mathrm b(w,u_{\widetilde  U},w_t)-\mathrm b(u_{\widetilde  U},w,w_t)-\beta  (f^\prime(u_{\widetilde  U})w, w_t)\nonumber\\
	&\leq& \frac{1}{4\delta_1}\|V(t)\|^2_{\mathbb H} + \delta_1 \|w_t(t)\|^2_{\mathbb H} +C(\delta_2)\|u_{\widetilde  U}(t)\|^2_{\mathbb{V}} \|w(t)\|_{\mathbb{V}}^2\nonumber\\
	&&+\delta_2 \|w_t(t)\|_{\mathbb{V}}^2-\beta  (f^\prime(u_{\widetilde  U})w, w_t),
\end{eqnarray}
for some $\delta_1,\delta_2>0.$ Using \eqref{ld} for any $r\geq 3,$   the continuous embedding $\mathbb H^2 \hookrightarrow \mathbb L^\infty$ gives
\begin{eqnarray} \label{ls6}
	&&\lefteqn{\beta  \int_\Omega (f^\prime(u_{\widetilde  U})w) \cdot w_t  dx}\nonumber\\
	&=&\beta\int_\Omega \left((r-1) |u_{\widetilde  U}|^{r-3}(u_{\widetilde  U}\cdot w)(u_{\widetilde  U}\cdot w_t) +|u_{\widetilde  U}|^{r-1}(w\cdot w_t)\right) dx\nonumber\\
	&\leq& C\int_\Omega|u_{\widetilde  U}|^{r-1}|w||w_t|dx
	\leq C\|u_{\widetilde  U}(t)\|_{\mathbb L^\infty}^{r-1}\|w(t)\|_{\mathbb{H}} \|w_t(t)\|_{\mathbb{H}} \nonumber \\
	& \leq& C(\Omega,C_P,\delta_1)\|u_{\widetilde  U}(t)\|_{\mathbb H^2}^{2(r-1)}\|w(t)\|^2_{\mathbb{V}} +\delta_1 \|w_t(t)\|_{\mathbb{H}}^2,
\end{eqnarray}
where we also employed the Poincar\'e inequality in the last inequality. Moreover, 	for  any $2\leq r<3, |u_{\widetilde  U}|\neq 0,$ the  estimation \eqref{ls6} holds.
%If $r=2$ and $|u_{\widetilde  U}|\neq 0,$  we have the following estimate by the embedding $\mathbb V \hookrightarrow \mathbb L^4$ and Young's inequality
%\begin{eqnarray} \label{ls66}
%	\beta  \int_\Omega (f^\prime(u_{\widetilde  U}) [w]) \cdot w_t  dx&=&\beta\int_\Omega \left(\frac{(u_{\widetilde  U}\cdot w_t)}{|u_{\widetilde  U}|}(u_{\widetilde  U}\cdot w) +|u_{\widetilde  U}|(w\cdot w_t)\right) dx \nonumber\\
%	&\leq& C\int_\Omega|u_{\widetilde  U}||w||w_t|dx\nonumber\\
%	&\leq& C\|u_{\widetilde  U}(t)\|_{\mathbb L^4}\|w(t)\|_{\mathbb L^4} \|w_t(t)\|_{\mathbb{L}^2} \nonumber\\
%	&\leq& C(\Omega,\delta_1)\|u_{\widetilde  U}(t)\|_{\mathbb V}^{2}\|w(t)\|^2_{\mathbb{V}} +\delta_1 \|w_t(t)\|_{\mathbb{H}}^2.
%\end{eqnarray}
Choosing $\delta_1=1/4,\delta_2=\mu/2,$ substituting \eqref{ls6} into \eqref{ls4}, and applying Gronwall's inequality gives
\begin{eqnarray} \label{ls8}
	\Psi(t,w)&\leq& \|V\|^2_{\mathsf L^2(0,T;\mathbb{H})}\nonumber\\
	&&\times \exp\left(C(\Omega,C_P) T\left(\|u_{\widetilde  U}\|^2_{\mathsf L^\infty(0,T;\mathbb{V})}+\|u_{\widetilde  U}\|_{\mathsf L^\infty(0,T;\mathbb H^2)}^{2(r-1)}\right)\right)<+\infty, \ \
\end{eqnarray}
for all $t\in (0,T],$ since by Theorem \ref{sst}, $u_{\widetilde  U}\in \mathsf{L}^\infty(0,T;\mathbb H^2),$ where
$$
\Psi(t,w):=\alpha\|w(t)\|^2_{\mathbb H}+\nu\|w(t)\|^2_{\mathbb V}+\|w_s\|^2_{\mathsf L^2(0,t;\mathbb H)}+\mu\|w_s\|^2_{\mathsf L^2(0,t;\mathbb V)}.
$$
From \eqref{ls8}, we infer that $w \in \mathsf L^2(0,T;\mathbb V), \ w_t\in \mathsf L^2(0,T;\mathbb V),$ whence $w \in \mathsf H^1(0,T;\mathbb V).$ Since the embedding  $\mathsf H^1(0,T;\mathbb V) \hookrightarrow \mathsf C([0,T];\mathbb V)$ is continuous, we can verify the initial condition. The uniqueness of the linear system \eqref{ls1} easily follows from \eqref{ls8}.
Thus, we have proved:
\begin{theorem}[Weak Solutions  of Linearized System] \label{WSLS}
	Let $\widetilde  U \in \mathcal U$ be any control and  $ u_{\widetilde  U}$ be the corresponding unique strong solution of \eqref{1.1}. Then for any $V\in \mathsf L^2(0,T;\mathbb H),$ there exists a unique weak solution of the linearized system \eqref{ls1} such that $w \in \mathsf L^\infty(0,T;\mathbb V), \ w_t\in \mathsf L^2(0,T;\mathbb V),$ and as a consequence $w\in \mathsf C([0,T];\mathbb V).$
\end{theorem}	
Now, we prove the Fr\'echet differentiability and Lipschitz continuity of the Fr\'echet derivative of the control-to-state operator $S$ on the open subset $\mathcal U.$ It is worth noting that these two results are proved for any open subset $\mathcal U$ of $\mathsf L^2(0,T;\mathbb H)$ rather than $\mathcal U_{ad}$ itself since the Fr\'echet derivative is merely defined for open subsets of $\mathsf L^2(0,T;\mathbb H).$
\begin{proposition} \label{FDS} For any $\widetilde  U\in \mathcal U, $ let $u_{\widetilde  U}$ be the unique strong solution of \eqref{1.1}. Then the following two conclusions hold:
	\begin{enumerate}	
		\item[(i)] 	The control-to-state  mapping $\mathcal S$ is Fr\'echet differentiable on $\mathcal U,$ that is, for any $\widetilde  U\in \mathcal U, $  there exists a bounded linear operator $\mathcal S^\prime(\widetilde  U): \mathsf L^2(0,T;\mathbb H)\to  \mathcal Z_1$  such that $$\frac{\|\mathcal S(\widetilde  U+U)-\mathcal S(\widetilde  U)-\mathcal S ^\prime (\widetilde  U)U\|_{\mathcal Z_1}}{\|U\|_{\mathsf L^2(0,T;\mathbb H) }} \to 0 \ \ \ \ \mbox{as} \ \ \ \|U\|_{\mathsf L^2(0,T;\mathbb H) } \to 0.$$ Moreover, for any $\widetilde  U\in \mathcal U,$ the Fr\'echet derivative $\mathcal S^\prime(\widetilde  U)$ is given by $\mathcal S ^\prime (\widetilde  U)U=w_{\widetilde U}^\prime[U], \ \forall U \in \mathsf L^2(0,T;\mathbb H),$ where $w_{\widetilde U}^\prime[U]$ is the unique weak solution of  \eqref{ls1}  associated with the control $U\in \mathsf L^2(0,T;\mathbb H).$
		\item[(ii)] The Fr\'echet derivative $\mathcal S^\prime$  is Lipschitz continuous, that is, for any controls $U_1, U_2\in\mathcal U$ and  $U\in \mathsf L^2(0,T;\mathbb H),$ there exists a constant $K_2>0$ depending only on system parameters,  $\Omega, T,C_P,R$ and $\|u_0\|_{\mathbb H^2}$ such that
		\begin{eqnarray*}\label{LCC1}
			\|\mathcal S^\prime(U_1)U-\mathcal S^\prime(U_2)U\|_{\mathcal Z_1} \leq K_2 \|U_1-U_2\|_{\mathsf L^2(0,T;\mathbb H)}\|U\|_{\mathsf L^2(0,T;\mathbb H)}.
		\end{eqnarray*}
	\end{enumerate}
\end{proposition}
\begin{proof}[Proof of (i)] For any arbitrary but fixed $\widetilde  U\in \mathcal U,$ let $u_{\widetilde  U}$ be the unique strong solution of \eqref{1.1}. Since $\mathcal U$ is an open subset of $\mathsf L^2(0,T;\mathbb{ H}),$ there exists some $\rho>0$ such that for any $ U\in \mathsf L^2(0,T;\mathbb{ H})$  with $\|U\|_{\mathsf L^2(0,T;\mathbb{ H})} \leq \rho,$ we have $\widetilde U+  U\in \mathcal U.$ Let  $u_{\widetilde  U+U}$ be the unique strong solution of the system \eqref{1.1} in response to the control $\widetilde U+  U\in \mathcal U.$ Let $w_{\widetilde U}^\prime[U]$ be the unique weak solution of the linearized equation \eqref{ls1}.
	Then the difference defined by $z:=	u_{\widetilde  U+U}-u_{\widetilde  U}-w_{\widetilde U}^\prime[U]$ solves the system
	\begin{eqnarray}\label{pls1}\left\{\begin{array}{lllr}
			\mathcal L_{\widetilde U}z+ \nabla \widetilde p
			= \mathfrak U_1+\mathfrak U_2  \ \ \mbox{in} \ \ \Omega_T\\[2mm]
			\nabla\cdot z = 0 \ \  \ \mbox{in} \ \ \ \Omega_T, \ \ \ \
			z(x,0) =0 \ \ \ \mbox{in} \ \Omega,
		\end{array}\right.
	\end{eqnarray}
	where $\mathcal L_{\widetilde U}z$ is defined in \eqref{LL1}, the pressure $\widetilde  p:= p_{\widetilde  U+ U}- p_{\widetilde  U}	-  q,$  and the   terms
	\begin{eqnarray*}
		\mathfrak U_1(x,t)&:=&-\big[(u_{\widetilde  U+ U}\cdot\nabla)u_{\widetilde  U+ U}-(u_{\widetilde  U}\cdot\nabla)u_{\widetilde  U}-(u_{\widetilde  U}\cdot\nabla)\widehat u-(\widehat u\cdot\nabla)u_{\widetilde  U}\big],\\
		\mathfrak U_2(x,t)&:=&-\big[\beta f(u_{\widetilde  U+ U})-\beta f(u_{\widetilde  U})-\beta f^\prime(u_{\widetilde  U})\widehat u\big], \ \ \ \widehat u:=u_{\widetilde  U+U}-u_{\widetilde  U}.
	\end{eqnarray*}	
	
	Let us invoke Theorem \ref{WSLS} for the sovability of \eqref{pls1}, which requires to prove that $\mathfrak U_1,\mathfrak U_2 \in \mathsf L^2(0,T;\mathbb H).$  Note that $\mathfrak U_1$ can be simplified as $\mathfrak U_1=-(\widehat u\cdot \nabla )\widehat u.$ Taking  $u_{\widetilde  U}, u_{\widetilde  U+ U} \in \mathsf L^\infty(0,T;\mathbb H^2)$ into account, the inequality \eqref{2.r22} and the continuous embedding $\mathsf H^1(0,T; \mathbb V)\hookrightarrow \mathsf L^\infty(0,T; \mathbb V)$ lead to the estimate
	\begin{eqnarray} \label{U1}
		\|\mathfrak U_1\|^2_{\mathsf L^2(0,T;\mathbb H)} &\leq& CT \|\widehat u\|_{\mathsf L^\infty(0,T;\mathbb H^2)}\|\widehat u\|^3_{\mathsf L^\infty(0,T;\mathbb V)}\nonumber\\
		&\leq& CT[M_0\big(\widetilde  U+ U,\widetilde U\big)] \|\widehat u\|^3_{\mathsf L^\infty(0,T;\mathbb V)}
		\leq C(C_2,K_1) \|U\|^3_{\mathsf L^2(0,T;\mathbb H) }, \quad
	\end{eqnarray}
	where we used \eqref{ps1}, \eqref{LCC}, and for any $V,W\in \mathcal U:$   $$[M_0\big(V,W\big)]^p:=\|u_{V}\|^p_{\mathsf L^\infty(0,T;\mathbb H^2)}+\|u_{W}\|^p_{\mathsf L^\infty(0,T;\mathbb H^2)},p\geq 1.$$
	Let's apply Taylor's formula,
	\begin{eqnarray}\label{SOTF}
		f(u_{\widetilde  U+ U})=f(u_{\widetilde  U})+f^\prime(u_{\widetilde  U}) \widehat u+\int_0^1(1-\theta)f^{\prime\prime}(u_{\widetilde  U}+\theta \widehat u)[\widehat u, \widehat u]d\theta.
	\end{eqnarray}	 
	By computing $f^{\prime\prime}(\cdot)$ from \eqref{ld}\footnote{The second derivative of $f(\cdot)$ is given by
		\begin{eqnarray}\label{U1x}
			f^{\prime\prime}(p)[q,g]&=&(r-1)(r-3)|p|^{r-5}(p\cdot q)(p\cdot g)p\nonumber\\
			&&+(r-1)|p|^{r-3}\big((p\cdot q)g+(p\cdot g)q+(g\cdot q)p\big),  \ \ \ r\geq 5.
		\end{eqnarray}
		Further, for any $2<r<5, p\neq 0,$ the expression \eqref{U1x} is valid for $f^{\prime\prime}(p)[\cdot,\cdot]$, and also need to set that $f^{\prime\prime}(p)[\cdot,\cdot]=0,$ if $p=0.$}, for any $r\geq 5,$ we have
	\begin{eqnarray*}
		f^{\prime\prime}(u_{\widetilde  U}+\theta \widehat u)
		[\widehat u, \widehat u]
		&=&  (r-1)(r-3)  |u_{\widetilde  U}+\theta \widehat u|^{r-5}|(u_{\widetilde  U}+\theta \widehat u)\cdot \widehat u|^2(u_{\widetilde  U}+\theta \widehat u)\\
		&&+ (r-1)|u_{\widetilde  U}+\theta \widehat u|^{r-3}\left(2\left((u_{\widetilde  U}+\theta \widehat u)\cdot\widehat u\right)\widehat u+(u_{\widetilde  U}+\theta \widehat u)|\widehat u|^2\right).\nonumber
	\end{eqnarray*}	
	Using the embeddings  $\mathbb H^2 \hookrightarrow \mathbb L^\infty, \mathbb V \hookrightarrow \mathbb L^4, $  and \eqref{LCC},  we obtain
	\begin{eqnarray} \label{U2}
		\|\mathfrak U_2\|^2_{\mathsf L^2(0,T;\mathbb H)}&\leq& C \sup_{\theta\in(0,1)} \int_{\Omega_{T}}\int_0^1|\theta u_{\widetilde  U+U}+(1-\theta)u_{\widetilde{U}}|^{2(r-2)}|\widehat u|^4 d\theta dxdt \nonumber\\	
		&\leq& C 2^{2r-5} \int_{\Omega_{T}}\left(|u_{\widetilde  U+U}|^{2(r-2)}+|u_{\widetilde{U}}|^{2(r-2)}\right)|\widehat u|^4 dxdt	\nonumber\\
		&\leq&	C2^{2r-5} \int_0^T\left( \|u_{\widetilde  U+U}(t)\|^{2(r-2)}_{\mathbb L^\infty}  +\|u_{\widetilde{U}}(t)\|^{2(r-2)}_{\mathbb L^\infty} \right) \|\widehat u(t)\|_{\mathbb L^4}^4dt  \nonumber\\
		&\leq& C T [M_0\big(\widetilde  U+ U,\widetilde U\big)]^{2(r-2)}\|\widehat u\|_{\mathsf L^\infty(0,T;\mathbb V)}^4 \nonumber \\
		&\leq& C(C_2,K_1) \|U\|^4_{\mathsf L^2(0,T;\mathbb H) }.
	\end{eqnarray} Further, for any $2<r<5, |u_{\widetilde  U}+\theta \widehat u|\neq 0,$ we can obtain the bound  \eqref{U2}.
	
	From the estimates \eqref{U1} and \eqref{U2}, we notice that $\mathfrak U_1,\mathfrak U_2 \in \mathsf L^2(0,T;\mathbb H).$ Thus, for any $r\in (2,\infty),$ by Theorem \ref{WSLS}, the system \eqref{pls1} has a unique weak solution $z \in \mathsf L^\infty(0,T;\mathbb V), \ z_t\in \mathsf L^2(0,T;\mathbb V).$
	Moreover, repeating the estimations similar to \eqref{ls8} together with \eqref{U1} and \eqref{U2} yield that
	\begin{eqnarray} \label{ls81}
		\Psi(t,z)	&\leq& 2\left(\|\mathfrak U_1\|^2_{\mathsf L^2(0,T;\mathbb{H})}+\|\mathfrak U_2\|^2_{\mathsf L^2(0,T;\mathbb{H})}\right) \nonumber\\
		&&\times \exp\left(C(\Omega,C_P) T\left(\|  u_{\widetilde  U}\|^2_{\mathsf L^\infty(0,T;\mathbb{V})}
		+\| u_{\widetilde  U}\|_{\mathsf L^\infty(0,T;\mathbb H^2)}^{2(r-1)}\right)\right)  \nonumber \\
		&\leq& C(C_2,K_1,C_P,T) \left(\|U\|^3_{\mathsf L^2(0,T;\mathbb H)}+\|U\|^4_{\mathsf L^2(0,T;\mathbb H)}\right).
	\end{eqnarray}
	For any $r\in (2,\infty),$ from   \eqref{ls81},  we deduce the following convergence:
	$$\|z\|_{\mathcal Z_1} \leq C \left(\|U\|^{3/2}_{\mathsf L^2(0,T;\mathbb H)}+\|U\|^2_{\mathsf L^2(0,T;\mathbb H)}\right),$$ whence $\|z\|_{\mathcal Z_1}/\|U\|_{\mathsf L^2(0,T;\mathbb H)}\to 0 \  \mbox{as}  \ \|U\|_{\mathsf L^2(0,T;\mathbb H)}\to 0.$
	
	Next, for  the case of $r=2,$   we follow the proof of Theorem \ref{WSLS}.  Let us multiply \eqref{pls1} by $z_t.$ For the term $\mathfrak U_1,$ we get from \eqref{2.1}  that
	\begin{eqnarray}\label{er12}
		(\mathfrak U_1,z_t)=-\mathrm b(\widehat{u},\widehat u,z_t)  \leq C\|\widehat u(t)\|^4_{\mathbb{V}} +\delta_2\|z_t(t)\|_{\mathbb{V}}^2.	
	\end{eqnarray}
	By the first-order Taylor's formula, we get  $ \mathfrak U_2 =- \beta\int_0^1\left(f^{\prime}(u_{\widetilde  U}+\theta \widehat u)\widehat u- f^\prime(u_{\widetilde  U})\widehat u\right)d\theta.$
	For $r=2,$  the first derivative expression \eqref{ld} clearly leads to the following estimate
	{\small\begin{eqnarray} \label{U21}
			(\mathfrak U_2,z_t) &=& -\beta\int_{\Omega}\int_0^1\left[\left(\frac{(u_{\widetilde  U}+\theta \widehat u)\cdot\widehat u}{|(u_{\widetilde  U}+\theta \widehat u)|}(\theta\widehat u)+\frac{(u_{\widetilde  U}+\theta \widehat u)\cdot \widehat u }{|(u_{\widetilde  U}+\theta \widehat u)||u_{\widetilde{U}}|}u_{\widetilde{U}}\left(|u_{\widetilde{U}}|-|u_{\widetilde  U}+\theta \widehat u|\right)
			\nonumber\right.\right.\\
			&&\left.\left.+\frac{(\theta\widehat u)\cdot \widehat u}{|u_{\widetilde U}|}u_{\widetilde U}\right)\cdot z_t  + \left(|u_{\widetilde  U}+\theta \widehat u|- |u_{\widetilde{U}}|\right)(\widehat u\cdot z_t) \right] d\theta dx	\nonumber \\
			&\leq& 4\beta\sup_{\theta\in(0,1)} \int_{\Omega}\int_0^1 |\theta\widehat{u}||\widehat{u}||z_t|d\theta dx
			\leq  C(\delta_1) \|\widehat u(t)\|_{\mathbb{L}^4}^4+\delta_1 \|z_t(t)\|_{\mathbb{H}}^2.
	\end{eqnarray}}By proceeding as in the proof of Theorem \ref{WSLS},  choosing $\delta_1=1/4,\delta_2=\mu/4,$ we obtain from \eqref{er12}, \eqref{U21} in view of \eqref{ls8} and \eqref{LCC}  that
	\begin{eqnarray*} \label{ls81-2}
		\Psi(t,z) \leq C(C_2,K_1,C_P,T) \|U\|^4_{\mathsf L^2(0,T;\mathbb H)}.
	\end{eqnarray*}
	Hence, for 	$r=2,$ we get the convergence as in $r\in(2,\infty),$ which completes the proof of (i) for all $r\in[2,\infty).$ \\\medskip
	\noindent{\emph {Proof of (ii).}} Let us denote $w_1:=w^\prime_{ U_1}[U], w_2:=w^\prime_{ U_2}[U],$ where $w^\prime_{ U_1}[U]$ and $w^\prime_{ U_2}[U]$ are the weak solutions of \eqref{ls1} with control $U.$ Then the function $w:=w_1-w_2$ satisfies the equation
	\begin{eqnarray}\label{ls11}\left\{\begin{array}{lclcrr}
			\mathcal L_{U_1}w+\nabla q =\mathfrak U_3+\mathfrak U_4  \ \mbox{in} \ \Omega_T\\[2mm]
			\nabla\cdot w=0  \ \ \ \mbox{in}  \ \ \Omega_T , \ \ \ \
			w(x,0)=0 \ \mbox{in} \ \Omega,
		\end{array}\right.
	\end{eqnarray}
	where  $q:=q_{U_1}-q_{U_2},$
	$\mathfrak U_3:= -(w_2\cdot\nabla)\widetilde u- (\widetilde u\cdot \nabla )w_2,$
	$\mathfrak U_4:=	-\beta\big(f^\prime(u_{U_1})-f^\prime(u_{U_2})\big)w_2,$ and  $\widetilde u:=u_{U_1}-u_{U_2}.$
	
	The proof again follows the lines of proof of Theorem \ref{WSLS}.  We only look at the right-hand side terms of \eqref{ls11}.  Testing \eqref{ls11} by $w_t$  and  applying \eqref{2.1} we get
	\begin{eqnarray*}\label{lcfd}
		( \mathfrak U_3,w_t) =-\mathrm b(w_2,\widetilde u,w_t)-\mathrm b(\widetilde u,w_2,w_t)  \leq C(\delta_2)\|w_2(t)\|^2_{\mathbb{V}} \|\widetilde u(t)\|_{\mathbb{V}}^2+\delta_2 \|w_t(t)\|_{\mathbb{V}}^2.	
	\end{eqnarray*}
	
	For $r\geq 5,$ invoking \eqref{U1x}, Taylor's formula and applying the embeddings $\mathbb H^2 \hookrightarrow \mathbb L^\infty, \mathbb V \hookrightarrow \mathbb L^4,$ we obtain
	\begin{eqnarray}\label{lcfd1}
		&&\lefteqn{	|(\mathfrak U_4,w_t)|}\nonumber\\ &\leq& C(\delta_1) \left\|\int_0^1 f^{\prime\prime}(u_{U_2}(t)+\theta \widetilde u(t))[\widetilde u(t),w_2(t)]d\theta\right\|^2_{\mathbb H} +\delta_1\|w_t(t)\|_{\mathbb H}^2\nonumber	\\
		&\leq& C\int_\Omega\int_0^1|\theta u_{U_1}(t)+(1-\theta)  u_{U_2}(t)|^{2(r-2)}|w_2(t)|^2|\widetilde u(t)|^2d\theta dx+\delta_1\|w_t(t)\|_{\mathbb H}^2\nonumber\\
		&\leq& C\left( \|u_{U_1}(t)\|^{2(r-2)}_{\mathbb L^\infty}  +\|u_{U_2}(t)\|^{2(r-2)}_{\mathbb L^\infty} \right)\|w_2(t)\|^2_{\mathbb L^4}\|\widetilde u(t)\|_{\mathbb L^4}^2
		+\delta_1\|w_t(t)\|_{\mathbb H}^2 \nonumber\\
		&\leq& C[M_0(U_1,U_2)]^{2(r-2)}\|w_2(t)\|^2_{\mathbb V}\|\widetilde u(t)\|_{\mathbb V}^2+\delta_1\|w_t(t)\|_{\mathbb H}^2,
	\end{eqnarray}
	where $M_0$ is given in  \eqref{U1}. Moreover, for any $2<r<5, (u_{U_2}+\theta \widetilde u)\neq 0,$ we can obtain the estimate  \eqref{lcfd1}. 	Thus, for any $r\in (2,\infty),$ by  repeating the calculations similar to \eqref{ls4}-\eqref{ls8}, choosing $\delta_1=1/4,\delta_2=\mu/4,$ and using \eqref{LCC}, one can get that
	\begin{eqnarray} \label{lcfd3}
		\Psi(t,w)	&\leq& C\exp\left(C(\Omega,C_P) T\left(\|u_{ U_1}\|^2_{\mathsf L^\infty(0,T;\mathbb{V})}+\|u_{ U_1}\|_{\mathsf L^\infty(0,T;\mathbb H^2)}^{2(r-1)}\right)\right)\nonumber\\
		&&\times\left(1+[M_0(U_1,U_2)]^{2(r-2)}\right)\|\widetilde u\|_{\mathsf L^\infty(0,T;\mathbb V)}^2\|w_2\|^2_{\mathsf L^2(0,T;\mathbb V)}\nonumber \\
		&\leq& C(C_2,K_1,C_P,T)\|U\|^2_{\mathsf L^2(0,T;\mathbb H)}\|U_1-U_2\|^2_{\mathsf L^2(0,T;\mathbb H)},\ \ \forall t\in (0,T],
	\end{eqnarray}
	where we used \eqref{ls8} for $\|w_2\|^2_{\mathsf L^2(0,T;\mathbb V)}.$ Besides,  for the case of $r=2,$ we can get the bound \eqref{lcfd3}, since  by invoking  \eqref{U21}, we have
	\begin{eqnarray} \label{U2ex1}
		(\mathfrak U_4,w_t) \leq 4\beta  \|w_2(t)\|_{\mathbb{L}^4}\|\widetilde u(t)\|_{\mathbb{L}^4}\|w_t(t)\|_{\mathbb{H}} \leq C \|w_2(t)\|_{\mathbb{V}}^2\|\widetilde u(t)\|_{\mathbb{V}}^2+\delta_1\|w_t(t)\|^2_{\mathbb{H}}.
	\end{eqnarray} Thus, for any $r\in[2,\infty),$ there exists a constant $K_2(C_2,K_1)>0$  such that
	$\|w\|_{\mathcal Z_1}  \leq K_2\|U\|_{\mathsf L^2(0,T;\mathbb H)}\|U_1-U_2\|_{\mathsf L^2(0,T;\mathbb H)}.$
	The proof of (ii) is thus completed.
\end{proof}	
\subsection{First-order optimality conditions}	
In this subsection, we derive  optimality conditions satisfied by an optimal control. From  Theorem \ref{EOC}, it is evident that there exists an optimal solution $(u_{\widetilde  U},\widetilde  U)$ satisfying $\mathcal S(\widetilde  U)=u_{\widetilde  U}$ and  the pair $(\mathcal S(\widetilde  U), \widetilde  U)$ is an optimal solution for (MOCP).

By the Fr\'echet differentiability of $\mathcal S$ given by Proposition \ref{FDS},  the reduced cost functional $\mathfrak J(U)$ is Fr\'echet differentiable at every $ U\in\mathcal U.$
Moreover, since the admissible control set $\mathcal U_{ad}$ is convex and $\mathfrak J(U)$ is Fr\'echet differentiable, for any  minimizer $\widetilde  U\in \mathcal U_{ad}$ of the reduced functional $\mathfrak J(U),$ the  following \emph{variational inequality} holds (see, \cite{Li, Tr}):
\begin{eqnarray} \label{Jp}
	\mathfrak J^\prime(\widetilde  U)(U-\widetilde  U)\geq 0, \  \ \forall U\in\mathcal U_{ad}.
\end{eqnarray}
Indeed, from Lemma \ref{LCS} and Proposition \ref{FDS},  we obtain a variational inequality satisfied by an optimal control $\widetilde  U\in \mathcal U_{ad}.$ Let $ U\in\mathcal U$ be arbitrary but fixed. Then there exists some $\rho>0$ such that for any $ V\in \mathsf L^2(0,T;\mathbb{ H})$  with $\|V\|_{\mathsf L^2(0,T;\mathbb{ H})} \leq \rho,$ we have $ U+  V\in \mathcal U,$ so that the variation of the functional  	$\mathfrak J(\cdot)$ is given by
\begin{eqnarray}\label{VFxx}
	\mathfrak J(U+ V)  -	\mathfrak J(U) &=& \mathcal J(\mathcal S(U+V),U+V)  -	\mathcal J(\mathcal S(U),U) \nonumber\\
	&=& \frac{\kappa}{2} \int_0^T \|\nabla (u_{U+ V}(t)-u_{U}(t))\|_{\mathbb H}^2 dt \nonumber \\
	&&+  \kappa\int_0^T \big(\nabla (u_{U+V}(t)-u_{U}(t)), \nabla (u_{U}(t)-u_{d}(t))\big) dt \nonumber \\
	&& +\frac{\lambda}{2} \int_0^T \|V(t)\|_{\mathbb H} ^2 dt +\lambda\int_0^T (U(t),V(t))dt.
\end{eqnarray}
It is clear from Lemma \ref{LCS}	that the following estimate holds:
\begin{eqnarray}\label{C1}
	\int_0^T \|\nabla (u_{U+V}(t)-u_{U}(t))\|_{\mathbb H}^2 dt =  \|u_{U+V}-u_{U}\|_{\mathsf L^2(0,T;\mathbb V)}^2
	\leq  K_1^2\|V\|^2_{\mathsf L^2(0,T;\mathbb H)}.
\end{eqnarray}
Let $\tilde w:=w^\prime_{U}[V]$ be the unique weak solution of \eqref{ls1}.  Notice  that
\begin{eqnarray}\label{C2}
	&&\lefteqn{\int_0^T \big(\nabla (u_{U+V}(t)-u_{U}(t)), \nabla (u_{U}(t)-u_{d}(t))\big) dt\nonumber}\\
	&=&\int_0^T \big(\nabla \tilde w(t), \nabla (u_{U}(t)-u_{d}(t)\big) dt\nonumber\\
	&&+\int_0^T \big(\nabla \big(u_{U+V}(t)-u_{U}(t)-\tilde w(t)\big), \nabla (u_{U}(t)-u_{d}(t))\big) dt :=I_1+I_2.
\end{eqnarray} By the application of H\"older's inequality, we get
\begin{eqnarray}\label{C2xx}
	|I_2|\leq  \|u_{U+V}-u_{U}-\tilde w\|_{\mathsf L^2(0,T;\mathbb V)}\|u_{U}-u_d\|_{\mathsf L^2(0,T;\mathbb V)}.
\end{eqnarray}
Taking $u_{U},u_d\in \mathsf L^2(0,T;\mathbb V)$ into account,   invoking Proposition \ref{FDS}, we see that $|I_2|/\|V\|_{\mathsf L^2(0,T;\mathbb H)}\to 0$ as $\|V\|_{\mathsf L^2(0,T;\mathbb H)}\to 0.$

Substituting the identity \eqref{C2} into \eqref{VFxx},  rearranging the resultant integrals, dividing both sides  by $\|V\|_{\mathsf L^2(0,T;\mathbb H)}$ and taking $\|V\|_{\mathsf L^2(0,T;\mathbb H)}\to 0,$ we obtain through \eqref{C1} and \eqref{C2xx}  that
\begin{eqnarray} \label{GD1}
	\mathfrak J^\prime(U)V = \kappa\int_0^T\left(\nabla\tilde w(t),\nabla(u_{U}(t)-u_d(t))\right)dt +\lambda\int_0^T (U(t),V(t))dt.
\end{eqnarray} 	
Thus, from \eqref{Jp}, we have the following optimality inequality characterizing an  optimal control $\widetilde  U\in\mathcal U_{ad}$ of (MOCP):
\begin{theorem}\label{foc1} Let $\Omega$ be a periodic domain in $\mathbb R^3.$
	Suppose Assumption \ref{assc} holds true. Let $U\in\mathcal U_{ad}$ be an arbitrary control with state $ u_{U}=\mathcal S(U),$ then the reduced functional $\mathfrak J(U)$ is Fr\'echet differentiable  with the derivative \eqref{GD1}. If $\widetilde  U\in\mathcal U_{ad}$ is an optimal control for (MOCP) with associated state $u_{\widetilde  U}=\mathcal S(\widetilde  U),$ then the following variational inequality holds:
	\begin{eqnarray} \label{VIC}
		\mathfrak J^\prime(\widetilde  U)(U-\widetilde  U)&=&	\kappa\int_{\Omega_T} \nabla w^\prime_{\widetilde U}[U-\widetilde  U]\cdot \nabla(u_{\widetilde  U}-u_d)dxdt \nonumber\\
		&&+\lambda\int_{\Omega_T} \widetilde  U\cdot (U-\widetilde  U)dxdt\geq 0,
	\end{eqnarray}
	for all $U\in\mathcal U_{ad},$ where $w^\prime_{\widetilde U}[U-\widetilde  U]=\mathcal S^\prime(\widetilde  U)(U-\widetilde  U)\in\mathcal Z_1$ is a unique weak solution of the linearized system \eqref{ls1} with control $V=U-\widetilde  U.$	
\end{theorem}	
Next, we follow the classical adjoint problem approach to simplify the variational inequality \eqref{VIC}, in particular the first term, by expressing it as an equivalent integral defined in terms of a solution of an \emph{adjoint problem} of \eqref{1.1}. This will lead to devising a compact first-order optimality condition characterizing an optimal control of (MOCP). The adjoint system is derived by applying the formal Lagrangian method (see, \cite{Tr}, Chapter 3).

\noindent Consider the adjoint problem
\begin{eqnarray}\label{as}
	\textrm{(A-NSVD)}\left\{\begin{array}{lll}
		\mathcal E_{U}\varphi+\nabla \psi = -\kappa\Delta ( u_{U}-u_d) 	 \ \ \ \mbox{in} \ \ \ \Omega_0\\ [2mm]
		\nabla\cdot\varphi =0 \ \ \ \mbox{in} \ \ \  \Omega_0 ,  \ \ \  \varphi(x,T)-\mu\Delta \varphi(x,T) =0  \ \ \ \mbox{in} \ \ \ \Omega,
	\end{array}\right.
\end{eqnarray}
where $\Omega_0:=\Omega\times [0,T),$ $\psi$ denotes the adjoint pressure, the linear operator $$\mathcal E_{U}\varphi:=-\varphi_t +\mu\Delta \varphi_t-\nu\Delta \varphi + (\nabla  u_{U})^T\varphi-( u_{  U}\cdot\nabla)\varphi +\alpha\varphi +\beta f^\prime( u_{U})\varphi,$$ and $
f^\prime(\cdot)$ is defined in \eqref{ld}

Next, we define the weak solution for the adjoint system \eqref{as}.
\begin{definition} \label{was}
	Let $ U\in\mathcal U_{ad}$ be any control  with associated state $u_{U}=\mathcal S(U)$ and $r\geq 2.$  A function $\varphi\in \mathsf H^1(0,T;\mathbb V)$ is called a weak solution of \eqref{as} on the interval $[0,T]$ if the following hold:
	\begin{align*}
		&(i) \ \ -(\varphi_t,v) -\mu (\nabla \varphi_t , \nabla v) +\nu  (\nabla \varphi,\nabla v)
		+\mathrm b(v,u_{U},\varphi)-\mathrm b(u_{U},\varphi,v)+\alpha (\varphi,v)\\
		&\hspace{1in}+\beta (f^\prime(u_U)\varphi,v)
		=\kappa(\nabla(u_{U}-u_d),\nabla v), \  \ \forall \ v\in \mathbb{V}, \ a.e. \ t\in[0,T],  \medskip \\
		&(ii) \ \ (\varphi(T),v)+\mu(\nabla \varphi(T),\nabla v) =0, \ \ \forall v\in \mathbb V.
	\end{align*}
\end{definition}
\begin{theorem}[First-Order Optimality Conditions] \label{foc2} Let $\Omega$ be a periodic domain in $\mathbb R^3.$
	Suppose Assumption \ref{assc} holds true. Let $\widetilde  U\in\mathcal U_{ad}$ be an optimal control for (MOCP)  with associated state $u_{\widetilde  U}=\mathcal S(\widetilde  U).$ Then there exits an adjoint state $\varphi_{\widetilde U}$ associated to the state $u_{\widetilde  U}$   such that
	\begin{itemize}
		\item [(i)] $\varphi_{\widetilde{U}}$ is a unique weak solution of  \eqref{as} in the sense of Definition \ref{was}.
		\item [(ii)] for any admissible control $U\in\mathcal U_{ad},$   the following variational inequality holds
	\end{itemize}
	\begin{eqnarray} \label{VIC1}
		\mathfrak J^\prime(\widetilde  U)(U-\widetilde  U)=\int_{\Omega_T} (\varphi_{\widetilde U}+\lambda\widetilde  U)\cdot (U-\widetilde  U)dxdt\geq 0.
	\end{eqnarray}	
\end{theorem}
\begin{proof}
	We will  only prove required an \emph{a priori} estimates for the solvability of the adjoint system \eqref{as}.  The justification of weak solution can be done by the standard Galerkin approximations and convergence arguments.  By setting $\varphi:=\varphi_{\widetilde U}$ and taking inner product of \eqref{as} with $\varphi,$ we get
	\begin{eqnarray}\label{as1}
		&&	-\frac{1}{2} \frac{d}{dt}\left[\|\varphi(t)\|^2_{\mathbb{H}}+\mu\|\varphi(t)\|^2_{\mathbb{V}}\right] +\nu\|\varphi(t)\|^2_{\mathbb{V}}+\alpha \|\varphi(t)\|^2_{\mathbb{H}}+\beta( f^\prime(u_{\widetilde  U})\varphi, \varphi )\nonumber\\
		&= &-((\nabla u_{\widetilde  U})^{T}\varphi,\varphi)-\kappa\langle \Delta (u_{\widetilde  U}-u_d),\varphi\rangle,
	\end{eqnarray}	
	where we used the fact that $\int_\Omega(u_{\widetilde  U}\cdot \nabla)\varphi\cdot\varphi dx=\mathrm b(u_{\widetilde  U},\varphi,\varphi)=0.$ Using H\"older's inequality   and Ladyzhenskaya's inequality \eqref{la1}  followed by Young's inequality, we obtain
	\begin{eqnarray}\label{rhs1}
		&&\lefteqn{-((\nabla u_{\widetilde  U})^{T}\varphi,\varphi)=-((\varphi\cdot\nabla)u_{\widetilde  U}, \varphi)
			\leq \|\varphi(t)\|_{\mathbb{L}^4}^2 \|\nabla u_{\widetilde  U}(t)\|_{\mathbb{L}^2}}\\
		&\leq & C \|\varphi(t)\|_{\mathbb{H}}^{1/2} \|\varphi(t)\|_{\mathbb{V}}^{3/2} \|u_{\widetilde  U}(t)\|_{\mathbb{V}} 	
		\leq C(\delta_3)\|\varphi(t)\|_{\mathbb{H}}^2 \|u_{\widetilde  U}(t)\|_{\mathbb{V}}^4+\delta_3 \|\varphi(t)\|_{\mathbb{V}}^2.\nonumber
	\end{eqnarray}
	Integrating 	by parts and applying Young's inequality, we also have
	\begin{eqnarray}\label{rhs}
		-\kappa\langle \Delta (u_{\widetilde  U}-u_d),\varphi\rangle &=&\kappa(\nabla (u_{\widetilde  U}-u_d),\nabla\varphi) \nonumber\\
		&\leq&  C(\delta_3)\left(\|u_{\widetilde  U}(t)\|^2_{\mathbb{V}}+\|u_d(t)\|^2_{\mathbb V}\right) +\delta_3 \|\varphi(t)\|^2_{\mathbb V}.	
	\end{eqnarray} For any $r\geq 3,$ notice from  	the derivative \eqref{ld}  that
	\begin{eqnarray*}\label{nni}
		\beta  \int_\Omega (f^\prime(u_{\widetilde U}) \varphi) \cdot \varphi  dx = \beta\int_\Omega \left((r-1) |u_{\widetilde U}|^{r-3}|u_{\widetilde U}\cdot \varphi|^2 +|u_{\widetilde U} |^{r-1}|\varphi|^2\right) dx\geq 0.
	\end{eqnarray*}
	Similarly, this integral is non-negative for $2\leq r<3$ as well.
	Let use define energy integral
	\begin{eqnarray*}\label{EI1}
		\Phi(t,\varphi):=\|\varphi(t)\|^2_{\mathbb{H}}+\mu\|\varphi(t)\|^2_{\mathbb{V}} +\alpha \|\varphi\|^2_{\mathsf L^2(t,T;\mathbb{H})}+\nu\|\varphi\|^2_{\mathsf L^2(t,T;\mathbb{V})}, \ t\in[0,T).	
	\end{eqnarray*}
	Using \eqref{rhs1}-\eqref{rhs} in \eqref{as1}, choosing $\delta_3=\nu/4$ and applying Gronwall's inequality in $(t,T),$ we obtain
	\begin{eqnarray}\label{as2}
		\Phi(t,\varphi) \leq C \left(\|u_{\widetilde  U}\|^2_{\mathsf L^2(0,T;\mathbb{V})}+\|u_d\|^2_{\mathsf L^2(0,T;\mathbb V)}\right)
		\exp\left(C\|u_{\widetilde  U}\|_{\mathsf L^4(0,T;\mathbb{V})}^4\right)<+\infty,
	\end{eqnarray}
	for all $t\in[0,T),$  since $u_{\widetilde  U}\in \mathsf L^\infty(0,T;\mathbb V)$ and $u_d\in \mathsf L^2(0,T;\mathbb V),$ where the terminal condition at $t=T$ vanishes due to the following relation:
	\begin{eqnarray*}
		\|\varphi(T)\|^2_{\mathbb{H}}+\mu\|\varphi(T)\|^2_{\mathbb{V}} =
		(\varphi(T),\varphi(T))+\mu(\nabla\varphi(T),\nabla\varphi(T)) =0.
	\end{eqnarray*}  Now, taking inner product of \eqref{as} with $-\varphi_t$ and using \eqref{2.1},  we infer from \eqref{rhs} that
	\begin{eqnarray*} \label{as3}
		&&\lefteqn{-\frac{1}{2}\frac{d}{dt}\left[\alpha\|\varphi(t)\|^2_{\mathbb{H}}+\nu\|\varphi(t)\|^2_{\mathbb{V}}\right] +\|\varphi_t(t)\|^2_{\mathbb{H}}+\mu\|\varphi_t(t)\|^2_{\mathbb{V}}}\nonumber\\
		&=&((\nabla u_{\widetilde  U})^T\varphi,\varphi_t )-((u_{\widetilde  U}\cdot \nabla)\varphi, \varphi_t)+\kappa\langle \Delta(u_{\widetilde  U}-u_d),\varphi_t\rangle+\beta  (f^\prime(u_{\widetilde  U}) \varphi,\varphi_t)\\
		&\leq& \delta_4 \|\varphi_t(t)\|_{\mathbb V}^2+C(\delta_4)  \|u_{\widetilde  U}(t)\|_{\mathbb V}^2\|\varphi(t)\|_{\mathbb V}^2\\
		&&+C(\delta_4)(\|u_{\widetilde  U}(t)\|^2_{\mathbb{V}}+\|u_d(t)\|^2_{\mathbb V})+\beta  (f^\prime(u_{\widetilde  U}) \varphi,\varphi_t).\nonumber
	\end{eqnarray*}
	By invoking  \eqref{ls6}, choosing $\delta_1=1/2,\delta_4=\mu/2$  and integrating over $(t,T],$ one may obtain
	\begin{eqnarray}\label{as4}
		\Phi_1(t,\varphi)&\leq &C(\Omega,C_P)\left(\|u_{\widetilde  U}\|^2_{\mathsf L^\infty(0,T;\mathbb{V})}+\|u_{\widetilde  U}\|_{\mathsf L^\infty(0,T;\mathbb H^2)}^{2(r-1)}\right)\|\varphi \|_{\mathsf L^2(0,T;\mathbb V)}^2 \\
		&&+ C\left(\|\varphi(T)\|^2_{\mathbb{H}}+\|\varphi(T)\|^2_{\mathbb{V}}+ \|u_{\widetilde  U}\|^2_{\mathsf L^2(0,T;\mathbb{V})}+\|u_d\|^2_{\mathsf L^2(0,T;\mathbb V)}\right)<+\infty, \nonumber \ \
	\end{eqnarray}
	for all $t\in[0,T),$ where
	$$\Phi_1(t,\varphi):=\alpha\|\varphi(t)\|^2_{\mathbb{H}}+\nu\|\varphi(t)\|^2_{\mathbb{V}} +\|\varphi_s\|^2_{\mathsf L^2(t,T;\mathbb{H})}+\mu\|\varphi_s\|^2_{\mathsf L^2(t,T;\mathbb{V})}  $$
	and note that $$\|\varphi(T)\|^2_{\mathbb{H}}+\|\varphi(T)\|^2_{\mathbb{V}} \leq \sup_{t\in[0,T]}\Big( \|\varphi(t)\|^2_{\mathbb{H}}+\|\varphi(t)\|^2_{\mathbb{V}}\Big)<+\infty.$$
	Here we employed the fact  that $u_{\widetilde  U}\in \mathsf L^\infty(0,T;\mathbb{H}^2),$ $u_d\in \mathsf L^2(0,T;\mathbb V)$ and  \eqref{as2}.  The estimate \eqref{as4} shows that $\varphi \in \mathsf L^2(0,T;\mathbb V), \ \varphi_t\in \mathsf L^2(0,T;\mathbb V)$ and it is enough to prove the existence of a weak solution of the adjoint system \eqref{as}. It also shows that $\varphi\in \mathsf C([0,T];\mathbb V)$ which is  sufficient to verify the  condition  $\varphi(x,T)-\mu\Delta \varphi(x,T) =0$ in the weak sense as in Definition \ref{was}-(ii). Moreover, the uniqueness  of weak solution of the linear system \eqref{as} directly follows from \eqref{as2}. This proves part (i).
	
	Next, let us express the  integral $\kappa\int_0^T\int_\Omega \nabla w^\prime_{\widetilde U}[U-\widetilde U]\cdot \nabla(u_{\widetilde  U}-u_d)dxdt$ of \eqref{VIC} in terms of the adjoint variable and control. For brevity, we use $\tilde w:=w^\prime_{\widetilde U}[U-\widetilde U]$ for a weak solution of \eqref{ls1}.   Choosing $V=U-\widetilde  U$ in \eqref{ls1}, testing  by the adjoint variable $\varphi$ and space integrating by parts, one may obtain
	\begin{eqnarray} \label{as5}
		\lefteqn{\int_{\Omega_T}\big(\varphi \cdot \tilde w_t+\mu\nabla\varphi\cdot \nabla \tilde w_t+\nu\nabla\varphi\cdot \nabla \tilde w\big)dxdt}\\
		&&\!\!+\!\!\int_{\Omega_T}\Big((\nabla   u_{\widetilde U})^T\varphi-(  u_{\widetilde U}\cdot \nabla)\varphi+\alpha \varphi+\beta f^\prime(  u_{\widetilde U})\varphi\Big)\cdot \tilde w  dxdt
		=  \int_{\Omega_T}\varphi\cdot (U-\widetilde  U)  dxdt,\!\!\nonumber
	\end{eqnarray}
	where we used $ \int_{\Omega_T} \nabla q\cdot \varphi dxdt=0,$ and \eqref{2.2} to get $\mathrm b( u_{\widetilde U}, \tilde w,\varphi) =-\mathrm b(  u_{\widetilde U},\varphi,\tilde w).$
	Integrating by parts with respect to time and using the initial/final data conditions of \eqref{ls1} and \eqref{as}, we  get
	$\int_{\Omega_T}\big(\varphi \cdot \tilde w_t+\mu\nabla\varphi\cdot \nabla \tilde w_t\big)dxdt=-\int_{\Omega_T}\big(\varphi_t \cdot \tilde w+\mu\nabla\varphi_t\cdot \nabla \tilde w\big)dxdt.$ On the other hand, testing the adjoint equation \eqref{as} with $\tilde w$ and comparing  with \eqref{as5} yields
	\begin{eqnarray} \label{as6}
		\kappa\int_{\Omega_T} \nabla\tilde w\cdot \nabla(u_{\widetilde  U}-u_d)dxdt= \int_{\Omega_T}\varphi\cdot (U-\widetilde  U) dxdt .
	\end{eqnarray}
	By invoking Theorem \ref{foc1}, we replace the first integral of \eqref{VIC}  by \eqref{as6} to complete the proof.
\end{proof}
The following corollary gives a pointwise version of the variational inequality \eqref{VIC1} which provides a useful characterization of an optimal control given by the regularization parameter and admissible control set $\mathcal U_{ad}.$ Since the admissible control set $\mathcal U_{ad}$ is a closed,  convex, and non-empty subset of $\mathsf L^2(0,T;\mathbb H),$ we can characterize the optimal control by a  \emph{projection formula}.

For any given $(a,b)\in\mathbb R^2$ with $a\leq b$ and $\tau\in \mathbb R,$ let $\mathcal P_{[a,b]}$ denote  the  projection of $\mathbb R$ onto $[a,b]$ which is defined by
%\begin{eqnarray}\label{pro}
$\mathcal P_{[a,b]}(\tau) :=\min\Big\{b,\max\big\{a,\tau\big\}\Big\}. $
%\end{eqnarray}

\begin{corollary} Let $\widetilde  U\in \mathcal U_{ad}$ be an optimal control for (MOCP) and $\varphi=\varphi_{\widetilde U}$ be the solution of the adjoint equation \eqref{as}. Then the optimal control is characterized by three different cases:
	\begin{itemize}
		\item [$(i)$] If $\lambda>0,$	then $\widetilde  U$ is given by the projection formula
		\begin{eqnarray}\label{pr1}
			\widetilde  U(x,t)= \mathcal P_{[U_{\min},U_{\max}]}\left(-\frac{\varphi(x,t)}{\lambda}\right), \ \  \mbox{for a.e.} \ \ (x,t)\in \Omega_T,
		\end{eqnarray}
		where $\mathcal P$ is the projection of $\mathsf L^2(0,T;\mathbb H)$ onto $\mathcal U_{ad}$ as defined above.
		\item [$(ii)$] If $\lambda>0$ and $\mathcal U_{ad}=\mathsf L^2(0,T;\mathbb H),$ then the unconstrained optimal control is given by the direct relation
		$\widetilde  U(x,t)=-\frac{\varphi(x,t)}{\lambda}, \ \  \mbox{for a.e.} \ \ (x,t)\in \Omega_T.$
		\item [$(iii)$] If $\lambda=0$  and $\mathcal U_{ad} \subseteq \mathsf L^2(0,T;\mathbb H),$ then the control is of bang-bang type given by
		\begin{eqnarray*}
			\widetilde  U(x,t)=	\left\{
			\begin{array}{lclclc}
				U_{\min}  &\mbox{if}& \varphi(x,t)>0 \\
				U_{\max}  &\mbox{if}& \varphi(x,t)<0, \ \ \mbox{for a.e.} \ \  (x,t)\in \Omega_T.
			\end{array}
			\right.
		\end{eqnarray*}
	\end{itemize}	
\end{corollary}	
\begin{proof}
	The proof can be completed  by invoking (see, \cite{Tr}, Lemma 2.26 and Theorem 2.28) that the variational inequality \eqref{VIC1} is equivalent to the pointwise variational inequality
	$
	(\varphi+\lambda\widetilde  U)\cdot (U-\widetilde  U)\geq 0 \ \ \forall U\in [U_{\min},U_{\max}], \ \mbox{for a.e.} \ (x,t)\in \Omega_T,
	$
	and the definition of $\mathcal P_{[U_{\min},U_{\max}]}.$	
\end{proof}		

\subsection{First-order optimality conditions in bounded domain  $\Omega\subset \mathbb R^3$} \label{FOBD}

In the previous section, we established a first-order optimality condition of  (MOCP) in the periodic domain. As we noticed earlier, while proving the well-posedness of the linearized system \eqref{ls1} and the adjoint system \eqref{as}, we needed the strong solution of \eqref{1.1} in estimating \eqref{ls6},\eqref{U2},\eqref{lcfd1} and \eqref{as4}. Nevertheless, the strong solution is obtained when the domain $\Omega$ is periodic in $\mathbb R^3.$ Now, we restrict the growth of the damping term $f(u)=|u|^{r-1}u, r\in[2,\infty) $ appropriately so that the optimality conditions (Theorem \ref{foc2}) hold true for bounded domain $\Omega\subset \mathbb R^3.$  In other words, we prove Theorem \ref{foc2} by using only the weak solution of \eqref{1.1}, which exists for bounded domain as well.

We will only list out the restrictions needed on $r$ to get the crucial estimations on the bounded domain.
The proof of Theorem \ref{WSLS} requires only the following estimate on the damping term. If we restrict the growth of the damping term in \eqref{1.1} to $3\leq r\leq 5,$ the estimate \eqref{ls6} can be modified using the embedding $\mathbb V\hookrightarrow \mathbb{L}^{\frac{6(r+1)}{11-r}}$ and $\mathbb V\hookrightarrow \mathbb{L}^{r+1}, \ r\leq 5$ as follows:
\begin{eqnarray} \label{ls6n1}
	\beta  \left\langle f^\prime(  u_{\widetilde U}) w, w_t\right\rangle &\leq& \beta \|f^\prime(  u_{\widetilde U}) w\|_{\mathbb{L}^{6/5}}\|w_t(t)\|_{\mathbb{L}^6} \nonumber \\
	&\leq& C\|  u_{\widetilde U}(t)\|_{\mathbb L^{r+1}}^{r-1}\|w(t)\|_{\mathbb{L}^{\frac{6(r+1)}{11-r}}}\|w_t(t)\|_{\mathbb{L}^6}\nonumber\\
	&\leq& C\|  u_{\widetilde U}(t)\|_{\mathbb V}^{r-1}\|w(t)\|_{\mathbb V}  \|w_t(t)\|_{\mathbb{V}} \nonumber\\
	&\leq& C(\delta_2)\| u_{\widetilde U}(t)\|_{\mathbb V}^{2(r-1)}\|w(t)\|^2_{\mathbb V}+  \delta_2\|w_t(t)\|^2_{\mathbb{V}}.
\end{eqnarray}
Choosing $\delta_1=1/2,\delta_2=\mu/4,$ the estimate \eqref{ls8} can now read as follows	
\begin{eqnarray} \label{5.7xx}
	\Psi(t,w) &\leq&  \exp\left(C(\Omega,C_P) T\left(\|  u_{\widetilde U}\|^2_{\mathsf L^\infty(0,T;\mathbb{V})}+\|  u_{\widetilde U}\|^{2(r-1)}_{\mathsf L^\infty(0,T;\mathbb V)}\right)\right)\nonumber\\
	&&\times\|V\|^2_{\mathsf L^2(0,T;\mathbb H)}  <+\infty, \ \ \forall t\in (0,T],
\end{eqnarray}
since $  u_{\widetilde U}\in  \mathsf L^\infty(0,T;\mathbb{V})$ by the weak solutions of \eqref{1.1} given in Theorem \ref{eu}. The bound  given above holds true for $2\leq r<3$ as well.

Next, let us examine Proposition \ref{FDS}.
Taking inner product of  \eqref{pls1} with $z_t$ and estimate the right-hand side terms.  For $\langle \mathfrak U_1, z_t\rangle,$ we use \eqref{er12}
%By invoking  \eqref{2.1}, the estimate \eqref{U1} can be modified as follows
%\begin{eqnarray}\label{U11}
%|\langle \mathfrak U_1, z_t\rangle | = |\mathrm b(\widehat u,\widehat u, z_t)| \leq C \|\widehat u(t)\|^2_{\mathbb V}\|z_t(t)\|_{\mathbb V} \leq C(\delta_2) \|\widehat u(t)\|^4_{\mathbb V}+ \delta_2\|z_t(t)\|^2_{\mathbb V}.
%\end{eqnarray}	
and for the term $\langle \mathfrak U_2, z_t\rangle,$ in the case of $ 2<r\leq 5,$ we alter \eqref{U2} using the embedding $\mathbb V\hookrightarrow \mathbb L^{\frac{12(r+1)}{17-r}},$ and $\mathbb V\hookrightarrow \mathbb{L}^{r+1}, \ r\leq 5$ as follows
\begin{eqnarray}\label{U22}
	&&\lefteqn{|\langle \mathfrak U_2, z_t\rangle |} \nonumber\\
	&\leq& C\sup_{\theta\in(0,1)}\left\|\int_0^1(1-\theta)f^{\prime\prime}(u_{\widetilde  U}+\theta \widehat u)[\widehat u, \widehat u]d\theta\right\|_{\mathbb L^{\frac{6}{5}}}  \|z_t(t)\|_{\mathbb L^6} \nonumber\\
	&\leq& C(r)\sup_{\theta\in(0,1)}\left[\int_\Omega\left(\int_0^1|\theta u_{\widetilde  U+U}+(1-\theta)u_{\widetilde{U}}|^{r-2}|\widehat u|^2d\theta\right)^{6/5}dx\right]^{5/6}  \|z_t(t)\|_{\mathbb L^6} \nonumber\\
	&\leq& C(r)\left[\int_\Omega\left(| u_{\widetilde  U+U}|+|u_{\widetilde{U}}|\right)^{r+1}dx\right]^{\frac{r-2}{r+1}}\|\widehat u(t)\|^{2}_{\mathbb L^{\frac{12(r+1)}{17-r}}}  \|z_t(t)\|_{\mathbb L^6} \nonumber\\
	&\leq&C(\delta_2,r)\left(\|u_{\widetilde  U+U}(t)\|^{2(r-2)}_{\mathbb V} +\|u_{\widetilde  U}(t)\|^{2(r-2)}_{\mathbb V}\right)\|\widehat u(t)\|^{4}_{\mathbb V}  +\delta_2 \|z_t(t)\|^2_{\mathbb V}.
\end{eqnarray}
By doing calculations similar to Theorem \ref{WSLS} (keeping in mind \eqref{ls6n1}),  using \eqref{er12} and \eqref{U22} with $\delta_2=\mu/8,$ one may obtain
\begin{eqnarray*} \label{ls811}
	\Psi(t,z)&\leq&  C\exp\left(CT\left(\| u_{\widetilde  U}\|^2_{\mathsf L^\infty(0,T;\mathbb{V})}
	+\|  u_{\widetilde  U}\|_{\mathsf L^\infty(0,T;\mathbb V)}^{2(r-1)}\right)\right)\|U\|^4_{\mathsf L^2(0,T;\mathbb H) } \\
	&\leq& C \|U\|^4_{\mathsf L^2(0,T;\mathbb H)},
\end{eqnarray*}	
where $C$ depends on $C_1,K_1$ and $R.$ By \eqref{U21}, the above conclusion also holds true for $r=2.$  Thus, we can prove part-(i) of Proposition \ref{FDS}. In view of \eqref{5.7xx}, the part-(ii) requires only to prove \eqref{lcfd1} and that can again be  obtained by following \eqref{U22}. This completes the proof of Proposition \ref{FDS} in the case of bounded domain when $2\leq r\leq 5.$ For the solvability of adjoint system \eqref{as}, the main estimation \eqref{as4}, which we concluded  from \eqref{ls6} can now be obtained from \eqref{ls6n1}. Thus, by combining the preceding arguments,   we infer that when the growth of the damping term is restricted to $2\leq r\leq 5,$ the first-order optimality condition of (MOCP) proved in Theorem \ref{foc2} holds true for bounded domain $\Omega\subset \mathbb R^3$ with zero Dirichlet boundary conditions.

\section{Second-order sufficient optimality conditions} \label{Se6}
It is known that for convex optimal control problems, any  control satisfying the first-order necessary optimality conditions is globally optimal. However, for the non-convex optimal control problems, one may need to do further higher derivative analysis to guarantee a local optimal control. In the case of optimal control problems governed by Navier-Stokes equations,  second-order sufficient optimality conditions play a pivotal role in the numerical analysis of these non-convex optimal control problems. A control that satisfies second-order sufficient optimality conditions is stable with respect to any perturbation of the given data (see, \cite{Tr1}, also \cite{Lij})
\subsection{Control-to-costate operator}
In this section, we establish a sufficient criteria  for local optimality condition. If $\widetilde U$ satisfies the variational ineqaulity \eqref{VIC1} and suppose, we assume that $\mathfrak J^{\prime\prime}(\widetilde U)[U,U]> 0,$ for all directions
$U\in \mathsf L^2(0,T;\mathbb H)\backslash \{0\},$ then the control $\widetilde U$ is a strict local minimizer of $\mathfrak J(\cdot)$ on the admissible control set $\mathcal U_{ad}.$   But the  positivity condition defined on all directions can be relaxed to only certain critical directions (see, \cite{Cas,Tr} and also \cite{Eb}).
To get the second Fr\'echet differentiability of the functional $\mathfrak J(\cdot),$ we study the \emph{control-to-costate} operator $\mathcal A :\mathcal U \to \mathcal Z_1,U\mapsto \mathcal A(U):=\varphi_U,$  which assigns for any control $ U\in\mathcal U,$ a unique weak solution $\varphi_U\in \mathcal Z_1$ of the adjoint system \eqref{as}. We prove that the operator $\mathcal A$ is Lipschitz continuous and Fr\'echet differentiable.  In this section, we assume that the growth parameter $r\geq 3$ in the nonlinearity of the damping  term. This restriction arises due to the fact that the third-order G\^ateaux derivative of the damping  term exists  only for $r\geq 3.$
\begin{lemma}\label{LCA} The control-to-costate mapping $\mathcal A :\mathcal U \to \mathcal Z_1$ is Lipschitz continuous, that is, there exists a constant $K_3>0$ depending only on   $K_1,\Omega, T,C_P,R,\|u_0\|_{\mathbb H^2}$  such that
	\begin{eqnarray}\label{LCC10}
		\|\mathcal A(U_1)-\mathcal A(U_2)\|_{\mathcal Z_1} \leq K_3 \|U_1-U_2\|_{\mathsf L^2(0,T;\mathbb H)}, \ \ \forall U_1,U_2\in\mathcal U.
	\end{eqnarray}
\end{lemma}
\begin{proof}
	Let $u_{U_1},u_{U_2}$ and $\varphi_{U_1},\varphi_{U_2}$  be the strong and weak solutions of \eqref{1.1} and \eqref{as}, respectively   associated with the controls $U_1,U_2\in\mathcal U.$  Define $\widehat u:=u_{U_1}-u_{U_2}, \varphi:=\mathcal A(U_1)-\mathcal A(U_2)=\varphi_{U_1}-\varphi_{U_2}$ and the pressure $\widehat \psi:=\psi_{U_1}-\psi_{U_2}.$  Then the pair $(\varphi,\widehat \psi)$ solves the equation
	\begin{eqnarray}\label{lcadxx}
		\left\{\begin{array}{lll}
			\mathcal E_{U_1}\varphi+\nabla\widehat\psi =-\kappa \Delta \widehat u+ \mathfrak V_1+\mathfrak V_2 \ \ \	 \mbox{in} \ \ \ \Omega_0\\ [2mm]
			\nabla\cdot\varphi =0 \ \ \ \ \mbox{in} \ \ \   \Omega_0 ,  \ \ \ \  \varphi(x,T)-\mu\Delta \varphi(x,T) =0 \ \ \mbox{in} \ \ \ \Omega,
		\end{array}\right.
	\end{eqnarray}
	where $\mathcal E_{U_1}\varphi$ is defined in \eqref{as},
	$$\mathfrak V_1= -(\nabla \widehat u)^T\varphi_{U_2}+(\widehat u\cdot \nabla)\varphi_{U_2},  \ \ \mbox{and} \ \
	\mathfrak V_2=-\beta\left(f^\prime(u_{U_1})-f^\prime(u_{U_2})\right)\varphi_{U_2}.$$
	Testing \eqref{lcadxx} with $\varphi$ and doing estimations similar to \eqref{rhs1},\eqref{rhs} and \eqref{lcfd1}, we obtain  that
	\begin{eqnarray}\label{lcad1}
		&&\lefteqn{|(-\kappa \Delta \widehat u+\mathfrak V_1+\mathfrak V_2,\varphi)|}\nonumber\\
		&\leq& C(\delta_3) \|\widehat u(t)\|^2_{\mathbb V} +\delta_3 \|\varphi(t)\|^2_{\mathbb V}+C(\delta_3) \|\widehat u(t)\|^2_{\mathbb V} \|\varphi_{U_2}(t)\|^2_{\mathbb V}\\
		&&+C(\delta_5)\left(\|u_{U_1}(t)\|^{2(r-2)}_{\mathbb H^2}+\|u_{U_2}(t)\|^{2(r-2)}_{\mathbb H^2}\right)\|\varphi_{U_2}(t)\|^2_{\mathbb V}\|\widehat u(t)\|_{\mathbb V}^2+\delta_5\|\varphi(t)\|_{\mathbb H}^2, \nonumber
	\end{eqnarray}
	for any $r\geq 3.$ By proceeding as in part-(i) of Theorem \ref{foc2} and using \eqref{lcad1}, we infer from \eqref{as2} and \eqref{as4} that the proof  can be completed by invoking Lemma \ref{LCS}.
	\iffalse
	\begin{eqnarray}\label{lcfd22x}
		\Phi(t,\varphi)	\leq M~ \|\widehat u\|_{\mathsf L^2(0,T;\mathbb V)}^2 \exp\left(C T\|u_{ U_1}\|_{\mathsf L^\infty(0,T;\mathbb V)}^4\right) 	\leq C(C_2,K_1)\|U_1-U_2\|^2_{\mathsf L^2(0,T;\mathbb H)},
	\end{eqnarray}
	for all $t\in [0,T),$ where we employed \eqref{ps1} and \eqref{LCC}, the integral $\Phi(t,\varphi)$ is defined in \eqref{EI1},  $M := \left(1+\left(1+[M_0(U_1,U_2)]^{2(r-2)}\right)\|\varphi_{U_2}\|^2_{\mathsf L^\infty(0,T;\mathbb V)}\right)$ and recall $M_0$ from \eqref{U1}.
	
	Multiplying \eqref{as1} by $\varphi_t,$ one can obtain an inequality similar to \eqref{lcad1} and so, we infer from \eqref{as4} and \eqref{lcfd2} that
	\begin{eqnarray}\label{lcfd33x}
		\Phi_1(t,\varphi)&\leq& C(C_P)\left(\left(\|u_{U_1}\|^2_{\mathsf L^\infty(0,T;\mathbb{V})}+\|u_{U_1}\|_{\mathsf L^\infty(0,T;\mathbb H^2)}^{2(r-1)}\right)\|\varphi \|_{\mathsf L^2(0,T;\mathbb V)}^2+\|\varphi\|^2_{\mathsf L^\infty(0,T;\mathbb{V})}\right) + CM\ \|\widehat u\|_{\mathsf L^2(0,T;\mathbb V)}^2.\nonumber\\
		&\leq& C(C_2,K_1)\|U_1-U_2\|^2_{\mathsf L^2(0,T;\mathbb H)} \ \ \forall t\in [0,T),
	\end{eqnarray}
	where $\Phi_1(t,\varphi)$ is defined in \eqref{as4}.
	\fi
\end{proof}
\begin{proposition} \label{FDC} For any $\widetilde  U\in \mathcal U, $ let $u_{\widetilde  U}$ be the unique strong solution of \eqref{1.1}. Then we have:
	\begin{enumerate}	
		\item[(i)] 	The control-to-costate  mapping $\mathcal A$ is Fr\'echet differentiable on $\mathcal U,$ that is, for any $\widetilde  U\in \mathcal U, $ there exists a bounded linear operator $\mathcal A^\prime(\widetilde  U): \mathsf L^2(0,T;\mathbb H)\to  \mathcal Z_1$  such that $$\frac{\|\mathcal A(\widetilde  U+U)-\mathcal A(\widetilde  U)-\mathcal A ^\prime (\widetilde  U)U\|_{\mathcal Z_1}}{\|U\|_{\mathsf L^2(0,T;\mathbb H) }} \to 0 \ \ \ \ \mbox{as} \ \ \ \|U\|_{\mathsf L^2(0,T;\mathbb H) } \to 0.$$ Moreover, for any $\widetilde  U\in \mathcal U, $ the Fr\'echet derivative is given by $\mathcal A ^\prime (\widetilde  U)U=\varphi^\prime_{\widetilde U}[U], \forall U\in \mathsf L^2(0,T;\mathbb H),$  where  $\varphi^\prime_{\widetilde U}[U]$ is the unique weak solution of the equation
		\begin{eqnarray}\label{as11}
			\left\{\begin{array}{lll}
				\mathcal E_{\widetilde U}\phi+\nabla \bar \psi = -\kappa\Delta  w^\prime_{\widetilde  U}[U] -(\nabla w^\prime_{\widetilde U}[U])^T\varphi _{\widetilde U}+(w^\prime_{\widetilde U}[U]\cdot \nabla)\varphi_{\widetilde U}\\ [1mm]\hspace{1in}-\beta f^{\prime\prime}(u_{\widetilde U})[w^\prime_{\widetilde U},\varphi_{\widetilde U} ] \ \ \	\ \ \mbox{in} \ \ \Omega_0 \\ [2mm]
				\nabla\cdot\phi =0 \ \ \mbox{in} \ \  \Omega_0 ,  \ \   \phi(x,T)-\mu\Delta \phi(x,T) =0 \  \ \ \ \mbox{in} \ \ \ \Omega,
			\end{array}\right.
		\end{eqnarray}
		while $w^\prime_{\widetilde U}[U]$ is the weak solution of \eqref{ls1} with control $U,$  $\varphi_{\widetilde U}$ is that of \eqref{as} and $\mathcal E_{\widetilde U}\phi$ is the linear operator defined in \eqref{as}.
		\item[(ii)] For any controls $U_1, U_2\in\mathcal U$ and  $U\in \mathsf L^2(0,T;\mathbb H),$ there exists a constant $K_4>0$  such that
		\begin{eqnarray*}\label{LCC1x}
			\|\mathcal A^\prime(U_1)U-\mathcal A^\prime(U_2)U\|_{\mathcal Z_1} \leq K_4 \|U_1-U_2\|_{\mathsf L^2(0,T;\mathbb H)}\|U\|_{\mathsf L^2(0,T;\mathbb H)},
		\end{eqnarray*}
		where $K_4$  depends on system parameters,  $K_1,K_2,K_3,\Omega, T,C_P,R$ and  $\|u_0\|_{\mathbb H^2}.$
	\end{enumerate}
\end{proposition}
\begin{proof}
	By following the line of proof  of Theorem \ref{foc2}, we can show that \eqref{as11}  admits a unique weak solution $\varphi^\prime_{\widetilde U}[U]\in \mathsf H^1(0,T;\mathbb V)$.
	To prove the Fr\'echet differentiability, as in the proof of Proposition \ref{FDS}-(i), let us consider  $\xi_{\widetilde U}:=\varphi_{\widetilde U+U}-\varphi_{\widetilde U}-\varphi^\prime_{\widetilde U}[U]$ and the pressure $\widetilde \psi:=\psi_{\widetilde U+U}-\psi_{\widetilde U}-\bar\psi, \; z:=u_{\widetilde U+U}-u_{\widetilde U}-w^\prime_{\widetilde U}[U],$ where we recall that $z$ is a weak solution of \eqref{pls1}. Then we can check that $(\xi_{\widetilde U},\widetilde \psi)$ is the unique weak solution of the equation
	\begin{eqnarray}\label{as12}
		\left\{\begin{array}{lll}
			\mathcal E_{\widetilde U}\xi+\nabla \widetilde\psi = -\kappa\Delta z +\mathfrak V_3+\mathfrak V_4+\mathfrak V_5  \	\ \ \mbox{in} \ \ \Omega_0 \\ [2mm]
			\nabla\cdot\xi =0 \ \ \mbox{in} \ \  \Omega_0 ,  \ \ \  \  \xi(x,T)-\mu\Delta \xi(x,T) =0 \  \ \ \ \mbox{in} \ \ \ \Omega,
		\end{array}\right.
	\end{eqnarray}
	where $ \mathfrak V_3:=-(\nabla z)^T\varphi_{\widetilde U}-(\nabla \widehat u)^T\widehat \varphi+(z\cdot \nabla )\varphi_{\widetilde U}+(\widehat u\cdot \nabla)\widehat \varphi,$
	\begin{eqnarray*}
		\mathfrak V_4&:=&-\beta f^{\prime\prime}(u_{\widetilde U})[z,\varphi_{\widetilde U}]-\beta\big(f^\prime(u_{\widetilde U+U})- f^\prime(u_{\widetilde U})\big)\widehat \varphi\\
		&=&-\beta f^{\prime\prime}(u_{\widetilde U})[z,\varphi_{\widetilde U}]-\beta \int_0^1f^{\prime\prime}(u_{\widetilde  U}+\theta \widehat u)[\widehat u,\widehat\varphi]d\theta ,\\
		\mathfrak V_5&:=&-\beta\big( f^\prime(u_{\widetilde U+U})\varphi_{\widetilde U}- f^\prime(u_{\widetilde U})\varphi_{\widetilde U} -f^{\prime\prime}(u_{\widetilde U})[\widehat u,\varphi_{\widetilde U}]\big)\\
		&=&-\beta\int_0^1(1-\bar \theta)f^{\prime\prime\prime}(u_{\widetilde U}+\bar\theta \widehat u)[\widehat u,\widehat u,\varphi_{\widetilde U}]d\bar\theta,
	\end{eqnarray*}
	where we used Taylor's formula and $\widehat u:=u_{\widetilde U+U}-u_{\widetilde U},\  \ \widehat \varphi:=\varphi_{\widetilde U+U}-\varphi_{\widetilde U}.$
	
	Let us estimate the right-hand side of \eqref{as12}.  Multiplying \eqref{as12} by $\xi,$   repeating the calculations \eqref{rhs1}-\eqref{rhs}, and using  \eqref{2.1}, one may obtain that
	\begin{eqnarray}\label{Fd1}
		|\langle-\kappa\Delta z+\mathfrak V_3,\xi\rangle|
		&\leq& C(\delta_3) \left(\|z(t)\|^2_{\mathbb V}+\|z(t)\|^2_{\mathbb V} \|\varphi_{\widetilde U}(t)\|^2_{\mathbb V} \right.\nonumber\\
		&&\left.+ \|\widehat u(t)\|^2_{\mathbb V} \|\widehat\varphi(t)\|^2_{\mathbb V}\right) +\delta_3 \|\xi(t)\|^2_{\mathbb V}.
	\end{eqnarray}
	For any $r\geq 3,$ by invoking  the second derivative formula \eqref{U1x}, we infer  from the inequality \eqref{lcfd1} that
	\begin{eqnarray*}\label{Fd2}
		|(\mathfrak V_4,\xi)|&\leq& 	C(\delta_5)\|u_{\widetilde U}(t)\|_{\mathbb H^2}^{2(r-2)}\|z(t)\|_{\mathbb V}^2\|\varphi_{\widetilde U}(t)\|_{\mathbb V}^2 +\delta_5\|\xi(t)\|_{\mathbb H}^2\nonumber \\
		&&+C(\delta_5)\left(\|u_{\widetilde U+U}(t)\|^{2(r-2)}_{\mathbb H^2}+\|u_{\widetilde U}(t)\|^{2(r-2)}_{\mathbb H^2}\right)\|\widehat u(t)\|_{\mathbb V}^2\|\widehat\varphi(t)\|_{\mathbb V}^2.
	\end{eqnarray*}
	Let us use \eqref{U1x} to compute the third derivative  $f^{\prime\prime\prime}(\cdot)[\cdot,\cdot,\cdot]$\footnote{\noindent For any $r\geq 7,$  we have\begin{eqnarray*}\label{tdf}
			f^{\prime\prime\prime}(p)[q,g,h]&=&(r-1)(r-3)(r-5)|p|^{r-7}(p\cdot q)(p\cdot g)(p\cdot h)p\nonumber\\
			&&+(r-1)(r-3)|p|^{r-5}\left[(p\cdot g)(h\cdot q)p+(p\cdot q)(h\cdot g)p+(p\cdot q)(p\cdot g)h\right.\\
			&&\left.+(p\cdot h)\Big((p\cdot q)g+(p\cdot g)q+(g\cdot q)p\Big)\right] \\
			&&+(r-1)|p|^{r-3}\left[(h\cdot q)g+(h\cdot g)q+(g\cdot q)h\right].	
		\end{eqnarray*} Further, for any $3<r<7, p\neq 0,$ the above formula is valid for $f^{\prime\prime\prime}(p)[\cdot,\cdot,\cdot]$, and also one needs to set that $f^{\prime\prime\prime}(p)[\cdot,\cdot,\cdot]=0,$ if $p=0.$ For $r=3,$ we get $f^{\prime\prime\prime}(p)[q,g,h]=2\left[(h\cdot q)g+(h\cdot g)q+(g\cdot q)h\right].$ }       and the embedding  $\mathbb V\hookrightarrow\mathbb L^6$ to get
	\begin{eqnarray}\label{Fd3}
		|(\mathfrak V_5,\xi)|
		&\leq& \beta\sup_{\bar\theta\in(0,1)}\left\|\int_0^1(1-\bar\theta) f^{\prime\prime\prime}(u_{\widetilde U}+\bar\theta \widehat u)[\widehat u,\widehat u,\varphi_{\widetilde U}]d\bar\theta\right\|_{\mathbb H}\|\xi(t)\|_{\mathbb H} \nonumber\\
		&\leq& C(\delta_5)\int_\Omega  (| u_{\widetilde  U+U}|+|u_{\widetilde{U}}|)^{2(r-3)}|\widehat u|^2|\widehat u|^2|\varphi_{\widetilde U}|^2 dx+\delta_5\|\xi(t)\|_{\mathbb H}^2\nonumber\\
		&\leq& C\left(\|u_{\widetilde  U+U}(t)\|^{2(r-3)}_{\mathbb L^\infty}\!+\!\|u_{\widetilde  U}(t)\|^{2(r-3)}_{\mathbb L^\infty}\right)\|\widehat u(t)\|_{\mathbb L^6}^4\|\varphi_{\widetilde U}(t)\|^2_{\mathbb L^6}\!+\!\delta_5\|\xi(t)\|_{\mathbb H}^2\nonumber\\
		&\leq& C\left(\|u_{\widetilde  U+U}(t)\|^{2(r-3)}_{\mathbb H^2}+ \|u_{\widetilde  U}(t)\|^{2(r-3)}_{\mathbb H^2}\right)\|\widehat u(t)\|_{\mathbb V}^4\|\varphi_{\widetilde U}(t)\|^2_{\mathbb V}\nonumber\\
		&&+\delta_5\|\xi(t)\|_{\mathbb H}^2,  \  \ \mbox{for any} \  r\geq 7.
	\end{eqnarray}
	For $3\leq r< 7, (u_{\widetilde U}+\bar\theta \widehat u)\neq 0,$ the bound \eqref{Fd3} holds true.
	
	For any $r\geq 3,$ taking the inequalities \eqref{Fd1}-\eqref{Fd3} into account, choosing $\delta_3=\nu/4,\delta_5=\alpha/4,$ we infer from   \eqref{as2}, and the continuous embedding $\mathsf H^1(0,T; \mathbb V)\hookrightarrow \mathsf L^\infty(0,T; \mathbb V)$ that
	\begin{eqnarray}\label{Fd4}
		\Phi(t,\xi)	&\leq&C\exp\left(C T\|u_{\widetilde U}\|_{\mathsf L^\infty(0,T;\mathbb V)}^4\right)\left[ M_1\| z\|_{\mathsf L^2(0,T;\mathbb V)}^2 \nonumber\right.\\
		&&\left.+M_2\|\widehat\varphi\|_{\mathsf L^\infty(0,T;\mathbb V)}^2\|\widehat u\|_{\mathsf L^2(0,T;\mathbb V)}^2 +M_3\|\widehat u\|_{\mathsf L^\infty(0,T;\mathbb V)}^4\right]\nonumber\\
		&\leq& C(C_2,K_1,K_3,R) \left(\|U\|^3_{\mathsf L^2(0,T;\mathbb H)}+\|U\|^4_{\mathsf L^2(0,T;\mathbb H)}\right), \ \ \forall t\in[0,T).
	\end{eqnarray}
	The last inequality follows from Proposition \ref{P1}, \eqref{ls81}, Theorem \ref{foc2}, Lemmas \ref{LCS},\ref{LCA} and
	\begin{eqnarray*}
		M_1&:=&\left[ 1+\|\varphi_{\widetilde U}\|^2_{\mathsf L^\infty(0,T;\mathbb V)}\left(1+\|u_{\widetilde U}\|_{\mathsf L^\infty(0,T;\mathbb H^2)}^{2(r-2)}\right)\right], \\
		M_2&:=&\left[1+[M_0\big(\widetilde  U+ U,\widetilde U\big)]^{2(r-2)}\right],  \ \
		M_3=[M_0\big(\widetilde  U+ U,\widetilde U\big)]^{2(r-3)}\|\varphi_{\widetilde U}\|^2_{\mathsf L^\infty(0,T;\mathbb V)},		
	\end{eqnarray*}
	and recall $M_0$  defined in \eqref{U1}.
	
	Multiplying \eqref{as12} by $-\xi_t,$ one can obtain in view of \eqref{as4} and straightforward modification of the  inequalities \eqref{Fd1}-\eqref{Fd3} that
	\begin{eqnarray}\label{Fd5}
		\Phi_1(t,\xi)	&\leq& C\left[\left(\|u_{\widetilde U}\|^2_{\mathsf L^\infty(0,T;\mathbb{V})}+\|u_{\widetilde U}\|_{\mathsf L^\infty(0,T;\mathbb H^2)}^{2(r-1)}\right)\|\xi\|_{\mathsf L^2(0,T;\mathbb V)}^2+\|\xi\|^2_{\mathsf L^\infty(0,T;\mathbb{V})}\right. \nonumber\\
		&&\left. +M_1\| z\|_{\mathsf L^2(0,T;\mathbb V)}^2 +M_2\|\widehat\varphi\|_{\mathsf L^\infty(0,T;\mathbb V)}^2\|\widehat u\|_{\mathsf L^2(0,T;\mathbb V)}^2 +M_3\|\widehat u\|_{\mathsf L^\infty(0,T;\mathbb V)}^4\right]\nonumber\\
		&\leq& C(C_2,K_1,K_3,R) \left(\|U\|^3_{\mathsf L^2(0,T;\mathbb H)}+\|U\|^4_{\mathsf L^2(0,T;\mathbb H)}\right), \ \ \forall t\in[0,T),
	\end{eqnarray}
	where we also used \eqref{Fd4}. From \eqref{Fd5}, one may notice that $\|\xi\|_{\mathcal Z_1}/\|U\|_{\mathsf L^2(0,T;\mathbb H)}\to 0$ as $\|U\|_{\mathsf L^2(0,T;\mathbb H)}\to 0,$ which completes the proof of part-(i).
	
	To prove part-(ii), for any two controls $U_1, U_2\in\mathcal U$ and  $U\in \mathsf L^2(0,T;\mathbb H),$ let    us set $\widehat\phi:=\varphi^\prime_{U_1}[U]-\varphi^\prime_{U_2}[U],  \psi:=\bar\psi_{U_1}-\bar\psi_{U_2},$  $\widehat u:=u_{U_1}-u_{U_2},\widehat w:=w^\prime_{U_1}[U]-w^\prime_{U_2}[U]$ and $\widehat\varphi:=\varphi_{U_1}-\varphi_{U_2}.$ For simplicity, we write $w^\prime_{U_i}[U],\varphi^\prime_{U_i}[U],i=1,2$ as $w^\prime_{U_i},\varphi^\prime_{U_i},i=1,2.$   Then $(\widehat \phi,\psi)$ solves the equation
	\begin{eqnarray*}\label{as13}
		\left\{\begin{array}{lll}
			\mathcal E_{U_1}\widehat\phi+\nabla \psi = -\kappa\Delta  \widehat w +\mathfrak V_6+\mathfrak V_7+\mathfrak V_8+\mathfrak V_9 \ \ \	\ \ \mbox{in} \ \ \Omega_0 \\ [2mm]
			\nabla\cdot\widehat\phi =0 \ \ \mbox{in} \ \  \Omega_0 ,  \ \ \  \  \widehat\phi(x,T)-\mu\Delta \widehat\phi(x,T) =0 \  \ \ \ \mbox{in} \ \ \ \Omega,
		\end{array}\right.
	\end{eqnarray*}
	where $\mathfrak V_6= -(\nabla \widehat u)^T\varphi^\prime_{U_2}+(\widehat u\cdot \nabla)\varphi^\prime_{U_2}  \ \mbox{and}  \
	\mathfrak V_7=-\beta\left(f^\prime(u_{U_1})-f^\prime(u_{U_2})\right)\varphi^\prime_{U_2},$ and
	{\small\begin{eqnarray*}
			\mathfrak V_8&:=&-(\nabla \widehat w)^T\varphi_{U_1}-(\nabla  w^\prime_{U_2})^T\widehat\varphi+(w^\prime_{U_1}\cdot \nabla)\widehat\varphi+(\widehat w\cdot \nabla)\varphi_{U_2},\\
			\mathfrak V_9&:=&-\beta\left(f^{\prime\prime}(u_{U_1})-f^{\prime\prime}(u_{U_2})\right)[w^\prime_{U_1},\varphi_{U_1}]-\beta f^{\prime\prime}(u_{U_2})[\widehat w,\varphi_{U_2}]-\beta f^{\prime\prime}(u_{U_2})[w^\prime_{U_1},\widehat \varphi].
	\end{eqnarray*}}The terms $\mathfrak V_6,\mathfrak V_7$ can be estimated as in \eqref{lcad1}.   In view of \eqref{2.1}, one can get
	\begin{eqnarray*}\label{as14}
		|\langle-\kappa\Delta  \widehat w+\mathfrak V_8,\widehat\phi\rangle|
		&\leq& \delta_3 \|\widehat\phi(t)\|^2_{\mathbb V}+C(\delta_3) \left(\|w^\prime_{U_1}(t)\|^2_{\mathbb V}+\|w^\prime_{U_2}(t)\|^2_{\mathbb V}\right) \|\widehat \varphi(t)\|^2_{\mathbb V}\nonumber\\
		&&+C(\delta_3) \left(1+\|\varphi_{U_1}(t)\|^2_{\mathbb V}+\|\varphi_{U_2}(t)\|^2_{\mathbb V}\right) \|\widehat w(t)\|^2_{\mathbb V}.
	\end{eqnarray*}
	Using again Taylor's formula, we write the term $\left(f^{\prime\prime}(u_{U_1})-f^{\prime\prime}(u_{U_2})\right)[w^\prime_{U_1},\varphi_{U_1}] =\int_0^1f^{\prime\prime\prime}(u_{U_2}+\theta \widehat u)[\widehat u,w^\prime_{U_1},\varphi_{U_1}]d\theta.$ 
	By computations similar to \eqref{lcfd1} and \eqref{Fd3}, for any $r\geq 3,$ we arrive at
	\begin{eqnarray*}
		|(\mathfrak V_9,\widehat\phi)|&\leq& C(\delta_5)\|u_{U_2}(t)\|_{\mathbb H^2}^{2(r-2)}\left(\|\widehat w(t)\|^2_{\mathbb V}\|\varphi_{U_2} (t)\|_{\mathbb V}^2+\|\widehat\varphi(t)\|^2_{\mathbb V}\|w^\prime_{U_1}(t)\|_{\mathbb V}^2\right)\nonumber \\
		&&+C(\delta_5)\left(\|u_{ U_1}(t)\|^{2(r-3)}_{\mathbb H^2}+ \|u_{U_2}(t)\|^{2(r-3)}_{\mathbb H^2}\right)\\
		&&\quad\times\|\widehat u(t)\|_{\mathbb V}^2\|w^\prime_{U_1}(t)\|_{\mathbb V}^2\|\varphi_{U_1}(t)\|^2_{\mathbb V}+\delta_5\|\widehat\phi(t)\|_{\mathbb H}^2.	
	\end{eqnarray*}
	The proof of this part follows by the same arguments of  part-(i) or Theorem \ref{foc2}. Indeed,
	note that $\varphi^\prime_{U_i}\in \mathsf L^\infty(0,T;\mathbb V),$  by Theorems \ref{WSLS},\ref{foc2}, the solutions $w^\prime_{U_i},\varphi_{U_i}\in \mathsf L^\infty(0,T;\mathbb V),$ and by Theorem \ref{sst}, strong solutions $u_{U_i}\in \mathsf L^\infty(0,T;\mathbb H^2).$ Thus, utilizing the Lipschitz continuity given by Lemmas \ref{LCS}, \ref{LCA} and Proposition \ref{FDS}, one can conclude the proof.
\end{proof}
\begin{remark}
	It is clear from Propositions \ref{FDS},\ref{FDC} that the control-to-state operator $\mathcal S:\mathcal U\to \mathcal Z_1$ and control-to-costate operator $\mathcal A:\mathcal U\to \mathcal Z_1$ are both Fr\'echet differentiable and the derivatives are Lipschitz continuous.  We may even show that they are  continuously differentiable  under appropriate conditions on data and the growth value of $r.$
\end{remark}

\subsection {Local optimality conditions} As we pointed out earlier in this section,  second-order sufficient conditions for a local optimal control of (MOCP) are written only in the \emph{cone of critical directions}. For details on the relation between the critical directions and second-order necessary/sufficient conditions, one may look at   \cite{Tr,Cas}.
\begin{definition}[Critical Cone] For  $\widetilde U\in \mathcal U_{ad},$ let $\mathcal C(\widetilde U)$ denotes the set of all  $U\in\mathsf L^2(0,T;\mathbb H)$ such that
	\begin{eqnarray*}
		U(x,t)\left\{\begin{array}{llll}
			\geq 0 &\mbox{if}& \widetilde U(x,t)=U_{\min}(x,t)\\
			\leq 0 &\mbox{if}& \widetilde U(x,t)=U_{\max}(x,t)\\
			= 0 &\mbox{if}& \varphi(x,t)+\lambda\widetilde U(x,t)\neq 0, \ \ \mbox{for allmost all}\ (x,t)\in\Omega_T.
		\end{array}
		\right.
	\end{eqnarray*}
	
\end{definition}
\begin{theorem} \label{SOC0} Let $\Omega$ be a periodic domain in $\mathbb R^3.$ Let $\widetilde U\in\mathcal U_{ad}$ be any  control  with  the adjoint sate $\varphi_{\widetilde U}$ satisfy the variational inequality \eqref{VIC1}. Moreover, assume that $\mathfrak J^{\prime\prime}(\widetilde U)[U,U]>0, $ that is,
	\begin{eqnarray}\label{sdc}
		-\int_{\Omega_T}(\varphi^\prime_{\widetilde U}[U]\cdot U)dxdt< \lambda \|U\|_{\mathsf L^2(0,T;\mathbb H)}^2, \ \ \mbox{for all}  \ \ U\in\mathcal C(\widetilde U)\backslash\{0\}.	
	\end{eqnarray}
	Then there exist constants $\theta>0$ and $\delta>0$ such that for all $\widehat U\in\mathcal U_{ad},$ the following inequality holds:
	\begin{eqnarray*}\label{svi1}
		\mathfrak J(\widehat U)\geq \mathfrak J(\widetilde U)+\frac{\theta}{2}\|\widehat U-\widetilde U\|^2_{\mathsf L^2(0,T;\mathbb H)}, \ \ \ \mbox{if} \ \ \ \|\widehat U-\widetilde U\|_{\mathsf L^2(0,T;\mathbb H)}< \delta.
	\end{eqnarray*}
	In other words, the control $\widetilde U$ is a strict local minimizer of the functional $\mathfrak J(\cdot)$ on the set $\mathcal U_{ad}.$
\end{theorem}	
\begin{proof}
	From the first-order Fr\'echet derivative \eqref{VIC1} of $\mathfrak J(U),$ the second derivative is given by
	$$ \mathfrak J^{\prime\prime}(\widetilde U)[U_1,U_2]=\lambda \int_{\Omega_T}U_1\cdot U_2dxdt +\int_{\Omega_T}\varphi_{\widetilde U}^\prime[U_2]\cdot U_1 dxdt, \ \ \mbox{for all} \ \ U_1,U_2\in \mathsf L^2(0,T;\mathbb{ H}).$$
	It is evident that the condition $-\int_{\Omega_T}(\varphi^\prime_{\widetilde U}[U]\cdot U)dxdt< \lambda \|U\|_{\mathsf L^2(0,T;\mathbb H)}^2 $ is equivalent to the positive definiteness of $\mathfrak J^{\prime\prime}(\cdot),$ that is,  $\mathfrak J^{\prime\prime}(\widetilde U)[U,U]>0,$  for all $U\in\mathcal C(\widetilde U)\backslash\{0\}.$	
	
	In view of Theorems 4.1, 4.3  of \cite{Cas} (see, also Theorem 27, \cite {Eb}), one can infer that the proof of this theorem can be completed if the following two convergences are attained:
	\begin{itemize}
		\item [(i)] For any sequence of admissible controls $\{\widetilde U_k\}\subset \mathcal U_{ad}$ and $\{U_k\}\subset  \mathsf L^2(0,T;\mathbb H)$  with  $\widetilde U_k \overset{s}{\rightarrow}  \widetilde U$ and $ U_k\overset{w}{\rightharpoonup}  U$ in $\mathsf L^2(0,T;\mathbb H),$  we need to show that
		$ \mathfrak J^\prime(\widetilde U_k)U_k\to \mathfrak J^\prime(\widetilde U)U\ \ \mbox{as} \ \ k\to\infty.$
		\item [(ii)] For any sequence  $\{U_k\}\subset  \mathsf L^2(0,T;\mathbb H)$  with   $ U_k\overset{w}{\rightharpoonup}  U$ in $\mathsf L^2(0,T;\mathbb H),$  it holds along a subsequence   that
		$$\int_{\Omega_T}\varphi_{\widetilde U}^\prime[U_k]\cdot U_k dxdt \to \int_{\Omega_T}\varphi_{\widetilde U}^\prime[U]\cdot Udxdt \ \ \mbox{as} \ \ k\to\infty.$$
	\end{itemize}
	To prove (i), recall the first-order Fr\'echet derivative  of $\mathfrak J(\cdot)$ is given by  \eqref{VIC1} that
	$$\mathfrak J^\prime(\widehat  U)U=\int_{\Omega_T} (\varphi_{\widehat U}+\lambda\widehat  U)\cdot Udxdt, \ \ \widehat U\in\mathcal U_{ad}, \ U\in\mathsf L^2(0,T;\mathbb H).$$
	Note that
	\begin{eqnarray}\label{poi}
		\mathfrak J^\prime(\widetilde U_k)U_k- \mathfrak J^\prime(\widetilde U)U= \mathfrak J^\prime(\widetilde U_k)U_k-\mathfrak J^\prime(\widetilde U)U_k+\mathfrak J^\prime(\widetilde U)U_k- \mathfrak J^\prime(\widetilde U)U.
	\end{eqnarray}	
	By applying H\"older's inequality, utilizing the fact that $\{U_k\}$ is uniformly bounded in $\mathsf L^2(0,T;\mathbb H)$ and  the Lipschitz continuity \eqref{LCC10}, we obtain
	\begin{eqnarray*}
		&&\lefteqn{\left|\mathfrak J^\prime(\widetilde U_k)U_k-\mathfrak J^\prime(\widetilde U)U_k\right| }\\
		&\leq& \left(\|\varphi_{\widetilde U_k}-\varphi_{\widetilde U}\|_{\mathsf L^2(0,T;\mathbb H)}+\lambda \|\widetilde  U_k-\widetilde{U}\|_{\mathsf L^2(0,T;\mathbb H)}\right)\|U_k\|_{\mathsf L^2(0,T;\mathbb H)} \nonumber\\
		&\leq& C(K_3,\lambda) \|\widetilde  U_k-\widetilde{U}\|_{\mathsf L^2(0,T;\mathbb H)}\to 0 \ \ \mbox{as} \ \  k\to\infty,
	\end{eqnarray*}
	where we invoked $\widetilde U_k \overset{s}{\rightarrow}  \widetilde U$ in $\mathsf L^2(0,T;\mathbb H).$
	Besides, since $ U_k\overset{w}{\rightharpoonup}  U$ in $\mathsf L^2(0,T;\mathbb H)$ as  $k\to\infty$ and $ (\varphi_{\widetilde U}+\lambda\widetilde  U)\in\mathsf L^2(0,T;\mathbb H),$ it is clear that $\mathfrak J^\prime(\widetilde U)U_k- \mathfrak J^\prime(\widetilde U)U\to 0$ as  $k\to\infty.$ Thus, taking $k\to\infty$ in \eqref{poi} and using the above two convergences, we arrive at the proof of (i).
	
	The proof of (ii) follows from  the steps analogues to that of (i). Let us consider the difference
	\begin{eqnarray}\label{poii}
		&&\lefteqn{	\int_{\Omega_T}\varphi_{\widetilde U}^\prime[U_k]\cdot U_k dxdt - \int_{\Omega_T}\varphi_{\widetilde U}^\prime[U]\cdot Udxdt}\\
		&&=\int_{\Omega_T}\left(\varphi_{\widetilde U}^\prime[U_k]-\varphi_{\widetilde U}^\prime[U]\right)\cdot U_kdxdt
		+\int_{\Omega_T}\varphi_{\widetilde U}^\prime[U]\cdot \left( U_k -U\right)dxdt :=I_1+I_2. \nonumber
	\end{eqnarray}
	Using Propositions \ref{FDS}, \ref{FDC} and the compact  embedding $\mathsf H^1(0,T;\mathbb V)\hookrightarrow \mathsf L^2(0,T;\mathbb H),$ by extracting a subsequence,  one can obtain that $\|\varphi_{\widetilde U}^\prime[U_k]-\varphi_{\widetilde U}^\prime[U]\|_{\mathsf L^2(0,T;\mathbb H)} \to 0 \ \mbox{as} \ \ k \to \infty.$    Consequently, since $\{U_k\}$ is bounded in $\mathsf L^2(0,T;\mathbb H),$ by applying H\"older's inequality, we get that
	$|I_1|\leq \|\varphi_{\widetilde U}^\prime[U_k]-\varphi_{\widetilde U}^\prime[U]\|_{\mathsf L^2(0,T;\mathbb H)}\|U_k\|_{\mathsf L^2(0,T;\mathbb H)} \to 0  \ \mbox{as}  \  k\to\infty.$
	Taking $ U_k\overset{w}{\rightharpoonup}  U$ in $\mathsf L^2(0,T;\mathbb H)$ as  $k\to\infty$ into account,  and again  by  Proposition \ref{FDC}, $\varphi_{\widetilde U}^\prime[U]$ is bounded in $\mathsf L^2(0,T;\mathbb H),$ we obtain that $I_2\to 0$  as  $k\to\infty.$ Thus, by taking limit $k\to\infty$ in \eqref{poii},  we complete the proof of (ii).
\end{proof}

\section{Global optimality conditions} \label{Se7} From Theorem \ref{SOC0}, we notice that a control $\widetilde U\in \mathcal U_{ad}$  which satisfies the variational inequality together with a second-order sufficient condition \eqref{sdc} defined on a  cone of critical directions is a local minimizer of the functional $\mathfrak J(\cdot).$ However, it is unclear whether such a control gives a global optimum of (MOCP) and is unique. In the following result, we obtain a way around answering these questions by employing the idea developed for a semilinear elliptic control problem in \cite{AA} and also refer to  \cite{Eb} for the diffuse interface model of tumor growth.  The main idea is to show that an admissible control satisfying the variational inequality together with a condition on the adjoint solution leads to a global optimal control of (MOCP).

\begin{theorem} \label{SCOCP} Let $\Omega$ be a periodic domain in $\mathbb R^3.$
	Let $\widetilde U\in\mathcal U_{ad}$ be any  control  with  the adjoint sate $\varphi_{\widetilde U}$ satisfy the variational inequality \eqref{VIC1}. In addition to that, assume the following conditions hold:
	\begin{eqnarray}\frac{\kappa}{2}\geq \left\{\begin{array}{lclcl}
			\label{SOCI}
			C\big(\|\varphi_{\widetilde U}\|_{\mathsf L^\infty(0,T;\mathbb V)}+2\beta C_r [\widetilde C_2]^{r-2}\|\varphi_{\widetilde U}\|_{\mathsf L^\infty(0,T;\mathbb H)}\big)  &\mbox{if} & r> 2\\[1mm]
			C\big(\|\varphi_{\widetilde U}\|_{\mathsf L^\infty(0,T;\mathbb V)}+4\beta \|\varphi_{\widetilde U}\|_{\mathsf L^\infty(0,T;\mathbb H)}\big) &\mbox{if} & r=2\\[1mm]
			C \|\varphi_{\widetilde U}\|_{\mathsf L^\infty(0,T;\mathbb V)} &\mbox{if} & r=1,
		\end{array}\right.
	\end{eqnarray}
	where the parameters $\kappa>0$ and $\beta>0,$ the constants $C,\widetilde C_2>0$ are from \eqref{2.1} and \eqref{ps1} (cf.\eqref{J14}) respectively, and $C_r>0$ depends  on $r.$
	
	Then for $r=1$ and any $r\geq 2,$ $\widetilde U\in\mathcal U_{ad}$  is a global optimal control of (MOCP). Furthermore, if the conditions  \eqref{SOCI} are replaced with strict inequality ($<\frac{\kappa}{2}$), then the global optimal control $\widetilde U$ is unique.
\end{theorem}
\begin{proof}  Let $U\in \mathcal  U_{ad}$ be an arbitrary control. Let $u:=u_{U}$ and $\widetilde u:=u_{\widetilde{U}}$ be the strong solutions of \eqref{1.1} corresponding to  $U$ and $\widetilde U$ respectively. It is easy to check that the following inequality holds:
	\begin{eqnarray} \label{ssoc3}
		\mathfrak J(U)-\mathfrak J(\widetilde U)
		&=& \frac{\kappa}{2}\int_0^T\|\nabla u(t)-\nabla\widetilde u(t)\|_{\mathbb H}^2 dt+\frac{\lambda}{2}\int_0^T\|U(t)-\widetilde U(t)\|_{\mathbb H}^2 dt\nonumber\\
		&&+\kappa\int_{\Omega_T}\nabla (\widetilde u-u_d)\cdot \nabla(u-\widetilde u)dxdt+\lambda\int_{\Omega_T} \widetilde  U\cdot (U-\widetilde U)dxdt.\nonumber\\
		&\geq& \frac{\kappa}{2}\int_0^T\|\nabla u(t)-\nabla\widetilde u(t)\|_{\mathbb H}^2 dt+\frac{\lambda}{2}\int_0^T\|U(t)-\widetilde U(t)\|_{\mathbb H}^2 dt+R, \ \ \  \
	\end{eqnarray}
	where
	\begin{eqnarray}\label{J0}
		R:=\kappa\int_{\Omega_T}\nabla (\widetilde u-u_d)\cdot \nabla(u-\widetilde u)dxdt-\int_{\Omega_T} \widetilde\varphi\cdot (U-\widetilde  U)dxdt,
	\end{eqnarray}
	and  we used the variational inequality \eqref{VIC1}:
	$$\lambda\int_{\Omega_T} \widetilde  U\cdot (U-\widetilde  U)dxdt\geq -\int_{\Omega_T} \widetilde\varphi\cdot (U-\widetilde  U)dxdt,$$  for any   $U\in\mathcal U_{ad},$ $\widetilde \varphi:=\varphi_{\widetilde U}$  is the weak solution of the adjoint system \eqref{as}.
	
	Our main idea here is to show that $\mathfrak J(U)\geq \mathfrak J(\widetilde U)$ for all $U\in\mathcal U_{ad}\backslash\{\widetilde U\}.$ 	To attain this end, let us evaluate the lower bound of the integral $R.$  For any  $U\in\mathcal U_{ad}\backslash\{\widetilde U\},$  let $\widehat u:= u-\widetilde u, \widehat U:=U-\widetilde U$ and  $\widehat p:=p_{U}-p_{\widetilde U}.$
	For any $w\in \mathbb V,$ we note that
	\begin{eqnarray*}\label{snoc1}
		((u\cdot \nabla)u,w)-((\widetilde  u\cdot \nabla)\widetilde  u,w)
		\!\!&=&\!\!\mathrm b(u-\widetilde  u,u-\widetilde  u,w) + \mathrm b(\widetilde  u,u-\widetilde  u,w)+\mathrm b(u-\widetilde  u,\widetilde  u, w) \nonumber\\
		&=&((\widehat u\cdot \nabla) \widehat u,w)+ ((\widetilde  u\cdot \nabla)\widehat u,w)+((\widehat u\cdot \nabla)\widetilde  u,w).
	\end{eqnarray*}
	Therefore, the triplet $(\widehat u,\widehat p,\widehat U) $ satisfies the system
	\begin{eqnarray}\label{ssoc1}\left\{\begin{array}{rlclcrr}
			\widehat u_t-\mu \Delta \widehat u_t -\nu \Delta \widehat u + (\widetilde u\cdot \nabla)\widehat u+(\widehat u\cdot \nabla)\widetilde u +\nabla \widehat p+\alpha \widehat u
			\\ [1mm] +(\widehat u\cdot\nabla)\widehat u+\beta f(u)-\beta f(\widetilde u) &=& \widehat U  \ &\mbox{in} \ &\Omega_T\\[1mm]
			\nabla\cdot \widehat u=0  \ \ \mbox{in} \ \ \Omega_T, \ \ \ \
			\widehat u(x,0)&=&0 \ &\mbox{in}&\Omega.
		\end{array}\right.
	\end{eqnarray}
	Taking inner product of \eqref{ssoc1} with $\widetilde\varphi,$  and integrating by parts, we obtain
	\begin{eqnarray} \label{ssoc11}
		\lefteqn{\int_{\Omega_T}\big(-\widetilde\varphi_t \cdot \widehat u-\mu\nabla\widetilde\varphi_t\cdot \nabla \widehat u+\nu\nabla\widetilde\varphi\cdot \nabla \widehat u\big)dxdt}  \nonumber\\
		&&+\int_{\Omega_T}\Big((\nabla \widetilde  u)^T\widetilde\varphi-(\widetilde u\cdot \nabla)\widetilde\varphi+\alpha \widetilde\varphi+\beta f^\prime(\widetilde  u)\widetilde\varphi\Big)\cdot \widehat u  dxdt \\
		&&+ \int_{\Omega_T}(\widehat u\cdot\nabla)\widehat u\cdot \widetilde \varphi dxdt+\beta \int_{\Omega_T} \big( f(u)- f(\widetilde u)- f^\prime(\widetilde u)\widehat u\big)\cdot\widetilde \varphi dxdt=\int_{\Omega_T}\widehat U\cdot \widetilde \varphi dxdt.  \nonumber 
	\end{eqnarray}
	By testing \eqref{as} with $\widehat u$ and comparing it with the left-hand side integrals of \eqref{ssoc11},  the integrals in $R$ can be expressed as follows
	\begin{eqnarray} \label{Re1}
		R&=&- \int_{\Omega_T}(\widehat u\cdot\nabla)\widehat u\cdot \widetilde \varphi dxdt
		-\beta \int_{\Omega_T} \big( f(u)- f(\widetilde u)- f^\prime(\widetilde u)\widehat u\big)\cdot\widetilde \varphi dxdt\\
		&=& -\int_{\Omega_T}\Big((\widehat u\cdot\nabla)\widehat u+\beta\int_0^1 (1-\theta)f^{\prime\prime}(\widetilde u+ \theta \widehat u)[\widehat u,\widehat u]d\theta\Big)\cdot\widetilde \varphi dxdt, \nonumber
	\end{eqnarray}
	where we also invoked the second-order Taylor's formula \eqref{SOTF} for $r>2.$
	Let us obtain a lower bound of $R.$ By invoking \eqref{2.1}, and the embedding  $\mathbb V\hookrightarrow \mathbb L^4,$  we get
	\begin{eqnarray}\label{J12}
		\left|\int_{\Omega}(\widehat u\cdot\nabla)\widehat u \cdot \widetilde \varphi dx\right| =|\mathrm b(\widehat u,\widehat u,\widetilde \varphi)| \leq \|\widehat u(t)\|_{\mathbb L^4}^2\|\nabla\widetilde \varphi(t)\|_{\mathbb L^2}\leq  C \|\widehat u(t)\|_{\mathbb V}^2\|\widetilde \varphi(t)\|_{\mathbb V}.
	\end{eqnarray}
	For $r>2,$ using the second derivative formula \eqref{U1x} and H\"older's inequality,   we get
	\begin{eqnarray}\label{J13}
		&&\lefteqn{	\beta\left|\int_\Omega\int_0^1 (1-\theta) f^{\prime\prime}(\widetilde u+ \theta \widehat u)[\widehat u,\widehat u]\cdot\widetilde \varphi  d\theta dx\right|\nonumber}\\
		&\leq& \beta C_r\sup_{\theta\in (0,1)}	\int_\Omega\int_0^1|\theta u+ (1-\theta)\widetilde  u|^{r-2}|\widehat u|^2|\widetilde \varphi|d\theta dx\nonumber\\
		&\leq& \beta C_r\left(\| u(t)\|_{\mathbb L^\infty}^{r-2}+ \|\widetilde u(t)\|_{\mathbb L^\infty}^{r-2}\right)\|\widehat u(t)\|_{\mathbb L^4}^2\|\widetilde \varphi(t)\|_{\mathbb L^2} \nonumber\\
		&\leq& C\beta C_r\big(\| u\|^{r-2}_{\mathsf L^\infty(0,T;\mathbb L^\infty)}+ \|\widetilde u\|^{r-2}_{\mathsf L^\infty(0,T;\mathbb L^\infty)}\big)\|\widehat u(t)\|_{\mathbb V}^2\|\widetilde \varphi(t)\|_{\mathbb H},
	\end{eqnarray}
	where $C_r>0$ depends only on $r$ and the constant $C>0$ in \eqref{J12} and \eqref{J13} arises from the inequality $\|\widehat u(t)\|_{\mathbb L^4} \leq \sqrt C \|\widehat u(t)\|_{\mathbb V}.$ Coupling \eqref{J12},\eqref{J13}, using \eqref{ps1} and invoking the condition \eqref{SOCI}, we get
	\begin{eqnarray}\label{J14}
		|R|&\leq& C\big(\|\widetilde \varphi\|_{\mathsf L^\infty(0,T;\mathbb V)}+2\beta C_r [\widetilde C_2]^{r-2}\|\widetilde \varphi\|_{\mathsf L^\infty(0,T;\mathbb H)}\big)\|\widehat u\|_{\mathsf L^2(0,T;\mathbb V)}^2 \nonumber \\
		&\leq& \frac{\kappa}{2}\|\widehat u\|_{\mathsf L^2(0,T;\mathbb V)}^2,
	\end{eqnarray}
	for all $r>2,$ where $\widetilde C_2:=C_*\sqrt {C_2}$ and $C_*$ is due to $\|u(t)\|_{\mathbb L^\infty} \leq C_* \| u(t)\|_{\mathbb H^2}.$  In the case of $r=2,$ we infer from \eqref{U21} that the damping integral in \eqref{Re1} can be estimated using the first-order Taylor's formula as follows
	\begin{eqnarray}\label{J15}
		\beta\left|\int_\Omega\int_0^1\big(f^{\prime}(\widetilde u+\theta \widehat u)\widehat u- f^\prime(\widetilde u)\widehat u\big)d\theta \cdot\widetilde \varphi dx \right|
		\leq 4C\beta \|\widehat u(t)\|_{\mathbb V}^2\|\widetilde \varphi(t)\|_{\mathbb H}.
	\end{eqnarray}
	
	By the inequalities \eqref{J12} and \eqref{J15} together with \eqref{SOCI}, the estimate \eqref{J14}  holds for $r=2.$
	Consequently, for $r=1$ and any $r\geq 2,$ $R\geq -\frac{\kappa}{2}\|\widehat u\|_{\mathsf L^2(0,T;\mathbb V)}^2$ for all $U\in\mathcal U_{ad}\backslash\{\widetilde U\},$ and hence from \eqref{ssoc3}, we arrive at the optimality inequality
	$\mathfrak J(U)\geq \mathfrak J(\widetilde U)$
	for any  $ U\in \mathcal U_{ad}\backslash\{\widetilde U\}.$ Thus, an admissible control $\widetilde U\in\mathcal U_{ad}$  satisfying the variational inequality \eqref{VIC1} is a global optimal of (MOCP).  Further, if the condition \eqref{SOCI} is replaced by a strict inequality, then \eqref{J14} holds with $|R|< \frac{\kappa}{2}\|\widehat u\|_{\mathsf L^2(0,T;\mathbb V)}^2.$ In this case, it is evident that $\mathfrak J(U)>\mathfrak J(\widetilde U)$ for any  $U\in\mathcal U_{ad}\backslash\{\widetilde U\}.$ Hence, the global optimal control $\widetilde U\in\mathcal U_{ad}$ of (MOCP) is unique.  This completes the proof.
\end{proof}	
\begin{remark} \label{SOCBD}
	In section \ref{FOBD}, we discussed the first-order optimality conditions (Theorem \ref{foc2}) in the bounded domain $\Omega$  by restricting the growth of the damping term $\beta |u|^{r-1}u$ to $2\leq r\leq 5.$ By a careful study of the proof of Theorem \ref{SCOCP}, it is evident that to prove this theorem for the bounded domain, we only need to prove the inequality \eqref{J13} with the help of the weak solutions of \eqref{1.1}.
	
	For any $ 2< r\leq 5,$ we infer from \eqref{U22} and \eqref{J13}  that
	\begin{eqnarray}\label{J16}
		&&\lefteqn{\beta\left|\left\langle\int_0^1(1-\theta) f^{\prime\prime}(\widetilde u+ \theta \widehat u)[\widehat u,\widehat u] d\theta, \widetilde \varphi \right\rangle\right|}\nonumber\\
		&\leq&C_5\beta C_r\big(\|\widetilde u\|^{r-2}_{\mathsf L^\infty(0,T;\mathbb L^{r+1})} +\|u\|^{r-2}_{\mathsf L^\infty(0,T;\mathbb L^{r+1})}\big)\|\widehat u(t)\|^{2}_{\mathbb V}  \|\widetilde \varphi(t)\|_{\mathbb V},
	\end{eqnarray}	
	where $ C_5=C_6 C_7>0$ stands for the constant from $\|\widehat u(t)\|_{\mathbb L^{\frac{12(r+1)}{17-r}}} \leq \sqrt {C_6} 	\|\widehat u(t)\|_{\mathbb V}$ and $\|\widetilde \varphi(t)\|_{\mathbb L^6} \leq C_7 \|\widetilde \varphi(t)\|_{\mathbb V}.$
	From \eqref{J12} and \eqref{J16}, we obtain that
	$|R|\leq \frac{\kappa}{2}\|\widehat u\|_{\mathsf L^2(0,T;\mathbb V)}^2$ for all  $2< r\leq 5,$
	provided $$\frac{\kappa}{2} \geq \Big(C+2\beta C_5 C_r[\widehat C_1]^{r-2}\Big)\|\widetilde \varphi\|_{\mathsf L^\infty(0,T;\mathbb V)},$$
	where $\widehat C_1:=\Big[\frac{\widetilde C_1(r+1)}{2\beta}\Big]^{\frac{1}{(r+1)}},$ and $\widetilde C_1$ is the constant from \eqref{ps2e}. When $r=2,$  the estimates \eqref{J12},\eqref{J15} and condition \eqref{SOCI} show that the above bound for $|R|$ holds true. Since the weak solution of \eqref{1.1} is obtained for the bounded domain, the global optimality conditions (Theorem \ref{SCOCP}) are valid for the bounded domain as well for all $2\leq r\leq 5$ and $r=1.$
\end{remark}
\section*{Acknowledgments} The author would like to thank  Dr.Manil T. Mohan, Indian Institute of Technology, Roorkee, for fruitful discussions during the preparation of the final version of the manuscript.


\begin{thebibliography}{99}
	
	\bibitem{Ab} 
	\newblock F. Abergel and R. Temam,
	\newblock {On some control problems in fluid mechanics},
	\newblock \emph{Theor. Comput. Fluid Dyn.,} {\bf 1} (1990), 303--325.
	
	\bibitem{Ag} 
	\newblock S. Agmon,
	\newblock \emph{Lectures on Elliptic Boundary Value Problems,}
	\newblock AMS, Rhode Island, 2010.  	
	
	\bibitem{AA} 
	\newblock A. Ahmad Ali, K. Deckelnick and M. Hinze.
	\newblock {Global minima for semilinear optimal control problems},
	\newblock \emph{Comput. Optim. Appl.}, {\bf 65} (2016), 261--288.
	
	\bibitem{An2} 
	\newblock C. T. Anh and T. N. Nguyet,
	\newblock {Optimal control of the instationary three dimensional Navier-Stokes-Voigt equations},
	\newblock \emph{Numer. Funct. Anal. Optim.,} {\bf 37} (2016), 415--439.
	
	\bibitem{An3} 
	\newblock C. T. Anh and T. N. Nguyet,
	\newblock {Time optimal control of the unsteady 3D Navier-Stokes-Voigt equations},
	\newblock \emph{Appl. Math. Optim.,} {\bf 79} (2019), 397--426.
	
	\bibitem{An} 
	\newblock C. T. Anh and P. T. Trang,
	\newblock {On the 3D Kelvin-Voigt-Brinkman-Forchheimer equations in some unbounded domains},
	\newblock \emph{Nonlinear Anal.,} {\bf 89} (2013), 36--54.
	
	\bibitem{Ant}
	\newblock S. N. Antontseva and H. B. de Oliveira,
	\newblock {The Navier-Stokes problem modified by an absorption term},
	\newblock \emph{Appl. Anal.,} {\bf 89} (2010), 1805--1825.
	
	\bibitem{Ba} 
	\newblock V. Barbu,
	\newblock {The time optimal control of Navier-Stokes equations},
	\newblock \emph{Systems Cont. Lett.,} {\bf 30} (1997), 93--100.
	
	\bibitem{Cai}
	\newblock X. Cai and Q. Jiu,
	\newblock {Weak and strong solutions for the incompressible Navier-Stokes equations with damping},
	\newblock \emph{J. Math. Anal. Appl.,} {\bf 343} (2008), 799--809.
	
	\bibitem{Ca} 
	\newblock Y. Cao, E. M. Lunasin and E. S. Titi,
	\newblock {Global well-posedness of the three-dimensional viscous and inviscid simplified Bardina turbulence models},
	\newblock \emph{Comm. Math. Sci.,} { \bf 4} (2006), 823--848.
	
	\bibitem{Cas} 
	\newblock E. Casas, J. C. de los Reyes and F. Tr\"oltzsch,
	\newblock {Sufficient second-order optimality conditions for semilinear control problems with pointwise state constraints},
	\newblock \emph{SIAM J. Optim.}, {\bf 19} (2008), 616--643.
	
	\bibitem{Do} 
	\newblock C. R. Doering and J. D. Gibbon,
	\newblock \emph{Applied Analysis of the Navier-Stokes Equations,}
	\newblock Cambridge University Press, 1995.
	
	\bibitem{Eb} 
	\newblock M. Ebenbeck and P. Knopf,
	\newblock {Optimal control theory and advanced optimality conditions for a diffuse interface model of tumor growth},
	\newblock \emph{ESAIM Control Optim. Calc. Var.,} {\bf 26}  (2020), Paper No. 71, 38 pp.
	
	\bibitem{Ev} 
	\newblock L. C. Evans,
	\newblock \emph{Partial Differential Equations}
	\newblock AMS, Rhode Island, 1998.
	
	\bibitem{Fa} 
	\newblock H. O. Fattorini and S. S. Sritharan,
	\newblock {Optimal control problems with state constraints in fluid mechanics and combustion},
	\newblock \emph{Appl. Math. Optim.,} {\bf 38} (1998), 159--192.
	
	\bibitem{Fl} 
	\newblock F. Flandoli,
	\newblock {Introduction to 3D stochastic fluid dynamics},
	\newblock {in} \emph{SPDE in Hydrodynamic: Recent Progress and Prospects, Lecture Notes in Mathematics}, Springer, \textbf{1942} (2008), 51--150.
	
	\bibitem{Foi} 
	\newblock C. Foias, O. Manley, R. Rosa and R. Temam,
	\newblock \emph{Navier-Stokes Equations and Turbulence},
	\newblock Cambridge University Press, 2001.
	
	\bibitem{Fu} 
	\newblock A. V. Fursikov,
	\newblock \emph{Optimal Control of Distributed Systems: Theory and applications,}
	\newblock AMS, Rhode Island, 2000.
	
	\bibitem{Fu1} 
	\newblock A. V. Fursikov, M. D. Gunzburger and L. S. Hou,
	\newblock {Optimal boundary control for the evolutionary Navier-Stokes system: The three-dimensional case},
	\newblock \emph{SIAM J. Control Optim.,} {\bf 43} (2005), 2191--2232.
	
	\bibitem{Ha} 
	\newblock K. W. Hajduk and J. C. Robinson,
	\newblock {Energy equality for the 3D critical convective Brinkman-Forchheimer equations},
	\newblock \emph{J. Differential Equations,} {\bf 263} (2017), 7141--7161.
	
	\bibitem{Ha1} 
	\newblock K. W. Hajduk, J. C. Robinson and W. Sadowski,
	\newblock {Robustness of regularity for the 3D convective Brinkman-Forchheimer equations},
	\newblock \emph{J. Math. Anal. Appl.,} {\bf 500} (2021), 125058, 23 pp.
	
	\bibitem{Ka} 
	\newblock V. Kalantarov and S. Zelik,
	\newblock {Smooth attractors for the Brinkman-Forchheimer equations with fast growing nonlinearities},
	\newblock \emph{Commun. Pure Appl. Anal.,} {\bf 11} (2012), 2037--2054.
	
	\bibitem{Ki} 
	\newblock B. T. Kien, A. R\"osch and D. Wachsmuth,
	\newblock {Pontryagin's principle for optimal control problem governed by 3D Navier-Stokes equations},
	\newblock \emph{J. Optim. Theory Appl.,} {\bf 173} (2017), 30--55.
	
	\bibitem{Ko} 
	\newblock H. Kozono and T. Yanagisawa,
	\newblock {$L^r$-variational inequality for vector fields and the Helmholtz-Weyl decomposition in bounded domains},
	\newblock \emph{Indiana Univ. Math. J.,} {\bf 58} (2009), 1853--1920.
	
	\bibitem{La} 
	\newblock O. A. Ladyzhenskaya,
	\newblock \emph{The Mathematical Theory of Viscous Incompressible Flow,}
	\newblock Gordon and Breach, Science Publishers, New York, 1969.
	
	\bibitem{Li} 
	\newblock J.-L. Lions,
	\newblock \emph{Optimal Control of Systems Governed by Partial Differential Equations,}
	\newblock Springer, 1971.
	
	\bibitem{Liu} 
	\newblock H. Liu,
	\newblock {Optimal control problems with state constraint governed by Navier-Stokes equations},
	\newblock \emph{Nonlinear Anal.,} {\bf 73} (2010), 3924--3939.
	
	\bibitem{Mo} 
	\newblock M. T. Mohan,
	\newblock {Global and exponential attractors for the 3D Kelvin-Voigt-Brinkman-Forchheimer equations},
	\newblock \emph{Discrete Contin. Dyn. Syst. Ser. B}, {\bf 25} (2020), 3393--3436.
	
	\bibitem{Mo1} 
	\newblock M. T. Mohan,
	\newblock {Optimal control problems governed by two dimensional convective Brinkman-Forchheimer equations},
	\newblock \emph{Evol. Equ. Control Theory}, {\bf 11} (2022), 649--679.
	
	\bibitem{Ni} 
	\newblock L. Nirenberg,
	\newblock On elliptic partial differential equations,
	\newblock \emph{Ann. Sc. Norm. Super. Pisa Cl. Sci.}, {\bf 13} (1959), 115--162.
	
	\bibitem{Os} 
	\newblock A. P. Oskolkov,
	\newblock The uniqueness and solvability in the large of boundary value problems for the equations of motion of aqueous solutions of polymers,
	\newblock \emph{Zap. Naucn. Sem. Leningrad. Otdel. Mat. Inst. Steklov,} {\bf 38} (1973), 98--136.
	
	\bibitem{Sa} 
	\newblock K. Sakthivel and S. S. Sritharan,
	\newblock {Martingale solutions for stochastic Navier-Stokes equations driven by L\'evy noise},
	\newblock \emph{Evol. Equ. Control Theory}, \textbf{1} (2012), 355--392.
	
	\bibitem{Si} 
	\newblock J. Simon,
	\newblock {Compact sets in the space $L^p(0, T ;B)$},
	\newblock \emph{Ann. Mat. Pura Appl.,} {\bf 146} (1987), 65--96.
	
	\bibitem{Sr1} 
	\newblock S. S. Sritharan,
	\newblock {Dynamic programming of Navier-Stokes equations},
	\newblock \emph{Systems Cont. Lett.,} {\bf 16} (1991), 299--307.
	
	\bibitem{Sr11} 
	\newblock S. S. Sritharan,
	\newblock {An optimal control problem in exterior hydrodynamics},
	\newblock \emph{Proc. Roy. Soc. Edinburgh Sect. A}, {\bf 121} (1992), 5--32.
	
	\bibitem{Sr2} 
	\newblock S. S. Sritharan,
	\newblock \emph{Optimal Control of Viscous Flow,}
	\newblock SIAM, Philadelphia, 1998.
	
	\bibitem{Sr3} 
	\newblock S. S. Sritharan,
	\newblock {Deterministic and stochastic control of Navier-Stokes equations with linear, monotone and hyper viscosities},
	\newblock \emph{Appl. Math. Optim.,} {\bf 41} (2000), 255--308.
	
	\bibitem{Te} 
	\newblock R. Temam,
	\newblock \emph{Navier-Stokes Equations: Theory and Numerical Analysis},
	\newblock AMS, 1984.
	
	\bibitem{Te0} 
	\newblock R. Temam,
	\newblock \emph{Navier-Stokes Equations and Nonlinear Functional Analysis,}
	\newblock Second Edition, CBMS-NSF Regional Conference Series in Applied Mathematics, 1995.
	
	\bibitem{Te1} 
	\newblock R. Temam,
	\newblock Some Developments on Navier-Stokes Equations in the second half of the 20th Century,
	\newblock In: \emph{Development of Mathematics 1950-2000}, J.-P. Pier, ed., Birkh\"auser, Basel, 2000, 1049--1106.
	
	\bibitem{Tr} 
	\newblock F. Tr\"oltzsch,
	\newblock \emph{Optimal Control of Partial Differential Equations,}
	\newblock AMS, Rhode Island, 2010.
	
	\bibitem{Tr1} 
	\newblock F. Tr\"oltzsch and D. Wachsmuth,
	\newblock {Second-order sufficient optimality conditions for the optimal control of Navier-Stokes equations},
	\newblock \emph{ESAIM Control Optim. Calc. Var.,} {\bf 12} (2006), 93--119.
	
	\bibitem{Wa} 
	\newblock G. Wang,
	\newblock {Optimal controls of 3-dimensional Navier-Stokes equations with state constraints},
	\newblock \emph{SIAM J. Control Optim.,} {\bf 41} (2002), 583--606.
	
	\bibitem{Lij}
	\newblock L. Wang and P. He,
	\newblock {Second order optimality conditions for optimal control problems governed by 3-dimensional Navier-Stokes equations},
	\newblock \emph{Acta Math. Sci. Ser. A,} {\bf 26} (2006), 729--734.
	
	\bibitem{We} 
	\newblock T. Weier, G. Gerbeth, G. Mutschke, O. Lielausis and G. Lammers,
	\newblock {Control of flow separation using electromagnetic forces},
	\newblock \emph{Flow Turb. Comb.,} {\bf 71} (2003), 5--17.
	
	
\end{thebibliography}
\end{document}